\documentclass[a4paper, 12pt]{amsart} 

\usepackage[pdftex]{graphicx} 

\usepackage{amssymb}
\usepackage{amsmath}
\usepackage{amsthm}
\usepackage{amscd}

\usepackage{comment}
\usepackage[all]{xy}
\usepackage{color}

\usepackage{tikz}
\usetikzlibrary{positioning}
\usetikzlibrary{intersections}
\usetikzlibrary{calc, quotes, angles}

\date{\today}
\theoremstyle{plain}
\newtheorem{thm}{Theorem}[section]

\newtheorem{lem}[thm]{Lemma}

\newtheorem{slem}[thm]{Sublemma}
\newtheorem{cor}[thm]{Corollary}
\newtheorem{prop}[thm]{Proposition}
\newtheorem{defn}[thm]{Definition}
\newtheorem{rem}[thm]{Remark}

\newtheorem{claim}[thm]{Claim}
\newtheorem{ex}[thm]{Example}

\makeatletter

\newcommand{\diam}[0]{\mathrm{diam}\,}

\newcommand{\e}[0]{\epsilon}

\newcommand{\R}[0]{\mathbb R}
\newcommand{\pa}[0]{\partial}
\newcommand{\CAT}{\mathrm{CAT}}

\newcommand{\sing}[0]{\rm sing}

\newcommand{\ca}[0]{\mathcal}

\newcommand{\pmed}[0]{\par\medskip}
\newcommand{\psmall}[0]{\par\smallskip}

\newcommand{\n}[0]{\noindent}

\newcommand{\beq}[0]{\begin{equation}}
\newcommand{\eeq}[0]{\end{equation}}
\newcommand{\beqq}[0]{\begin{equation*}}
\newcommand{\eeqq}[0]{\end{equation*}}

\newcommand{\bali}[0]{\begin{align}}
\newcommand{\eali}[0]{\end{align}}

\newcommand{\benum}[0]{\begin{enumerate}}
\newcommand{\eenum}[0]{\end{enumerate}}

\begin{document}

\title[Two-dimensional spaces 
     with curvature bounded above II]
{Two-dimensional metric spaces 
     with curvature bounded above II }

\author[K.Nagano]{Koichi Nagano}
\author[T.Shioya]{Takashi Shioya} 
\author[T.Yamaguchi]{Takao Yamaguchi}


\address{Koichi Nagano, Department of Mathematics, University of Tsukuba,  Tsukuba 305-8571, Japan}
\email{nagano@math.tsukuba.ac.jp}

\address{Takashi Shioya, Mathematical Institute, Tohoku  University, Sendai  980-8578, Japan}
\email{shioya@math.tohoku.ac.jp}

\address{Takao Yamaguchi, Department of Mathematics, Kyoto University, Kyoto  606-8502, Japan}
 \curraddr{Department of Mathematics, University of Tsukuba,  Tsukuba 305-8571, Japan}
 \email{takao@math.tsukuba.ac.jp}

\subjclass[2010]{Primary 53C20, 53C23}
\keywords{Upper curvature bound; homotopy convergence; Gauss-Bonnet theorem; convergence of curvature measures}
\thanks{This work was supported by JSPS KAKENHI Grant Numbers 20K03603, 19K03459, 18H01118,  21H00977.
  }


\begin{abstract} As a continuation of \cite{NSY:local}, 
we mainly discuss the global structure of 
two-dimensional locally compact geodesically complete metric spaces with curvature bounded above.
We first obtain the result on 
the Lipschitz homotopy approximations of such spaces by polyhedral spaces. 
We define the curvature measures on 
our spaces making use of the convergence of the curvature measures, and establish Gauss-Bonnet Theorem.
We also give a characterization of such spaces.
\end{abstract}

\maketitle

\tableofcontents

\section{Introduction} \label{sec:intro}

In our previous paper \cite{NSY:local}, we 
have completely described the local structure of 
two-dimensional locally compact, geodesically complete
metric spaces with curvature bounded above.
In the present paper, as a continuation of \cite{NSY:local}, 
we mainly discuss the global structure of such spaces.
We also use the equivalent terminology,
locally $\CAT(\kappa)$-spaces, for spaces with curvature 
$\le\kappa$.

Let $X$ be a  two-dimensional locally compact, geodesically complete
metric space with curvature $\le \kappa$.
We denote by  $\Sigma_p(X)$ the space of directions
at $p\in X$.
We assume the most essential case when 
\begin{align} \label{eq:Sigma=CAT(1)}
   \text{ $\Sigma_p(X)$  is a connected graph 
without endpoints}
\end{align}
for every point $p\in X$.
The main results in  \cite{NSY:local} 
are roughly stated as follows:
\begin{itemize}
\item For each $p\in X$ and for small enough $0<r \le r_p$,  the closed $r$-ball $B(p,r)$ around $p$ is the union of finitely many properly embedded $\CAT(\kappa)$-disks or properly immersed
branched $\CAT(\kappa)$-disks\,;
\item The topological singular set $\ca S(X)$ of $X$ is a locally finite union of Lipschitz curves 
and has a graph structure in a generalized sense.
\end{itemize}

Burago and Buyalo \cite{BurBuy:upperII} 
gave a complete characterization of two-dimensional locally compact
polyhedra of curvature bounded above. The present work extends their results to the general case.  
Our results in the present paper are stated as follows.

First, we obtain  the following,
which solves a strengthened version of  the problem raised in Burago and Buyalo \cite{BurBuy:upperII}.
We use the terminologies defined in \cite{FMY}.
For $C,\e>0$,
a Lipschitz map $f:X\to Y$ between metric spaces
is called a {\it $(C,\e)$-Lipschitz homotopy equivalence} if there are $C$-Lipschitz maps
$g:Y\to X$,   $F:X\times [0,\e]\to X$
and $G:Y\times [0,\e]\to Y$ such 
that $F(\,\cdot\,,0)=g\circ f$, $F(\,\cdot\,,\e)=1_X$,
$G(\,\cdot\,,0)=f\circ g$ and $G(\,\cdot\,,\e)=1_Y$.
For sequences $C_n$, $\e_n$,
We say that a sequence of pointed metric spaces $(X_n,p_n)$ converges to a pointed metric space $(X, p)$ 
with respect to the {\it pointed 
$(C_n,\e_n)$-Lipschitz  
homotopy convergence}
if there are $o_n$-approximations 
$\varphi_n:(X_n,p_n)\to (X,p)$ that are 
 $(C_n,\e_n)$-Lipschitz  
homotopy equivalences with 
$\lim_{n\to\infty} o_n=0$
\pmed

\begin{thm} \label{thm:approx}
Let $(X,p)$ be a pointed two-dimensional locally compact, geodesically complete locally $\CAT(\kappa)$-space.  Then there is a sequence of pointed two-dimensional locally compact geodesically complete, polyhedral locally $\CAT(\kappa)$-spaces $(X_n,p_n)$ 
which converges to $(X, p)$ with respect to the pointed $(C_n,\e_n)$-Lipschitz  homotopy convergence with 
$\lim_{n\to\infty} C_n=1$, $\lim_{n\to\infty} \e_n=0$. 
\end{thm}
 
 Applying the approximation result of \cite[Theorem 0.11]{BurBuy:upperII} 
 and our method in the proof of Theorem
 \ref{thm:approx},
 to  each $(X_n,x_n)$, we can actually replace 
 $(X_n,p_n)$ in Theorem \ref{thm:approx} by 
 pointed two-dimensional locally compact geodesically complete polyhedral locally  
 $\CAT(\kappa)$-spaces with {\it piecewise smooth metrics}.

\pmed

Next  we discuss the Gauss-Bonnet theorem, 
which was established for surfaces with 
bounded curvature in \cite{AZ:bddcurv},
and for polyhedral locally $\CAT(\kappa)$-spaces
in \cite{BurBuy:upperII} (cf. \cite{ArsBuy}).
For our space $X$, we define the curvature measure $\omega^X$ of $X$ 
as the limit of those of $X_n$.

For a bounded domain $D$ of $X$, we assume that 
$\pa D$ 
 meets $\ca S(X)$ transversally. We call such $D$ {\it  admissible}  (see Definition \ref{defn:admissible-D}). 
 Then the turn $\tau_D(\pa D)$ of $\pa D$ from the side $D$ can be defined 
 (see Subsection \ref{ssec:conv}), and we have the following.

\begin{thm} \label{thm:GBf}
For each admissible domain $D$ of $X$,  we have  
\[
      \omega^X(D)= 2\pi\chi(D) -\tau_D(\pa D),
\]
where $\chi(D)$ denotes the Euler characteristic 
of $D$.
\end{thm}
\psmall

In the course of the proof of Theorem \ref{thm:GBf}, we also show the following approximation result
for admissibe domains.
\psmall
\begin{thm} \label{thm:conv-omega}
Let $(X,p)$ and $(X_n,p_n)$ be as in Theorem
 \ref{thm:approx}.
For each admissible domain
$D\subset X$, there is a sequence of admissible domains
$D_n\subset X_n$ satisfying  
\begin{enumerate}
\item  $D_n$ converges to $D$ under the convergence
$(X_n,  p_n)\to (X,p)\,;$
\item $D_n$ has the same homotopy type as $D\,;$
\item $\lim_{n\to\infty}\tau_{D_n}(\pa D_n)=\tau_D(\pa D)\,;$
\item $\lim_{n\to\infty} \omega_{BB}^{X_n}(D_n)=\omega^X(D)$,
\end{enumerate}
where  $\omega_{BB}^{X_n}$ denotes the curvature measure 
of  $X_n$ defined in \cite{BurBuy:upperII}.
\end{thm}
\psmall

%
%
%
Since  $\omega^X$ is defined  
as the limit of  the curvature measure
of the polyhedral approximation $X_n$, it is interesting to give an explicit formula for $\omega^X$ on singular 
curves in terms of the singulalities.
We carry out this in Theorem \ref{thm:explicit}.

In the course of the proof of Theorem \ref{thm:explicit},
we establish the following result on the local boundedness of the turn of the 
topological singular locus $\ca S(X)$ of $X$
(see Definition \ref{defn:SXfinite-turn} for the precise definition).

\begin{thm} \label{thm:bdd-turn(Si)}
 $\ca S(X)$ has locally finite turn
variation.
\end{thm}

When $X$ is a polyhedral space in addition,
all of whose faces are bounded, any edge $e$
 has finite turn variation from each face adjacent to $e$ (\cite{BurBuy:upperII}). Theorem \ref{thm:bdd-turn(Si)} extends this result to the general case.

Finally, we discuss conditions for two-dimensional locally compact geodesic spaces to have curvature bounded above.

\begin{thm} \label{thm:converse}
For a  locally compact geodesic space $X$ and a closed subset
$\ca S\subset X$, we assume the following $:$
\begin{enumerate}
\item $X\setminus \ca S$ is a locally
 geodesically complete, locally 
 $\CAT(\kappa)$-surface$\,;$
\item For each $p\in \ca S$, the space $\Sigma_p(X)$ of directions at $p$ equipped with the upper angle
 is a $\CAT(1)$-graph without endpoints $\,;$
\item For each $p\in \ca S$ there is $r>0$ such that
for every $v\in V(\Sigma_p(X))$ and for a  
closed arc $I$ meeting $V(\Sigma_p(X))$ only at $v$ in its interior, 
 there is a $\CAT(\kappa)$-sector $S$ homeomorphic to a disk with vertex $p$ bounded 
 by $S(p,r)$ and $X$-geodesics from $p$ directing to $\pa I$ together with an isometric embedding 
 $\iota:\Sigma_p(S)\to\Sigma_p(X)$
 such that $\iota(\Sigma_p(S))= I\,;$
\item $S(p,r)$ is a graph having the same homotopy type 
as $\Sigma_p(X)$ for any small $r>0$.
\end{enumerate}
Then $X$ is a geodesically complete locally 
$\CAT(\kappa)$-space.
\end{thm}
\pmed

We call $S$ as in Theorem \ref{thm:converse}(3)
a {\it sector surface} at $p$.

\begin{rem} \label{rem:converse} \upshape
\par\n (1)\,
From the conditions in Theorem \ref{thm:converse},  $\ca S$ coincides with the set of topological 
singular points of $X$, and 
has the structure of metric graph in the sense of \cite[Definition 6.8]{NSY:local}. 
\par\n (2)\,
We verify that in Theorem \ref{thm:converse}(3),
the image $\iota(\Sigma_p(S))$ coincides with 
the extrinsic space of directions, 
$\Sigma_p(S) \subset \Sigma_p(X)$, with angle metric $\angle^X$.
\par\n
(3)\, In \cite[Corollary 1.7]{NSY:local},
we showed that $S(p,r)$ is a graph {\it without endpoints} having the same homotopy type 
as $\Sigma_p(X)$. However, in the condition
(4) of Theorem \ref{thm:converse}, we do not require that  $S(p,r)$ does not have endpoints.
\end{rem}

\begin{rem} \label{rem:1.4optimal} \upshape
The results in \cite{NSY:local} show that 
all the conditions $(1) - (4)$ are satisfied 
for any two-dimensional locally compact, geodesically complete locally 
$\CAT(\kappa)$-space $X$ 
and the topological singular point set $\ca S$ in the
essential case when $X$ satisfies 
\eqref{eq:Sigma=CAT(1)}. 

Conversely, the conditions $(1) - (4)$ are optimal
in the sense that there are counter examples if one drops any condition (see Examples \ref{ex:gluing-disk}  $-$ \ref{ex:counter-inv}).
\end{rem}


\pmed
The organization of the paper is as follows.
In Section \ref{sec:prelim}, we review  basic results on two-dimensional locally $\CAT(\kappa)$-spaces, 
some fundamental results on polyhedral spaces with curvature bounded above (\cite{BurBuy:upperII})
and the local structure results (\cite{NSY:local}).

In Section \ref{sec:Polyhedral}, we prove Theorem \ref{thm:approx}. We construct the polyhedral spaces 
by gluing locally defined surgeries. After showing  the Gromov-Hausdorff convergence,  we 
establish the Lipschitz homotopy convergence.  
Here we make gluing of locally defined almost isometries
and make them partly degenerate inductively.

In Section \ref{sec:GB},  we establish a key result
(Lemma \ref{lem:constDn})
on the approximation of an admissible domain in 
the limit space under the homotopy convergence as in 
Theorem \ref{thm:approx}.
Then  Theorems  \ref{thm:GBf} and \ref{thm:conv-omega}
follow almost
immediately.
After establishing the local boundedness of the turn of the topological singular locus (Theorem \ref{thm:bdd-turn(Si)}),
we provide an explicit formula for $\omega^X$ on singular curves in terms of the singulalities.

In Section \ref{sec:converse},
we prove Characterization Theorem  \ref{thm:converse}. We begin with some examples
showing the optimum of the conditions in Theorem  \ref{thm:converse}. Then we derive some basic 
properties of $X$ from those conditions.
While providing fundamental properties of sector surfaces,  we show some significant geometric  properties of $X$, from which we get the conclusion through
the surgery construction together with angle comparison estimates.

Some  results related those in the present paper were announced in \cite{Ya:upper}.

\setcounter{equation}{0}

\section{Preliminaries}
\label{sec:prelim}
\pmed
In the present paper, we follow the notations used in 
\cite{NSY:local}. For instance, 
$\tau_p(\epsilon_1,\ldots, \epsilon_k)$ denotes a function 
depending on $p$ and $\epsilon_1,\ldots, \epsilon_k$ satisfying 
$\lim_{\epsilon_1,\ldots, \epsilon_k\to 0}\tau_p(\epsilon_1,\ldots, \epsilon_k)=0$.

\pmed   

\pmed\n
\subsection{Basic preliminary results} \label{ssec:basic-prelim}

For basic properties of $\CAT(\kappa)$-spaces, we refer to \cite{AKP},
\cite {bridson-haefliger}, \cite{Kl:local}, 
\cite{NL:geodesically}.

Let $X$ be a locally compact, geodesically complete locally $\CAT(\kappa)$-space for a constant $\kappa$.

\begin{lem}$($\cite[Lemma 2.7]{NSY:local}$)$ \label{lem:comparison}
For every $p\in X$, we have the following for all distinct  $x,y\in B(p,r)\setminus \{ p\}$
\begin{align*}\label{eq:diff-comarison-angle}
   \tilde\angle xpy -\angle xpy <\tau_p(r), \quad
   \tilde\angle pxy -\angle pxy <\tau_p(r),\quad
   \tilde\angle pyx -\angle pyx <\tau_p(r).
\end{align*}
\end{lem}

\begin{prop} $($\cite{NL:geodesically}$)$ \label{prop:metsph-tree}
For every $p\in X$,
$S(p, r)$  has the  same homotopy type as 
$\Sigma_p(X)$ for small enough $r>0$.
\end{prop}

For more precise information in the two-dimensional 
case, see \cite[Corollary 1.7] {NSY:local}.

In what follows, we assume that  $X$ is two-dimensional
in addition.

\begin{lem}$($\cite[Lemma 2.2]{NSY:local}$)$  \label{lem:top-vert1}
Let $p\in\mathcal S(X)$.
Then $\Sigma_p(\mathcal S(X))$ coincides with the set  $V(\Sigma_p)$ of all 
vertices of the graph $\Sigma_p(X)$.
\end{lem}

We denote by $\pm\nabla d_p(x)$ the union of 
$-\nabla d_p(x)$ and $\nabla d_p(x)$.

\begin{lem}$($\cite[Prop 6.6, Cor 13.3]{NL:geodesically}$])$  \label{lem:near-vert}
Let $p\in\mathcal S(X)$.
For any $x\in \mathcal S(X)\cap (B(p,r)\setminus\{ p\})$, $V(\Sigma_x(X))$ is contained in a $\tau_p(r)$-neighborhood of $\pm\nabla d_p(x)$.

Therefore there is a positive integer $m\ge 3$ such that 
the Gromov-Hausdorff distance between $\Sigma_x(X)$ 
and the spherical suspension over $m$ points is less than
$\tau_p(r)$.
\end{lem}

As an immediate consequence of Lemmas \ref{lem:top-vert1} and \ref{lem:near-vert},
we have 

\begin{cor} $($\cite[Cor.2.12]{NSY:local}$)$\label{cor:vert}
For each $p\in\mathcal S(X)$ and
for every $x\in\mathcal S(X)\cap (B(p,r)\setminus\{ p\})$,
$\Sigma_x(\mathcal S(X))$ is contained in a 
$\tau_p(r)$-neighborhood of $\pm \nabla d_p(x)$.
\end{cor}

\psmall
Finally in this subsection, we shortly discuss the cardinality of singular points
in a two-dimensional manifold $X$ with curvature $\le\kappa$.
Let $\e>0$. We say that $x\in X$ is an {\it $\e$-singular} point if 
$L(\Sigma_x(X))\ge 2\pi+\e$.   
We also say that $x$ is a {\it singular} point if it is $\e$-singular for some $\e>0$. 

\begin{lem} $($cf. \cite{AZ:bddcurv}, \cite[Prop.4.5]{BurBuy:upperII}$)$  \label{lem:sing-mfd}
For a bounded domain $D$ of a two-dimensional manifold $X$ 
with curvature $\le\kappa$, the set of all singular points contained in $D$
is at most countable.
\end{lem}

\psmall

\subsection{Polyhedral spaces and turns} \label{ssub:polyhed-prelim}
In this subsection, we recall the results of 
Burago and Buyalo \cite{BurBuy:upperII} on polyhedral spaces with curvature bounded above.

First, we need to recall the notion of turn, which was 
first defined in the context of  surfaces with 
bounded curvature in \cite{AZ:bddcurv}
 (see also \cite{Rsh:2mfd}).

For a moment, let $X$ be a surface with bounded curvature.
Let $F$ be a domain in $X$ with boundary $C$.
For an open arc $e$ of $C$, we assume that
$e$ has definite directions at the endpoints $a, b$ and  
the spaces of directions of $F$ at $a$ and $b$ have  positive lengths.
Then  the {\it turn}  $\tau_F(e)$ of $e$ (\cite{AZ:bddcurv}) from the side of $F$ is 
defined as follows:
Let $\gamma_n$ be a simple broken geodesic in 
$F\setminus e$ 
joining $a$ and $b$ and converging to $e$
as $n\to\infty$. Let $\Gamma_n$ be the domain 
bounded by $e$ and $\gamma_n$.
We denote by $\alpha_n$ and $\beta_n$ the sector angle
of $\Gamma_n$ at $a$ and $b$ respectively.
Let $\theta_{n,i}$, $(1\le i\le N_n)$, denote
the sector angle at the break points of $\gamma_n$, viewed from $F\setminus\Gamma_n$. 
Then the turn $\tau_F(e)$ is defined as 
\begin{align*}   
     \tau_F(e) := \lim_{n\to\infty}\sum_{i=1}^{N_n} (\pi -\theta_{n,i}) +\alpha_n+\beta_n, 
\end{align*}
where the existence of the above limit is shown in \cite[Theorem VI.2]{AZ:bddcurv}.

\begin{center}
\begin{tikzpicture}
[scale = 1]
\draw [-, thick] (4.5,0) to [out=150, in=30] (-4.5,0);
\draw [-, thick] (-4.5,0) -- (-3,-0.5) -- (-1.5,0) -- (0,-0.5) -- (1.5,0) -- (3,-0.5) -- (4.5,0);
\filldraw [fill=gray, opacity=.1] 
(4.5,0) to [out=150, in=30] (-4.5,0)
(-4.5,0) -- (-3,-0.5) -- (-1.5,0) -- (0,-0.5) -- (1.5,0) -- (3,-0.5) -- (4.5,0);
\draw [-, thick] (4.5,0) to [out=330, in=150] (4.8,-0.2);
\draw [-, thick] (-4.5,0) to [out=210, in=30] (-4.8,-0.2);
\draw (2.1,-0.5) node[circle] {$\gamma_n$};
\draw (0.6,0.7) node[circle] {$\Gamma_n$};
\draw (0,1.6) node[circle] {$e$};
\fill (4.5,0) coordinate (A) circle (1pt) node [above right] {$b$};
\fill (-4.5,0) coordinate (A) circle (1pt) node [above left] {$a$};
\fill (4.2,0.05) coordinate (A) circle (0pt) node [left] {$\beta_n$};
\fill (-4.2,0) coordinate (A) circle (0pt) node [right] {$\alpha_n$};
\fill (-1.4,0) coordinate (A) circle (0pt) node [below] {$\theta_{n,i}$};
\end{tikzpicture}
\end{center}

For an interior point $c$ of $e$ having definite two directions  $\Sigma_c(e)$, the turn of $e$ 
at $c$ from the side $F$ is defined as 
\[
      \tau_F(c):= \pi-L(\Sigma_c(F)).
\]
 
 We now assume the following additional conditions
 for all $c\in e$:
 \begin{enumerate}
  \item $L(\Sigma_c(F))>0\,;$
  \item $e$ has definite two directions
   $\Sigma_c(e)$.
 \end{enumerate}
Consider the constant $\mu_F(e)\in [0,\infty]$ defined
by
  \begin{align} \label{eq:bdd-turn}
  \mu_F(e):=\sup_{\{ a_i\}} \sum_{i=1}^{n-1}|\tau_{F}((a_i,a_{i+1}))|
   +\sum_{i=2}^{n-1}|\tau_{F}(a_i)|,
 \end{align}
where $\{ a_i\}=\{ a_i\}_{i=1,\ldots,n}$  runs over all the consecutive points on $e$.
The constant $\mu_F(e)$ is called the {\it turn variation} of $e$ from the side $F$, and $e$ has {\it finite turn variation} when $\mu_F(e)<\infty$.
 
 
For a simple closed curve $\lambda$ in $X$ bounding a domain $D$ of $X$, consider a sequence of 
simple closed broken geodesics $\lambda_n$ in $D$
converging to $\lambda$ with sector angles $\theta_{n,i}$
\, $(1\le i\le N_n)$
 at the break points of $\lambda_n$, viewed from the subdomain of $D$ bounded by 
 $\lambda_n$. 
Then the turn $\tau_D(\lambda)$ is defined as 
\begin{align*}
     \tau_D(\lambda) := \lim_{n\to\infty} \sum_{i=1}^{N_n} (\pi -\theta_{n,i}), 
\end{align*}
where the existence of the above limit follows by an  
argument similar to that of $\tau_F(e)$
(see \cite[Theorem VI.6]{AZ:bddcurv}).

 For general treatments of curves with finite turn variation in $\CAT(\kappa)$-spaces, see \cite{ABG}
and the references therein.

 Let $e$ be a simple arc on $X$. 
 One can define the notion of sides $F_+$ and $F_-$ of $e$.
Under the corresponding assumptions, 
 we define the turns $\tau_{F_+}(e)$, $\tau_{F_-}(e)$
 of $e$ from $F_+$ and $F_-$ respectively, as above.
 
 \begin{prop} \label{prop:turn(geodesic)}
 If $e$ is a geodesic, we have 
 $\tau_{F_{\pm}}(e)\le 0$.
 \end{prop}

Similarly, we define the turn variations $\mu_{F_+}(e)$, $\mu_{F_-}(e)$ from $F_+$ and $F_-$.    
We call $e$ to have finite turn variation
if $\mu_{F_+}(e)<\infty$ and $\mu_{F_-}(e)<\infty$
(actually both are finite if one is so
(\cite[Lemma IX.1]{AZ:bddcurv})).
 When $e$ has finite turn variation, $\tau_{F_+}$ and $\tau_{F_-}$ provide signed Borel measures on $e$
 (\cite[Theorem IX.1]{AZ:bddcurv}).

Let $\ca F_\kappa$ be the family of two-dimensional  polyhedral 
locally $\CAT(\kappa)$-spaces $F$ possibly with boundary
$\pa F$ such that any edge of $\pa F$ has finite turn variation. For a collection $\{ F_i\}$ of 
$\ca F_\kappa$, let $X$ be the polyhedron resulting from certain gluing of $\{ F_i\}$ along their edges. We always consider the 
intrinsic metric of $X$ induced from those of $F_i$.
We consider the following two conditions:
\pmed\n
(A)\, For any Borel subset $B$ of an arbitrary edge $e$ of
$X$, and arbitrary faces $F_i$, $F_j$ adjacent to $e$, we have
\[
    \tau_{F_i}(B)+\tau_{F_j}(B)\le 0\,;
\]
(B)\,For any vertex $x$ of $X$, $\Sigma_x(X)$ is $\CAT(1)$.

\begin{thm} $($\cite[Theorem 0.5]{BurBuy:upperII}\label{thm:BB-gluing}$)$
A polyhedron $X$ resulting from certain gluing of 
$\{ F_i\}\subset\ca F_\kappa$ along their edges belongs to
$\ca F_\kappa$ if and only if the conditions $(A), (B)$
are satisfied.
\end{thm}

\begin{thm} $($\cite[Theorem 0.6]{BurBuy:upperII}\label{thm:BB-character}$)$
For $X\in \ca F_\kappa$,  let 
 $\{ F_i\}$ be the collection consisting of
 the metric completions of the components of 
 $X\setminus\ca S(X)$. 
Then   $\{ F_i\}\subset \ca F_\kappa$ 
and $X$ is the gluing of $\{ F_i\}$ along their edges in such a way that the conditions $(A), (B)$
are satisfied.

In particular, each edge of $\ca S(X)$ has finite turn
variation.
\end{thm}

\pmed\n
\subsection{Local structure results}\label{ssec:lcal-str}

In this subsection, we recall several local structure results 
obtained in \cite{NSY:local}.

Let $p\in \ca S(X)$, and $r=r_p>0$ be as in Theorem 1.1 in  \cite{NSY:local}.
From now, we work on $B(p,r)$. Fix any $v\in V(\Sigma_p(X))$, and let $N=N_v$ be the branching number of $\Sigma_p(X)$ at $v$. For a small enough $\delta>0$, let  
$\gamma_1, \ldots, \gamma_N$ be the geodesics
starting from $p$ with $\angle(\dot\gamma_i(0), v) = \delta$ and
$\angle(\dot\gamma_i(0), \dot\gamma_j(0)) = 2\delta$ for 
$1\le i\neq j\le N$.
For $1\le i\neq j\le N$, the ruled surface
$S_{ij}$ bounded by $\gamma_i$, $\gamma_j$ and 
$S(p,r)$ is defined.  
We call such an $S_{ij}$ {\it properly} embedded in $B(p,r)$.
For $2\le k\le N$, 
we define $E_k$ as the union of  ruled surfaces $S_{ij}$ determined 
by  $\gamma_i$ and $\gamma_j$ for all $1\le i\neq j \le k$. 

\begin{thm} $($\cite[Theorem 4.1]{NSY:local}$)$\label{thm:ruled}
With respect to the interior metric,
$S_{ij}$ is a ${\rm CAT}(\kappa)$-space
homeomorphic to a two-disk.

Furthermore, for each $3\le k\le N$, $E_k$ is also a  ${\rm CAT}(\kappa)$-space.
\end{thm}

 See also \cite{LS:dim2} for a related result.

\begin{thm} $($\cite[Theorem 1.1]{NSY:local}$)$\label{thm:embedded-disk}
For small enough $r=r_p>0$, $B(p,r)$ is a finite union of properly embedded $\CAT(\kappa)$-disks
and properly immersed branched  $\CAT(\kappa)$-disks.

In particular, for any simple loop $\zeta$ in $\Sigma_p(X)$, there is a  properly embedded $\CAT(\kappa)$-disk $E$
in $B(p,r)$ such that $\Sigma_p(E)=\zeta$.
\end{thm}
\psmall
For distinct $1\le i, j, k\le N$, we set 
\begin{equation*}
   C_{ijk} := S_{ij}\cap S_{jk}\cap S_{ki}.
\end{equation*}

\begin{center}
\begin{tikzpicture}
[scale = 0.8]
\draw [-, thick] (4.5,0) to [out=150, in=30] (3.5,0);
\draw [-, thick] (3.5,0) to [out=210, in=330] (2.5,0);
\draw [-, thick] (2.5,0) to [out=150, in=30] (1.5,0);
\draw [-, thick] (1.5,0) to [out=210, in=330] (0.5,0);
\draw [-, thick] (0.5,0) to [out=150, in=30] (-0.5,0);
\draw [-, thick] (-0.5,0) to [out=210, in=330] (-1.5,0);
\draw [-, thick] (-1.5,0) to [out=160, in=20] (-2.5,0);
\draw [-, thick] (-2.5,0) to [out=200, in=350] (-3.5,0);
\draw [-, thick] (-3.5,0) to [out=170, in=10] (-4.5,0);
\fill (-4.5,0) coordinate (A) circle (2pt) node [left] {$p$};
\fill (4.5,0) coordinate (A) circle (0pt) node [right] {$C_{ijk}$};
\draw [thick] (4.5,0) arc(0:18:9);
\draw [thick] (4.5,0) arc(15:0:9);
\draw [thick] (4.5,0) arc(0:-18:9);
\draw [thick] (-4.5,0) -- (4.05,2.75);
\draw [thick] (-4.5,0) -- (4.05,-2.75);
\draw [thick, dotted] (-4.5,0) -- (4.8,-2.3);
\draw [thick] (4.2,-2.15) -- (4.8,-2.3);
\fill (4.05,2.75) coordinate (A) circle (0pt) node [right] {$\gamma_i$};
\fill (4.05,-2.75) coordinate (A) circle (0pt) node [right] {$\gamma_j$};
\fill (4.8,-2.3) coordinate (A) circle (0pt) node [right] {$\gamma_k$};
\end{tikzpicture}
\end{center}

\begin{lem} $($cf. \cite [Lemma 6.1]{NSY:local}$)$    \label{lem:sing-arc}
$C_{ijk}$ is a simple rectifiable arc in $\mathcal S(X)$ having definite directions everywhere
satisfying 
\begin{enumerate}
 \item it starts from $p$ and reaches a point of $\partial B(p,r)$;
 \item its length is less than $(1+\tau_p(r))r$;
 \item for all $x,y\in C_{ijk}$, we have 
  $\frac{|d_p(x)-d_p(y)|}{|x,y|}\ge 1-\tau_p(r)$, 
  and $\angle pxy >\pi-\tau_p(r)$ if $d_p(x)<d_p(y)$.
\end{enumerate}
\end{lem}
\begin{proof}
We show only the second inequality in (3), since 
the others are proved in \cite [Lemma 6.1]{NSY:local}.
Let $\pi:X\setminus B(p, d_p(x))\to S(p,d_p(x))$ 
be the projection along the geodesics to $p$.
The first inequality in (3) implies 
\begin{align}\label{eq:ratio=L(C)/dist}
L(C_{x,y})\le (1+\tau_p(r))|x,y|,
\end{align}
where 
$C_{x,y}$ denotes the arc of $C_{ijk}$ between $x$ and $y$.
 It follows from Corollary \ref{cor:vert} that  
 $|x,\pi(y)|\le \tau_p(r)|x,y|$.
 This implies $\tilde\angle yx\pi(y)>\pi/2-\tau_p(r)$.
 Since $\angle yx\pi(y)>\pi/2-\tau_p(r)$ from
 Lemma \ref{lem:comparison},
 we have $\angle pxy>\pi-\tau_p(r)$.
 \end{proof}
\pmed

The set of all topological singular points of 
$E_N$ arises from the intersections of distinct disks
$S_{ij}$ and $S_{i'j'}$ for all $1\le i<j\le N$ and 
$1\le i'<j'\le N$.
In particular, $\ca S(X)\cap E_N$ 
is a finite union of those  rectifiable curves $C_{ijk}$
with $1\le i<j<k\le N$ which have the direction $v$ at 
$p$.
Thus,
$\ca S(X)\cap B(p,r)$ 
is a finite union of those  rectifiable curves $C_{ijk}$
with $1\le i<j<k\le N_v$ for all $v\in V(\Sigma_p(X))$.
This yields the following:

\begin{thm}$($\cite[Corollary 1.4]{NSY:local}$)$ \label{thm:graph}
$\ca S(X)$ has a graph structure in a generalized sense.
Moreover, the branching orders of vertices of the graph
$\ca S(X)$ are locally uniformly bounded.
\end{thm}

We denote by $E(\ca S(X))$ and $V(\ca S(X))$ 
the set of all edges and all vertices of the graph.
It should be emphasized that 
the set $V_*(\ca S(X))$ of accumulation points of $V(\ca S(X))$ can be uncountable 
(see \cite[Example 6.9]{NSY:local}).
 We also call $\ca S(X)$ the {\it singular locus} of $X$,
 and a simple arc in $\ca S(X)$ starting from a point $p\in \ca S(X)$ 
  a {\it singular curve} if $d_p$ is strictly increasing.

\pmed\n

\setcounter{equation}{0}

\section{Lipschitz homotopy approximations via polyhedral spaces} \label{sec:Polyhedral}

Let 
$X$ be a two-dimensional  locally compact, geodesically complete 
locally $\CAT(\kappa)$-space.
For the proof of Theorem \ref{thm:approx},
we may assume that  for every $x\in X$,
$\Sigma_x(X)$ is a connected $\CAT(1)$-graph without 
endpoints.

First we recall the surgery construction 
discussed in \cite[Section 7]{NSY:local},
where the surgeries were carried out on the space
$E$, a finite union of ruled surfaces at a reference
point. In the present situation, however, we need to consider only the case of $E$ with maximal number
of ruled surfaces.

We call a continuous curve $c$ in 
$\ca S(X)$ a {\it singular curve} or a 
{\it singular arc}
if $c$ satisfies the monotone property as in 
Lemma \ref{lem:sing-arc}(3).

\begin{defn} \label{defn:singlar-vertex} \upshape
We call $x\in V(\ca S(X))$ a {\it singular vertex}
if  at least one of the following holds:   
\begin{enumerate}
 \item $x\in V_*(\ca S(X))\,;$
 \item  there are singular curves $C_1,C_2$  starting from 
 $x$ such that 
 \begin{align} \label{eq:C1C2}
      \angle_x(C_1,C_2)=0, \quad C_1\cap C_2=\{ x \}.
 \end{align} 
\end{enumerate}
We denote by  $V_{\rm sing}(\ca S(X))$  the set of all 
singular vertices of $\ca S(X)$.

For $x\in V_{\rm sing}(\ca S(X))$, 
we call $v\in \Sigma_x(\ca S(X))$ a {\it singular direction}
if at least one of the following holds:   
\begin{enumerate}
 \item $v=\lim_{n\to\infty} \uparrow_x^{x_n}$ for a sequence $x_n$ in 
 $V(\ca S(X))\,;$
 \item  there are singular curves $C_1,C_2$ starting from 
 $x$ in the direction $v$ satisfying \eqref{eq:C1C2}.
 \end{enumerate}
 We denote by $\Sigma_x^{\rm sing}(\ca S(X))$ the set of all singular directions.   
\end{defn}

By the surgeries, we can remove the set $V_{\rm sing}(\ca S(X))$ of singular vertices to
obtain a polyhedral locally $\CAT(\kappa)$-space.

Take a sequence  of positive numbers $\e_n$
converging to $0$, and fix $n$ for a while.
For $p\in V_{\rm sing}(\ca S(X))$, let $r=r_p$ be  as in Subsection \ref{ssec:lcal-str}.
For $x\in \ca S(X)\cap (B(p, r)\setminus \{ p\})$,
 let 
$\Sigma_{x,+}(\ca S(X)):=\Sigma_x(\ca S(X)\setminus \mathring{B}(p, r(x)))$ and 
$\Sigma_{x,-}(\ca S(X)):=\Sigma_x(\ca S(X)\cap B(p, r(x)))$.
Let $\theta_0:=10^{-10}$.
We may assume that
for all $x\in \ca S(X)\cap (B(p, r)\setminus \{ p\})$,
\beq  \label{eq:10-10}
\begin{aligned}
 & \angle^X(\nabla d_p(x),  \Sigma_{x,+}(\ca S(X)))<\theta_0, \quad
  {\rm diam}(\Sigma_{x,+}(\ca S(X)))<\theta_0, \\
&\angle^X(-\nabla d_p(x),  \Sigma_{x,-}(\ca S(X)))<\theta_0,\quad
   {\rm diam}(\Sigma_{x,-}(\ca S(X)))<\theta_0.
\end{aligned}
\eeq 
In addition, we assume
\begin{align} \label{eq:S(p,r)NotV}
 \text{$S(p,r)\cap V(\ca S(X)) $ is  empty}. \hspace{2cm}
 \end{align}
For every $x\in B(p,r)\cap V_{\rm sing}(\ca S(X))$ and 
every $v\in \Sigma_x^{\rm sing}(\ca S(X))$, 
choose small enough $\delta=\delta_x$ and $\e=\e_x$
with 
\begin{align}\label{eq:epsilon-delta}
\e\ll\delta\le\e_n
\end{align}
satisfying the following:
Consider the cone-like region $C(x,v,\delta,\e)$ at $x$ in the the direction $v$ with angle
$\delta$ and radius $\e$ defined as 
the domain  bounded by  
 geodesics $\gamma_{\xi_i}$ \,$(1\le i\le N_v)$ and $S(x,\e)$, where $\xi_i\in\Sigma_x(X)$ satisfy for all $1\le i\neq j\le N_v$
 $$
 \angle(\xi_i,\xi_j)=
\angle(\xi_i,v)+\angle(v,\xi_j)=2\angle(\xi_i,v)
=2\delta.
$$ 
Then if $d_p$ is increasing (resp. decreasing) in the direction $v$, we have 
$C(x,v,\delta,\e)\cap\ca S(X)=B_+(x,\e)\cap \ca S(X)$
(resp. 
$C(x,v,\delta,\e)\cap\ca S(X)=B_-(x,\e)\cap \ca S(X)$),
where $B_+(x,\e):=B(x,\e)\setminus \mathring{B}(p,d_p(x))$ and 
$B_-(x,\e):=B(x,\e)\cap B(p,d_p(x))$.

Let $\hat T:=C(x,v,\delta,\e)\cap S(p,
d_p(x)+\e)$, which is a tree.
As in \eqref{eq:S(p,r)NotV}, we take $\e$ so as to satisfy 
\begin{align} \label{eq:S(x,epsilon)NotV}
 \text{$S(x,\e)\cap V(\ca S(X)) $ is  empty}. \hspace{2cm}
 \end{align} 
Replacing each edge of $\hat T$ by the 
$X$-geodesic between the endpoints,
we have the geodesic tree $T$.
For $x\in X$ and a closed subset $S$ of $X$,
let $x*S$ be the join defined as the union of all minimal geodesics 
from $x$ to points $s$ for all $s\in S$.
Then $K(x,v,\delta,\e)$ abbreviated as $K(x,v)$ is defined as the  join  
\begin{align}\label{eq:K(xv)=x*T}
K(x,v,\delta,\e)=K(x,v)=x*T.
\end{align}
\eqref{eq:10-10} ensures that $d_x$ is almost constant on $V(T)$.

For each edge $e\in E(T)$, let 
$\triangle(x,e)$ denote the $X$-geodesic triangle with vertices $x$ and $\pa e$.
Now the comparison object $\tilde K(x,v)$ for 
$K(x,v)$
is defined as the natural gluing of the comparison triangle regions $\tilde\blacktriangle(x,e)$
\,$(e\in E(T))$.
Let $\tilde\angle_x(\triangle(x,e))$ denote the 
angle of the comparison triangle $\tilde\triangle(x,e)$
at the vertex corresponding to $x$. 
The basic surgery construction is to replace 
$K(x,v)$ by $\tilde K(x,v)$ which is realized as the gluing of $X\setminus \mathring{K}(x,v)$ and $\tilde K(x,v)$ along their isometric 
boundaries.  In \cite[Section 7]{NSY:local},
we essentially carry out 
the surgery along  $B(p,r)\cap\ca S(X)$ 
in a systematic way to have 
a polyhedral $\CAT(\kappa)$-space
$\tilde B(p,r)$. In this construction, we replace
finitely many such $K(x_\alpha,v_\alpha)$
whose interiors are pairwise disjoint,
by the comparison objects 
$\tilde K_n(x_\alpha,v_\alpha)$.
We call such a union of $K_n(x_\alpha,v_\alpha)$ the {\it surgery part} in the construction of 
$\tilde B(p,r)$.
As a special case in \cite{NSY:local}, we proved that 
$\tilde B(p,r)$
converges to 
$B(p,r)$ as $n\to\infty$.

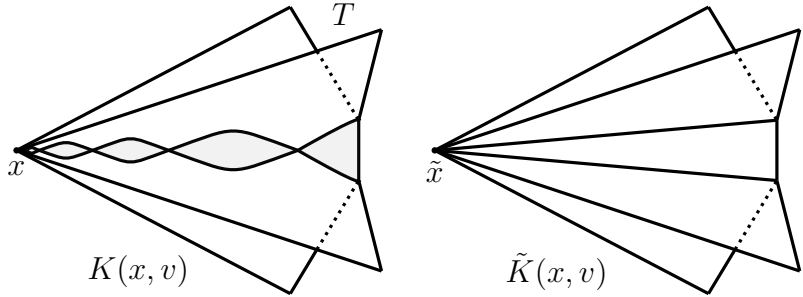
\begin{center}
\begin{tikzpicture}
[scale = 1]
\fill (0,0) coordinate (A) circle (1.2pt) node [below] {$x$};
\draw (4.3,1.8) node[circle] {$T$};
\draw (1.6,-1.6) node[circle] {$K(x,v)$};
\draw [-, very thick] (0,0)--(0.1,0);
\draw [-, very thick] (4.5,0.4)--(4.8,1.6);
\draw [-, very thick] (4.5,-0.4)--(4.8,-1.6);
\draw [-, very thick] (0,0)--(4.8,1.6);
\draw [-, very thick] (0,0)--(4.8,-1.6);
\draw [dotted, very thick] (3.95,1.315)--(4.5,0.4);
\draw [-, very thick] (3.95,1.315)--(3.6,1.9);
\draw [-, very thick] (0,0)--(3.6,1.9);
\draw [-, very thick] (0,0)--(3.6,-1.9);
\draw [dotted, very thick] (3.95,-1.315)--(4.5,-0.4);
\draw [-, very thick] (3.95,-1.315)--(3.6,-1.9);
\coordinate (P1) at (3.7,0);
\coordinate (P2) at (2,0);
\coordinate (P3) at (1,0);
\coordinate (P4) at (0.3,0);
\coordinate (P5) at (0.1,0);
\coordinate (P6) at (4.5,0.4);
\coordinate (P7) at (4.5,-0.4);
\draw [very thick]
(P1) .. controls +(30:1cm) and +(180:0cm) ..
(P6) .. controls +(270:0.5cm) and +(90:0.5cm) ..
(P7) .. controls +(180:0cm) and +(330:1cm) .. (P1); 
\filldraw [fill=gray, opacity=.1]
(P1) .. controls +(30:1cm) and +(180:0cm) ..
(P6) .. controls +(270:0.5cm) and +(90:0.5cm) ..
(P7) .. controls +(180:0cm) and +(330:1cm) .. (P1); 
\draw [very thick]
(P1) .. controls +(160:1cm) and +(20:1cm) ..
(P2) .. controls +(340:1cm) and +(200:1cm) .. (P1); 
\filldraw [fill=gray, opacity=.1]
(P1) .. controls +(160:1cm) and +(20:1cm) ..
(P2) .. controls +(340:1cm) and +(200:1cm) .. (P1); 
\draw [very thick]
(P2) .. controls +(160:0.6cm) and +(20:0.6cm) ..
(P3) .. controls +(340:0.6cm) and +(200:0.6cm) .. (P2); 
\filldraw [fill=gray, opacity=.1]
(P2) .. controls +(160:0.6cm) and +(20:0.6cm) ..
(P3) .. controls +(340:0.6cm) and +(200:0.6cm) .. (P2); 
\draw [very thick]
(P3) .. controls +(160:0.4cm) and +(20:0.4cm) ..
(P4) .. controls +(340:0.4cm) and +(200:0.4cm) .. (P3); 
\filldraw [fill=gray, opacity=.1]
(P3) .. controls +(160:0.4cm) and +(20:0.4cm) ..
(P4) .. controls +(340:0.4cm) and +(200:0.4cm) .. (P3); 
\draw [very thick]
(P4) .. controls +(160:0.15cm) and +(20:0.15cm) ..
(P5) .. controls +(340:0.15cm) and +(200:0.15cm) .. (P4); 
\filldraw [fill=gray, opacity=.1]
(P4) .. controls +(160:0.15cm) and +(20:0.15cm) ..
(P5) .. controls +(340:0.15cm) and +(200:0.15cm) .. (P4); 
\draw [-, very thick] (5.5,0)--(5.6,0);
\fill (5.5,0) coordinate (A) circle (1.2pt) node [below] {$\tilde{x}$};
\draw (7.1,-1.6) node[circle] {$\tilde{K}(x,v)$};
\draw [-, very thick] (5.5,0)--(10,0.4);
\draw [-, very thick] (5.5,0)--(10,-0.4);
\draw [-, very thick] (10,0.4)--(10,-0.4);
\draw [-, very thick] (10,0.4)--(10.3,1.6);
\draw [-, very thick] (10,-0.4)--(10.3,-1.6);
\draw [-, very thick] (5.5,0)--(10.3,1.6);
\draw [-, very thick] (5.5,0)--(10.3,-1.6);
\draw [dotted, very thick] (9.45,1.315)--(10,0.4);
\draw [-, very thick] (9.45,1.315)--(9.1,1.9);
\draw [-, very thick] (5.5,0)--(9.1,1.9);
\draw [-, very thick] (5.5,0)--(9.1,-1.9);
\draw [dotted, very thick] (9.45,-1.315)--(10,-0.4);
\draw [-, very thick] (9.45,-1.315)--(9.1,-1.9);
\end{tikzpicture}
\end{center}

In what follows, we shall continue the above surgery 
construction along all over $\ca S(X)$ to have a geodesically complete polyhedral $\CAT(\kappa)$-space $X_n$.

\pmed
\n
\subsection{GH-approximations} \label{ssec:GH-conv} 

In this subsection, we prove the following.

\begin{thm} \label{thm:GH-approx}
There is a sequence of 
pointed two-dimensional geodesically complete polyhedral 
locally $\CAT(\kappa)$-spaces $(X_n, p_n)$ which converges to $(X,p)$ with respect to the pointed Gromov-Hausdorff
topology.
\end{thm}
\begin{proof} Let $p\in V_{\rm sing}(S(X))$.
Let $r=r_p$, $v\in V(\Sigma_p(X))$, $N=N_v$ and 
$E_k$\, $(2\le k\le N)$ be as in Subsection \ref{ssec:lcal-str}. Consider the union 
\begin{align}\label{eq:E(v)}
E(v):=E_{N_v}
\end{align}
of all ruled surfaces at $p$ in the direction $v$ with  
radius $r$. Fix any $\e_0>0$.
As in the proof of \cite[Theorem 7.4]{NSY:local},  we construct a polyhedral $\CAT(\kappa)$-space $\tilde E(v)$,
which is defined as the result of surgeries  replacing finitely many $K(x_\alpha,v_\alpha)$
contained in $E(v)$ by 
$\tilde K(x_\alpha,v_\alpha)$. Here,
the construction of $\tilde E(v)$
depends on $\e_0$, and 
$\tilde E(v)$ converges to $E(v)$ as
$\e_0\to 0$ for the Gromov-Hausdorff topology.
Note also that the construction of $\tilde E(v)$ is based on surgeries 
 at singular vertices $x\in V_{\rm sing}(\ca S(X))\cap B(p,r)$ and the singular directions
 in  $\Sigma_x^{\rm sing}(\ca S(X))$.

Now consider $\tilde E(v)$ for all $v\in V(\Sigma_p(X))$.
 Gluing $B(p,r)\setminus \bigcup_{v\in V(\Sigma_p(X))} E(v)$ and 
 $\bigcup_{v\in V(\Sigma_p(X))} \tilde E(v)$ along their isometric boundaries, we have a 
  polyhedral $\CAT(\kappa)$-space $\tilde B(p,r)$
 which converges to $B(p,r)$ as
 $\e_0\to 0$. 
 We may also assume $\pa\tilde B(p,r)=S(p,r)$.
 In other words,
 we do not touch $\pa B(p,r)$ in the course of the surgery construction of  $\tilde B(p,r)$
 (see \eqref{eq:S(p,r)NotV}).

  Let $\{ B(p_i,r_i)\}_{i}$ be a locally finite covering of 
 $ V_{\rm sing}(\ca S(X))$ 
 by metric  balls $B(p_i,r_i)$ as above with 
 polyhedral $\CAT(\kappa)$-spaces $\tilde B(p_i,r_i)$
 which converge to $B(p_i,r_i)$ as
 $\e_0\to 0$. 
Here we have to modify the construction of 
$\tilde B(p_i,r_i)$ taking into account of the compatibility of
locally defined surgeries.

 First set 
 $\tilde N_1 :=\tilde B(p_1,r_1)$,
and define  the gluing
$$
 M_1:= (X\setminus  B(p_1,r_1))\cup \tilde N_1,
$$
which is locally $\CAT(\kappa)$.
Next, in the construction of  $\tilde B(p_2,r_2)$, 
we modify the construction so that 
we perform the surgeries only in $B(p_2,r_2)\setminus
B(p_1,r_1)$. This is certainly possible, since  
$S(p_1,r_1)$ does not meet $V(\ca S(X))$.
We denote by the same symbol  $\tilde B(p_2,r_2)$
the result of the above surgeries in  $B(p_2,r_2)$.
Then we set
$\tilde N_2 :=\tilde B(p_2,r_2)\setminus 
\mathring{B}(p_1,r_1)$.
Note that $\pa \tilde N_2$ is isometric to 
$\pa (B(p_2,r_2)\setminus B(p_1,r_1))$.
Now define the gluing
$$
    M_2:= (M_1\setminus {\rm int} (B(p_2,r_2)\setminus B(p_1,r_1)))\cup \tilde N_2,
$$
which is locally $\CAT(\kappa)$.
 Assuming a locally $\CAT(\kappa)$-space $M_{k-1}$ is defined, we inductively define 
the gluing
$$
M_k:= (M_{k-1}\setminus ({\rm int}(B(p_k,r_k)\setminus  (B(p_1,r_1)\cup\cdots\cup B(p_{k-1},r_{k-1})))))
 \cup \tilde N_k,
$$
where 
 $\tilde N_k :=\tilde B(p_k,r_k)\setminus
 (\mathring{B}(p_1,r_1)\cup\cdots\cup 
 \mathring{B}(p_{k-1},r_{k-1}))$,
 and 
 in the construction of  $\tilde B(p_k,r_k)$, 
we perform the surgeries only in $B(p_k,r_k)\setminus
(B(p_1,r_1)\cup\cdots\cup B(p_{k-1},r_{k-1}))$. 
Observe that $M_k$ is locally $\CAT(\kappa)$.
Repeating this procedure possibly infinitely many times, we obtain a locally $\CAT(\kappa)$-space $\tilde X$ as 
the inductive limit of $M_k$.

  From here on, we call the above mentioned surgeries  in $X$ 
 {\it $\e_0$-surgeries}, and $\tilde X$ as the result of the  $\e_0$-surgeries.

 We show the convergence $(\tilde X,\tilde p)\to (X,p)$
 as $\e_0\to 0$. Here we may assume $p\in X\setminus \ca S(X)$.
 Since $p$ is not in the surgery part in the construction of 
 $\tilde X$ for small enough $\e_0$, we may take  
 $\tilde p=p$.
 Let $W(\e_0)=\bigcup_{\alpha\in A} K(x_\alpha,v_\alpha)\subset X$ be the $\e_0$-surgery part 
 in the construction of $\tilde X$,  where
 $|A|\le\infty$ and $K(x_\alpha,v_\alpha)$ is a piece in the surgery part defined as in \eqref{eq:K(xv)=x*T}.
 Then we have 
 a natural inclusion $\iota:X\setminus W(\e_0) \to
 \iota(X\setminus W(\e_0))\subset\tilde X$, and 
 $\tilde X$ is the gluing of $\iota(X\setminus W(\e_0))$
 and $\bigcup_{\alpha\in A} 
 \tilde K(x_\alpha,v_\alpha)$ along the isometric 
 boundaries.
 We define $\varphi:\tilde X\to X$ in such a way that 
 \begin{enumerate}
 \item  $\varphi|_{\iota(X\setminus W(\e_0))}$ is the canonical
 identification $\iota(X\setminus W(\e_0))\to 
 X\setminus W(\e_0)\,;$
 \item  $\varphi:\tilde K(x_\alpha,v_\alpha)\to  K(x_\alpha,v_\alpha)$ is any map
 for each $\alpha\in A$.
 \end{enumerate}
 We show that $\varphi$ provides a $\tau(\e_0)$-approximation.
 Clearly, $\varphi(\tilde X)$ is $\tau(\e_0)$-dense in 
 $X$. 
 %
  For any $R>0$, let $N(R)$ be the maximum of branching 
 numbers $N_v$ for all $v\in V(\Sigma_x(X))$
 and $x\in B(p,R)$.
 Then for arbitrary $\tilde y,\tilde y'\in 
 B(\tilde p,R)$, from the argument in the final stage of the 
 proof of Theorem 7.4 in  \cite{NSY:local},
 we have 
 \[
    ||\varphi(\tilde y), \varphi(\tilde y')|-|\tilde y,\tilde y'||
       \le RN(R)\tau(\e_0).
\]
  This completes the proof of Theorem \ref{thm:GH-approx}.
\end{proof}

\pmed
\n
\subsection{Lipschitz homotopy approximations} \label{ssec:circular} 

\psmall
In what follows, we prove Theorem \ref{thm:approx}.
 Let $M^2_\kappa$ be the simply connected complete surface of constant curvature $\kappa$.

\begin{defn}\label{defn:Graph} \upshape
For a generalized graph $G$, $\pa G$ denotes the set of all endpoints of $G$ (if any). 
We denote by $E(G)$ and $V(G)$ the sets of all edges 
and of all vertices of $G$ respectively.
Let  $E_{\rm ext}(G)$ denote the set of all exterior edges  of  $G$ meeting $\pa G$,
and let $E_{\rm int}(G):=E(G)\setminus E_{\rm ext}(G)$,
the set of all interior edges  of  $G$.
Let $G_*$ be the {\it core} of $G$ defined as 
the closure of the complement of the union of  all $\bar e$
for $e\in E_{\rm ext}(G)$.
\end{defn}
For $A\subset X$ and $C>0$, 
we denote by $\tau_{A,C}(\e)$ a positive function 
depending on $A,C$ such that 
$\lim_{\e\to 0}\tau_{A,C}(\e)=0$.

\pmed\n
{\bf Structure of $K(x,v)$}.
For any $x\in V_{\rm sing}(\ca S(X))$ and 
$v\in\Sigma^{\rm sing}_x(\ca S(X))$, let $K(x, v)$
and $\tilde K(x, v)$  be defined as before.
Recall that $T$ is a geodesic tree such that 
$d_p$ is constant on $V(T)$, and hence
$d_x$ is almost constant on $V(T)$ by
\eqref{eq:10-10}.

First we investigate the structure of $K(x,v)=K(x,v,\delta,\e)=x*T$ for small enough $\e$, where 
$x,v$ and $\delta$ are fixed.
Recall  \eqref{eq:epsilon-delta} and \eqref{eq:K(xv)=x*T}.
In particular, $K(x,v)$ is the closed domain bounded by 
the geodesics $\sigma_i:=\gamma_{\xi_i}$ \, $(1\le i\le N_v)$
and $T$. 
For each $1\le i\le N_v$, let $F_i=F_i(x,v)$ be the connected component of  $K(x,v)\setminus \ca S(X)$ containing $\sigma_i$. We call $F_i$ the {\it $i$-th wing} of 
$K(x,v)$.
Let $W_*$ be the {\it core} of $K(x,v)$ defined as 
the complement of 
$\mathring F_1\cup  \cdots \cup \mathring F_{N_v}$.

Let $e$ be any edge in $E_{\rm int}(T)$ (if any) with $\{ y, y'\}:=\pa e$.
Let $A_{e}$ be the disk domain in $K(x,v)$
bounded by $e$ and the two singular curves,
say $C, C'$ starting from $y,y'$ respectively, having  the same endpoint, say $u$, such that 
$A_{e}\setminus (e\cup C\cup C')$
does not meet $\ca S(X)$ and is an 
open disk in $X$.

\begin{lem} \label{eq:L(C)/L(e)}
$$
       |y,y'| <\tau_x(\e) \min\{ |y,u|, |y',u|\}.
$$
\end{lem}
\begin{proof}
Using Lemma \ref{lem:sing-arc}  together with \eqref{eq:10-10}, we have 
\begin{align} \label{eq:angle(zyy)}
|\angle uyy'-\angle p yy'|\le 
\angle pyx+\angle xyu <\theta_0+\tau_x(\e),
\end{align}
which implies $|\angle uyy' -\pi/2|<\tau_x(\e) +\theta_0$. Similarly we have 
 $|\angle uy'y -\pi/2|<\tau_x(\e) +\theta_0$, and hence
\begin{align}\label{eq:angle=tilde(zyy)}
|\tilde\angle uyy' -\pi/2|<\tau_x(\e) +\theta_0,
\quad
|\tilde\angle uy'y -\pi/2|<\tau_x(\e) +\theta_0.
\end{align}
Let us assume $d_u(y)\le d_u(y')$ and take 
the point $w'\in C'$ with $d_u(y)=d_u(w')$.
From \eqref{eq:angle=tilde(zyy)}, we have $|y,y'|<2|y,w'|$.
In a way similar to \eqref{eq:angle(zyy)},
we then have 
$|\angle uyw' -\pi/2|<\tau_x(\e)$ and 
$|\angle uw'y -\pi/2|<\tau_x(\e)$
which yield 
\begin{align}\label{eq:angle=tilde(zyw)}
|\tilde\angle uyw'-\pi/2|<\tau_x(\e),
\end{align}
and hence 
$|y,y'|\le \tau_x(\e) \min\{ |y,u|, |y',u|\}$.
This completes the proof.
\end{proof}

From now, we consider a large positive number $R\gg 1$, which will be fixed finally. Clearly we have  
$\lim_{\e\to 0} L(e)/\e=0$.
We consider small enough $\e$ with 
$$
      RL(e)\ll \e, \quad \tau_x(\e)\ll 1/R,
$$ 
and set $e':=A_{e}\cap S(x, \e-RL(e))$.
Let $D_{e}$ be the disk domain in $A_{e}$ bounded by 
$e$, $C$, $C'$ and $e'$,
and let $z:=e'\cap C$ and  $z':=e'\cap C'$.
Let $[\tilde y,\tilde y']=\tilde e$ be the edge of $\tilde T$ corresponding to $e$, and
 $\tilde \blacktriangle(x,e):=\tilde x*\tilde e\subset \tilde K(x,v)$.
Set $\tilde C:=\gamma_{\tilde y, \tilde x}$, 
$\tilde C':=\gamma_{\tilde y', \tilde x}$.
Let $\tilde e':=\tilde \blacktriangle(x,e)\cap  S(\tilde x, \e-RL(e))$, and 
let $\tilde D_e\subset \tilde\blacktriangle(x,e)$
be the domain bounded by 
$\tilde e$, $\tilde C$, $\tilde C'$ and $\tilde e'$. 
Set $\tilde z:=\tilde e'\cap \tilde C$ and  
$\tilde z':=\tilde e'\cap \tilde C'$.
\eqref{eq:angle=tilde(zyw)} shows  
\begin{align}\label{eq:quot=ell/ell}
     | L(e)/L(e') -1|<\tau_x(\e).
\end{align}

\begin{center}
\begin{tikzpicture}
[scale = 1]
\coordinate (P1) at (3.7,0);
\coordinate (P2) at (2,0);
\coordinate (P3) at (1,0);
\coordinate (P4) at (0.3,0);
\coordinate (P5) at (0.1,0);
\draw [-, very thick] (3.7,0) to [out=30, in=180] (9,1);
\draw [-, very thick] (3.7,0) to [out=330, in=180] (9,-1);
\draw [-, very thick] (9,1) -- (9,-1);
\filldraw [fill=gray, opacity=.25]
(3.7,0) to [out=30, in=180] (9,1) -- (9,-1) 
to [out=180, in=330] (3.7,0);
\draw [very thick]
(P1) .. controls +(160:1cm) and +(20:1cm) ..
(P2) .. controls +(340:1cm) and +(200:1cm) .. (P1); 
\filldraw [fill=gray, opacity=.1]
(P1) .. controls +(160:1cm) and +(20:1cm) ..
(P2) .. controls +(340:1cm) and +(200:1cm) .. (P1); 
\draw [very thick]
(P2) .. controls +(160:0.6cm) and +(20:0.6cm) ..
(P3) .. controls +(340:0.6cm) and +(200:0.6cm) .. (P2); 
\filldraw [fill=gray, opacity=.1]
(P2) .. controls +(160:0.6cm) and +(20:0.6cm) ..
(P3) .. controls +(340:0.6cm) and +(200:0.6cm) .. (P2); 
\draw [very thick]
(P3) .. controls +(160:0.4cm) and +(20:0.4cm) ..
(P4) .. controls +(340:0.4cm) and +(200:0.4cm) .. (P3); 
\filldraw [fill=gray, opacity=.1]
(P3) .. controls +(160:0.4cm) and +(20:0.4cm) ..
(P4) .. controls +(340:0.4cm) and +(200:0.4cm) .. (P3); 
\draw [very thick]
(P4) .. controls +(160:0.15cm) and +(20:0.15cm) ..
(P5) .. controls +(340:0.15cm) and +(200:0.15cm) .. (P4); 
\filldraw [fill=gray, opacity=.1]
(P4) .. controls +(160:0.15cm) and +(20:0.15cm) ..
(P5) .. controls +(340:0.15cm) and +(200:0.15cm) .. (P4); 
\fill (3.7,0) coordinate (A) circle (2pt) node [below] {$u$};
\fill (9,1) coordinate (A) circle (2pt) node [above right] {$y'$};
\fill (9,-1) coordinate (A) circle (2pt) node [below right] {$y$};
\fill (9,0) coordinate (A) circle (0pt) node [right] {$e$};
\fill (6.2,1) coordinate (A) circle (0pt) node [above right] {$C'$};
\fill (6.2,-1) coordinate (A) circle (0pt) node [below right] {$C$};
\fill (5.2,0) coordinate (A) circle (0pt) node [right] {$e'$};
\draw [-, very thick] (5.2,0.68) -- (5.2,-0.68);
\fill (5.2,0.68) coordinate (A) circle (2pt) node [above] {$z'$};
\fill (5.2,-0.68) coordinate (A) circle (2pt) node [below] {$z$};
\fill (7.5,0) coordinate (A) circle (0pt) node [left] {$D_e$};
\fill (5,0) coordinate (A) circle (0pt) node [left] {$A_e$};
\filldraw [fill=gray, opacity=.1]
(5.2,0.68) to [out=10, in=180] (9,1) -- (9,-1) 
to [out=180, in=350] (5.2,-0.68);
\end{tikzpicture}
\end{center}

\begin{lem} \label{lem:isom=DtoD}
There exists a $\tau_x(\e)$-almost isometry 
$\varphi:D_e\to \tilde D_e$.
\end{lem}
\begin{proof}
We proceed by contradiction. If the lemma 
does not hold, there is a sequence $\e_i\to 0$ as 
$i\to\infty$ such that in 
$K(x,v,\delta,\e_i)=x*T_i$, for some 
$e_i=[y_i,y_i']\in E_{\rm int}(T_i)$,  the disk domain $D_{e_i}$ defined by $e_i$ 
and its comparison object $\tilde D_{e_i}$ do not admit
any $o_i$-almost isometry $D_{e_i}\to \tilde D_{e_i}$.
Take a ruled surface $S_i=S_{j\ell}$ containing
$D_{e_i}$. Let $a_i:=|y_i,y_i'|$, and
consider the convergence $(\frac{1}{a_i} S_i, y_i)
\to (\R^2,o)$.
\eqref{eq:angle=tilde(zyw)} implies that $D_{e_i}$ converges to 
a quadrangle $Q$, which is almost a rectangle with edges of lengths $1$ and $R$, 
under the above convergence.
Therefore by \cite[Theorem 2.6]{NSY:local}, 
we obtain an
$o_i$-isometric embedding
$\psi_i:D_{e_i}\to {\rm Image}(\psi_i)\subset \R^2$ such that 
${\rm Image}(\psi_i)$ is $o_i$-close to $Q$.
Let $\{ y_i,y_i', z_i',z_i\}$ be the vertices of 
$D_{e_i}$, and 
$\{ y_{\infty,i},y_{\infty,i}', z_{\infty,i}',z_{\infty,i} \}\subset 
{\rm Image}(\psi_i)$ be their $\psi_i$-images.
Then in a way similar to the proof of
Lemma \cite[Lemma 5.1]{NSY:local},
${\rm Image}(\psi_i)$ can be deformed $o_i$-almost isometrically  to the quadrangle 
$Q_i:=y_{\infty,i} y_{\infty,i}' z_{\infty,i}' z_{\infty,i}$
with the same vertices.
Since $Q_i$ converges to $Q$ in the Hausdorff-distance, it is easily verified that $Q_i$ is  $o_i$-almost isometric to $Q$.

On the other hand,
clearly we have an isometric embedding 
$\tilde D_{e_i}\hookrightarrow M^2_\kappa$, and
under the convergence $(\frac{1}{a_i} M^2_\kappa, \tilde y_i) \to (\R^2,o)$,  $\tilde D_{e_i}$ also converges to $Q$. Similarly, we have an
$o_i$-almost isometry $\tilde D_{e_i}\to Q$.
This is a contradiction.
\end{proof}
\psmall\n
{\bf Gluing of almost isometries}.
The following lemma will be used several times in the
proof of Theorem \ref{thm:approx}.

\begin{lem} \label{lem:psi:DtoD}
Given a $\tau_x(\e)$-almost isometry 
$\varphi_0:\pa D_e\to \pa \tilde D_e$
sending each edge of $\pa D_e$  to the corresponding one of $\pa \tilde D_e$,
there exists a $\tau_x(\e)$-almost isometry 
 $\varphi_e:D_e\to \tilde D_e$ 
such that $\varphi_e|_{\pa D_e}=\varphi_0$.
\end{lem}
\begin{proof} Let $0<\delta<1$ be a fixed small number,  and set $a:=|y,y'|$.
Take a point $\hat y$ (resp. $\hat y'$) in the interior of 
$\tilde D_e$ near $\tilde y$ (resp. $\tilde y'$) such that 
$|\hat y, \gamma_{\tilde y,\tilde z}|=|\hat y, \gamma_{\tilde y,\tilde y'}|=\delta a$
(resp.
$|\hat y', \gamma_{\tilde y',\tilde z'}|=|\hat y', \gamma_{\tilde y',\tilde y}|=\delta a$).
Similarly we take $\hat z$ and $\hat z'$ near 
$\tilde z$ and $\tilde z'$ respectively by the similar 
conditions as above.
Consider the quadrangular disk domain $\tilde F\subset {\rm int}\tilde D_e$  bounded by the polygon
$\hat y \hat z\hat z' \hat y'$.

Let $\varphi:D_e\to\tilde D_e$ be the $\tau_x(\e)$-almost isometry given in Lemma \ref{lem:isom=DtoD}, and set
$F:=\varphi^{-1}(\tilde F)$.
In what follows, for any $\tilde u\in \tilde D_e$
and $\tilde A\subset \tilde D_e$,
we use the symbol $u:=\varphi^{-1}(\tilde u)$
and  $A:=\varphi^{-1}(\tilde A)$ implicitly.
Consider the quadrangular domains 
\begin{align*}
&\text{
$\tilde Q_y:=\tilde y\tilde z\hat z\hat y$,\,\,\,
$\tilde Q_z:=\tilde z\tilde z'\hat z'\hat z$,\,\,\,
$\tilde Q_{z'}:=\tilde z'\tilde y'\hat y'\hat z'$,
\,\,\,
$\tilde Q_{y'}:=\tilde y'\tilde y\hat y\hat y'$.
}
\end{align*}
We fix coordinates on these domains.
For $\tilde Q_y$ for instance, 
take the geodesics $\tilde c:=\gamma_{\tilde y,\tilde z}$ and 
$\hat c:=\gamma_{\hat y,\hat z}$, parametrized 
on $[0,|\tilde y,\tilde z|]$, 
and define the product coordinates  
$(u, t)\in [0,|\tilde y,\tilde z|]\times [0,\delta a]$ defined by 
$(u, t)\mapsto \gamma_{\tilde c(u),\hat c(u)}(t)$.
Thus we have the compatible product  structures on
$\tilde D_e\setminus {\rm int} \tilde F$, and hence the one on $D_e\setminus\mathring F$ induced by $\varphi$ from those on 
$\tilde D_e\setminus {\rm int} \tilde F$.

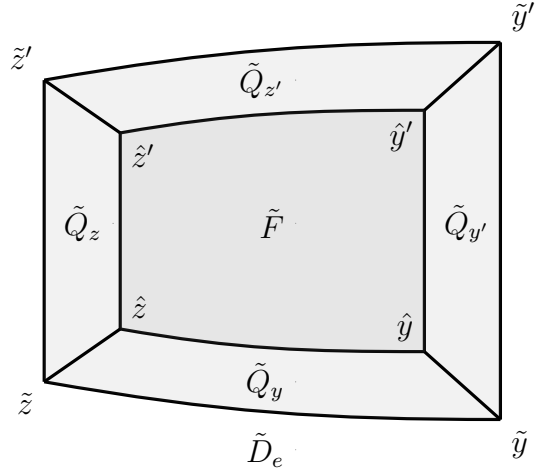
\begin{center}
\begin{tikzpicture}
[scale = 1]
\draw [-, very thick] (2,2) to [out=10, in=180] (8,2.5);
\draw [-, very thick] (2,-2) to [out=350, in=180] (8,-2.5);
\draw [-, very thick] (8,2.5) -- (8,-2.5);
\draw [-, very thick] (2,2) -- (2,-2);
\draw [-, very thick] (3,1.3) to [out=10, in=180] (7,1.6);
\draw [-, very thick] (3,-1.3) to [out=350, in=180] (7,-1.6);
\draw [-, very thick] (7,1.6) -- (7,-1.6);
\draw [-, very thick] (3,1.3) -- (3,-1.3);
\filldraw [fill=gray, opacity=.1]
(2,2) to [out=10, in=180] (8,2.5) -- (8,-2.5) to [out=180, in=350]
(2,-2) -- (2,2);
\filldraw [fill=gray, opacity=.1]
(3,1.3) to [out=10, in=180] (7,1.6) -- (7,-1.6) to [out=180, in=350]
(3,-1.3) -- (3,1.3);
\draw [-, very thick] (3,1.3) -- (2,2);
\draw [-, very thick] (3,-1.3) -- (2,-2);
\draw [-, very thick] (7,1.6) -- (8,2.5);
\draw [-, very thick] (7,-1.6) -- (8,-2.5);
\fill (5.3,0.1) coordinate (A) circle (0pt) node [left] {$\tilde{F}$};
\fill (5.3,2) coordinate (A) circle (0pt) node [left] {$\tilde{Q}_{z'}$};
\fill (5.3,-2) coordinate (A) circle (0pt) node [left] {$\tilde{Q}_y$};
\fill (2.9,0.1) coordinate (A) circle (0pt) node [left] {$\tilde{Q}_z$};
\fill (8,0.1) coordinate (A) circle (0pt) node [left] {$\tilde{Q}_{y'}$};
\fill (5.3,-2.9) coordinate (A) circle (0pt) node [left] {$\tilde{D}_e$};
\fill (3,1.3) coordinate (A) circle (0pt) node [below right] {$\hat{z}'$};
\fill (3,-1.3) coordinate (A) circle (0pt) node [above right] {$\hat{z}$};
\fill (7,1.6) coordinate (A) circle (0pt) node [below left] {$\hat{y}'$};
\fill (7,-1.6) coordinate (A) circle (0pt) node [above left] {$\hat{y}$};
\fill (2,2) coordinate (A) circle (0pt) node [above left] {$\tilde{z}'$};
\fill (2,-2) coordinate (A) circle (0pt) node [below left] {$\tilde{z}$};
\fill (8,2.5) coordinate (A) circle (0pt) node [above right] {$\tilde{y}'$};
\fill (8,-2.5) coordinate (A) circle (0pt) node [below right] {$\tilde{y}$};
\end{tikzpicture}
\end{center}

Now consider the restriction
$\varphi|_{\pa F}: \pa F \to\pa\tilde F$.
Using the product structures on 
$D_e\setminus\mathring F$  and 
$\tilde D_e\setminus
{\rm int} \tilde F$,
we extend $\varphi_0$ and $\varphi|_{\pa F}$ to 
the map $g:D_e\setminus\mathring F\to
\tilde D_e\setminus {\rm int} \tilde F$ 
as follows.
Using the product coordinates 
$(u,t)\in [0,|\tilde y,\tilde z|]\times [0,\delta a]$ of $\tilde Q_y$, we define $g:Q_y\to\tilde Q_y$ by
\[
 g(\varphi^{-1}(u,t))=
 \gamma_{\tilde\varphi_0(u,0), (u,\delta a)}(t),
\]
where $\tilde\varphi_0:=\varphi_0\circ\varphi^{-1}$, 
$(u,\delta a)$ is the point of $\tilde Q_y$ with coordinates
$(u,\delta a)$
 and the geodesic 
$\gamma_{\tilde\varphi_0(u,0), (u,\delta a)}$
is parametrized on $[0,\delta a]$. 
Clearly $g\circ\varphi^{-1}$ is Lipschitz, and hence differentiable
almost everywhere.  

From the assumption on $\varphi_0$, we have
\begin{align}\label{eq:d(varphi0)}
|\tilde\varphi_0(u,0), (u,0)|<\tau_x(\e)Ra.
\end{align}
Fix any $(u,t)\in \tilde Q_y$.
\eqref{eq:d(varphi0)} implies that 
\begin{align}\label{eq:isom-t}
      ||\pa (g\circ\varphi^{-1})/\pa t|-1|<\tau_{x,R,\delta}(\e).
\end{align}

Set $\gamma_0(t):=\gamma_{(u,0), (u,\delta a)}(t)$
and 
$\gamma(t):=\gamma_{\tilde\varphi_0(u,0), (u,\delta a)}(t)$.
Let 
$$
\theta_0:=\angle \tilde y (u,0) (u,\delta a),\quad
\theta:=\angle \tilde y \tilde\varphi_0(u,0) (u,\delta a).
$$
From construction, we have  
\begin{align}\label{eq:d1d2}
         0<d_1<\theta_0<\pi-d_2
\end{align}
for some uniform constants $d_1,d_2$.
\eqref{eq:d(varphi0)} implies that 
\begin{align}\label{eq:theta-theta0}
     |\theta-\theta_0|<\tau_{x,R,\delta}(\e).
\end{align}
Choose any sequence $u_n\to u$, and set
$v_n:=\tilde\varphi_0(u_n)$, $v:=\tilde\varphi_0(u)$.
 From
\eqref{eq:d1d2}  and \eqref{eq:theta-theta0},
we obtain 
\beq
\begin{aligned} \label{eq:isom-u}
   & \biggl|\frac{|(v_n,t), (v,t)|}
      {|(u_n,t),(u,t)|}-1\biggr|<\tau_{x,R,\delta}(\e) \\
& \bigl|\angle\bigl(\dot\gamma(t),
        \uparrow_{(v,t)}^{(v_n,t)}\bigr)-
  \angle\bigl(\dot\gamma_0(t),
        \uparrow_{(u,t)}^{(u_n,t)}\bigr)\bigr|<\tau_{x,R,\delta}(\e)
\end{aligned}
\eeq
\eqref{eq:isom-t} and \eqref{eq:isom-u} yield that
the derivative
$d(g\circ\varphi^{-1})$ is $\tau_{x,R,\delta}(\e)$-almost
isometric almost everywhere, and hence
$g\circ\varphi^{-1}$ is $\tau_{x,R,\delta}(\e)$-almost
isometric. 
Thus $g$ is $\tau_{x,R,\delta}(\e)$-almost isometric on $Q_y$,
and hence similarly on $Q_z,Q_{z'}$ and $Q_{y'}$.
From construction, $g$ defined on those domain 
 can be glued together, and provides a
$\tau_{x,R,\delta}(\e)$-almost isometry $g:D_e\setminus\mathring F\to
\tilde D_e\setminus {\rm int}(\tilde P)$
extending  $\varphi_0$ and $\varphi|_{\pa F}$. 
 This completes the proof.
\end{proof}

\begin{rem} \label{rem:almost-iso}\upshape
 Lemma \ref{lem:psi:DtoD} holds for other situations. See the proof of Lemma \ref{lem:induction=Di}.
\end{rem}

For any $y\in V_{\rm int}(T)$, let $e_1,\ldots, e_\ell$
be the  edges of $T$ adjacent to $y$ with
$e_i=[y,y_i']$. For small enough $\e$,
$D_{e_i}=D_{e_i}(R)$ are defined for all $1\le i\le \ell$.
Let $v_y\in\Sigma_y(\ca S(X))$ be the unique
direction directing  to $x$.
For any $t\in [d_x(y)-R\min\{ L(e_i)\}, d_x(y)]$,
let $y_i(t)\in\pa D_{e_i}\cap S(x,t)$ be the singular curve reaching $y$.
Let $T_*(y,t)$ be the connected subtree of 
$S(x,t)\cap K(x,v)$ with $\pa T_*(y,t)=\{ y_i(t)\}_{i=1}^\ell$.

Set $I_y:=[d_x(y)-R\min\{ L(e_i)\}, d_x(y)]$, and 
\beq  \label{eq:T(y,t)}  
      T_*(y):=\bigcup_{t\in I_y} T_*(y,t).
\eeq

\begin{lem} \label{lem:no-swinging}
For any  singular curve $C$ starting from $y$ in the 
direction $v_y$, we have  for all $t\in I_y$
\begin{align}\label{eq:CcapS=T}
      C\cap S(x,t)\subset T_*(y,t).
\end{align}
\end{lem}
\begin{proof}
\eqref{eq:CcapS=T} certainly holds if $t$ is close to 
$d_x(y)$. The conclusion follows from an easy 
continuity argument.
\end{proof}

\begin{lem}\label{lem:monotone-De}
For  $e_1',e_2', e_3'\in E_{\rm int}(T)$, 
let $t_0:=\min_{1\le i\le 3}d_x(e_i')$ and $t_1:=\max_{1\le i\le 3}d_x(A_{e_i'})$.
 For any $t_1'\in (t_1,t_0)$, set $I:=[t_1',t_0]$.
For $t\in I$, let $c_t$ be the smallest arc in the tree $S(x,t)\cap K(x,v)$
containing the arcs $A_{e_1'}\cap S(x,t)$ and $A_{e_3'}\cap S(x,t)$.
If 	$A_{e_1'}\cap S(x,t_0)$,  $A_{e_2'}\cap S(x,t_0)$ and 
$A_{ e_3'}\cap S(x,t_0)$  lay in this order on $c_{t_0}$,  
then 
$A_{e_1'}\cap S(x,t)$,  $A_{e_2'}\cap S(x,t)$, 
$A_{ e_3'}\cap S(x,t)$  also lay in this order on $c_t$ for all $t\in I$.
\end{lem}
\begin{proof} 
First of all, from the choice of $t_1'$, $A_{e_2'}\cap S(x,t)$ is an arc for each $t\in I$.
We shall verify that $A_{e_2'}\cap S(x,t)\subset c_t$
for all $t\in I$. Let $G$ be the set of such
$t\in I$ with  $A_{e_2'}\cap S(x,t)\subset c_t$.
Obviously $G$ is closed. 
If $G$ is not open, choose $t\in G$ and 
$t_n\in I\setminus G$ converging to $t$.
 Then from continuity,  $A_{e_2'}\cap S(x,t_n)\cap c_{t_n}$ is nonempty  for large $n$.
Since $A_{e_2'}\cap S(x,t_n)\setminus c_{t_n}$ is also nonempty,
This implies $\mathring{A}_{e_2'} \cap  S(x,t_n)$ meets 
$\ca S(X)$, which is a contradiction.

Now the conclusion immediately follows from continuity.
\end{proof}

\begin{rem} \label{rem:noswinging}\upshape
As shown in \cite[Example 4.5]{NSY:local}, 
the singular curves starting from $x$ in the direction 
$v$ can meet each other quite often in a strange way
that there are a lot of twistings of the wings along singular 
curves. In particular, \eqref{eq:CcapS=T} fails for smaller
$t>0$.
Lemmas \ref{lem:no-swinging}  and \ref{lem:monotone-De} show that there are no such 
twistings  relatively near $T$.
\end{rem}

\pmed  
\begin{proof}[{\bf Proof of Theorem \ref{thm:approx}}]
Take a sequence $\e_n>0$ converging to $0$.
By the proof of Theorem \ref{thm:GH-approx}, we have a sequence of 
pointed two-dimensional locally compact geodesically complete polyhedral 
locally $\CAT(\kappa)$-spaces $(X_n, q_n)$ obtained as a result of  $\e_n$-surgeries of $(X,q)$, which converges to $(X,q)$ as $n\to \infty$ for the pointed Gromov-Hausdorff
topology.
Replacing $X_n$ by another polyhedral 
locally $\CAT(\kappa)$-space if necessary
(see Remark \ref{rem:choice=onstant}),
we show that $X_n$ has the same Lipschitz homotopy type as $X$ via 
$(C_n,\e_n)$-Lipshitz 
 homotopies with $\lim_{n\to\infty}C_n=1$,
 fixing the outside of the 
$\e_n$-surgery part of $(X,q)$.

In what follows, we omit the subscript $n$,  
writing $\tilde X=X_n$, etc. 
Recall that $\tilde X$ is obtained as the gluing 
\[
\tilde X :=(X\setminus \bigcup_{\alpha\in A} K(x_\alpha,v_\alpha))
\cup(\bigcup_{\alpha\in A} \tilde K(x_\alpha,v_\alpha)).
\]

For each $\alpha\in A$, 
let  $\iota_\alpha:\pa K(x_\alpha, v_\alpha)\to \pa\tilde K(x_\alpha,v_\alpha)$ and 
   $\tilde\iota_\alpha:\pa\tilde K(x_\alpha, v_\alpha)\to \pa K(x_\alpha,v_\alpha)$ be the identical maps. 
Since the surgeries in the construction of $\tilde X$ are performed locally and since $K(x_\alpha, v_\alpha)$ is
convex in $X$ and $\tilde K(x_\alpha, v_\alpha)$ is 
$\CAT(\kappa)$, for the proof of Theorem
\ref{thm:approx}, it suffices to construct $C_n$-Lipschitz maps $f_\alpha:K(x_\alpha,v_\alpha)\to \tilde K(x_\alpha,v_\alpha)$ and  $\tilde f_\alpha:\tilde K(x_\alpha,v_\alpha)\to K(x_\alpha,v_\alpha)$ such that
\begin{align} \label{eq:homotopiesFF}
 f_\alpha|_{\pa K(x_\alpha,v_\alpha)}=\iota_\alpha, \qquad
 \tilde f_\alpha|_{\pa \tilde K(x_\alpha,v_\alpha)}=\tilde\iota_\alpha.
\end{align}
In fact, we then define the homotopy $F:K(x_\alpha,v_\alpha)\times [0,\e_n]
\to K(x_\alpha,v_\alpha)$ between 
$\tilde f_\alpha\circ f_\alpha$ and ${\rm id}_{K(x_\alpha, v_\alpha)}$ 
by
\[
      F(y,s)=\gamma_{y, \tilde f_\alpha\circ f_\alpha(y)}
               (s/\e_n),
\]
where $\gamma_{y, \tilde f_\alpha\circ f_\alpha(y)}
:[0,1]\to K(x_\alpha,v_\alpha)$ is the minimal geodesic 
from $y$ to $\tilde f_\alpha\circ f_\alpha(y)$.
Note that the diameter of $K(x_\alpha,v_\alpha)$
is less than $\e_n$ by \eqref{eq:epsilon-delta}.

\begin{lem}\label{lem:geodesic-homotopy}
$F$ is $C_n$-Lipschitz with $\lim_{n\to\infty} C_n=1$.
\end{lem}
\begin{proof}
Clearly, $|F(y,s),F(y,s')|\le |s-s'|$.
Set $g:=\tilde f_\alpha\circ f_\alpha$.
For arbitrary $y,z\in K(x_\alpha,v_\alpha)$,
consider the point $\gamma_{y,g(z)}(s)$.
From the 
curvature condition, we have
\begin{align*}
    & |F(y,s),\gamma_{y,g(z)}(s)| \le C_n s|g(y),g(z)|,  \\
    & |\gamma_{y,g(z)}(s), F(z,s)| \le C_n(1-s)|y,z|. 
\end{align*}
It follows from the triangle inequality that 
$|F(y,s),F(z,s)|\le C_n'|y,z|$ with $\lim_{n\to\infty} C_n'=1$.
This completes the proof.
\end{proof}
Similarly we can define the $(C_n,\e_n)$-Lipschitz homotopy between 
$f_\alpha\circ\tilde f_\alpha$ and ${\rm id}_{\tilde K(x_\alpha, v_\alpha)}$.

\pmed
In what follows, we construct the maps $f_\alpha$ and 
$\tilde f_\alpha$.
We omit the subscript $\alpha$.
Let $k:=\sharp E_{\rm int}(T)$.
Choose a shortest edge $e_1$ in $E_{\rm int}(T)$
(if any). We relabel the elements of 
$E_{\rm int}(T)$ in such a way that 
$e_1\cup\cdots \cup e_i$ is connected for each 
$1\le i\le k$.
We set $T_i:=e_1\cup\cdots \cup e_i$, and 
$a_i:=L(e_i)$.

For each $i$, 
 let $\{ y_i, y_i'\}:=\pa e_i$, and let 
$D_{e_i}=D_{e_i}(R_i)$ be the disk domain in $K(x,v)$ bounded by $e_i$,  the two singular
curves $C_i, C_i'$ starting from $y_i, y_i'$ respectively,
and the arc $e_i':=D_{e_i}\cap S(x,\e-R_i a_i)$.
Let $\{ z_i, z_i'\}:=\pa e_i'$ and $\{ \tilde z_i, \tilde z_i'\}:=\pa \tilde e_i'$ as before.
Here we may assume 
\begin{align*}
\text{
$a_i\le a_j$ if and only if $R_i\le R_j$.}
\end{align*}

Note that the ratios $R_i/R_1$ might be quite large.
Set
\begin{align*}
 D_{\rm int}^i:= \bigcup_{j=1}^i D_{e_j},
    \quad
\tilde D_{\rm int}^i:= \bigcup_{j=1}^i \tilde D_{e_j}, \quad
D_{\rm int}:=D_{\rm int}^k,
    \end{align*}
where $\tilde D_{e_i}=\tilde D_{e_i}(R_i)$ is also defined 
as before. 

We use the symbols $\tau_{i}$ and $\tau_{x,i}(\e)$
to denote functions of the forms  
\begin{align*}
\tau_{i}=\tau(1/R_1,\ldots,1/R_i),
\quad
\tau_{x,i}(\e)=\tau_x(\e)+\tau_{i}.
\end{align*}
We say that $a$ is $c$-almost equal to $b$ if 
$|a-b|<c$.

\begin{lem}\label{lem:induction=Di}
There exists a $(1+\tau_{x,i}(\e))$-Lipschitz map
 $f_i:D_{\rm int}^i \to\tilde K(x,v)$ extending
$f_{i-1}:D_{\rm int}^{i-1} \to\tilde K(x,v)$
such that 
\begin{enumerate}
\item $f_i$ is identical on $T_i\,;$ 
\item $f_i$  sends each singular curve in $D_{\rm int}^i$
to a broken geodesic, $\tau_{x,i}(\e)$-almost isometrically
with respect to the length metric$\,;$ 
\item $f_{i}(e_j')$ is a point for each $1\le j\le i\,;$
\item $f_i$ {\it respects} $d_x$, $d_{\tilde x}$ on 
each singular curves in $D_{\rm int}^i$, i.e., $d_{\tilde x}\circ f_i=d_x$ on 
$D_{\rm int}^i\cap\ca S(X)\,;$
\item $f_i(D_{\rm int}^i)$ is a polygonal domain, denoted by $P_i$, in $\tilde K(x,v)$
such that 
 \begin{itemize}
  \item the number of vertices of the polygon $\pa P_i$ is at most $2i+1\,;$
  \item the inner angles of $P_i$ at $\pa T_i$
  are $\tau_{i}$-almost equal to corresponding inner angles of $\tilde D_{\rm int}^i\,;$
  \item the nearest point of $P_i$ from $\tilde x$
  belongs to $\tilde C_1$, and the inner angle of $P_i$ there is $\tau_{x,i}(\e)$-almost equal to $0\,;$
  \item the inner angles of $P_i$ at the other vertices
  of $\pa P_i$ are $\tau_{i}(\e)$-almost equal to $\pi$. 
  \end{itemize}
\end{enumerate}
\end{lem}

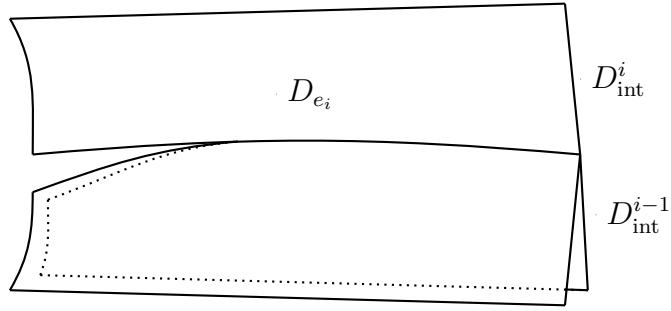
\begin{center}
\begin{tikzpicture}
[scale = 1]
\fill (8,1) coordinate (A) circle (0pt) node [right] {$D_{\mathrm{int}}^i$};
\fill (8.2,-0.8) coordinate (A) circle (0pt) node [right] {$D_{\mathrm{int}}^{i-1}$};
\fill (4,0.8) coordinate (A) circle (0pt) node [right] {$D_{e_i}$};
\draw [thick] (8,0) -- (7.8,2);
\draw [thick] (8,0) -- (7.8,-2);
\draw [thick] (8,0) -- (8.1,-1.8);
\draw [thick] (0.8,0) to [out=5, in=175] (8,0);
\draw [thick] (0.8,0) to [out=90, in=300] (0.5,1.8);
\draw [thick] (0.8,-0.5) to [out=270, in=60] (0.5,-1.8);
\draw [thick] (0.8,-0.5) to [out=20, in=185] (3.5,0.172);
\draw [thick, dotted] (1,-0.6) to [out=270, in=70] (0.9,-1.6);
\draw [thick, dotted] (1,-0.6) to [out=20, in=185] (3.5,0.172);
\draw [thick] (0.5,1.8) -- (7.8,2);
\draw [thick] (0.5,-1.8) -- (7.8,-2);
\draw [thick, dotted] (0.9,-1.6) -- (8.1,-1.8);
\draw [thick] (7.8,-1.79) -- (8.1,-1.8);
\end{tikzpicture}
\end{center}

We show  Lemma \ref{lem:induction=Di} by the induction on $i$.
The following is the first step. 

\begin{lem} \label{lem:fe:DtoD}
There exists a $(1+\tau_{x,1}(\e))$-Lipschitz map
 $f_1:D_{e_1}\to\tilde\blacktriangle y_1y_1'z_1
 \subset \tilde K(x,v)$ 
such that 
\begin{enumerate}
\item $f_1$  sends 
$e_1, C_1, C_1'$ to $\tilde e_1, \tilde C_1, \gamma_{\tilde y_1', \tilde z_1}$ respectively \,$;$
\item $f_1$ is identical on $e_1\,;$
\item $f_1$ {\it respects} $d_x$, $d_{\tilde x}$ on 
$C_1\cup C_1'\,;$ 
\item $f_1(e'_1)= \{ \tilde z_1\}\,;$
\item $f_1(D_{e_1})$ satisfies the condition (5) of 
Lemma \ref{lem:induction=Di}.
\end{enumerate}
\end{lem}
\begin{proof} 
Let $\varphi_0:\pa D_{e_1}\to \pa\tilde D_{e_1}$ be 
a $\tau_x(\e)$-almost isometry sending each edge of $\pa D_{e_1}$  to the corresponding one of $\pa \tilde D_{e_1}$ and satisfying (2), (3).
Let $\varphi_{e_1}:D_{e_1}\to \tilde D_{e_1}$ be a $\tau_x(\e)$-almost isometry with $\varphi_e|_{\pa D_{e_1}}=\varphi_0$ provided in Lemma  \ref
{lem:psi:DtoD}.
We divide $\tilde D_{e_1}$ into $\tilde\blacktriangle y_1y_1'z_1$ and 
the other part by the geodesic $[\tilde y_1',\tilde z_1]$.
Let $\tilde \pi_{e_1}:\tilde D_{e_1} \to\tilde\blacktriangle y_1y_1'z_1$ be the retraction via the projection along the
$d_{\tilde x}$-fibers  
for $t\in [d_{\tilde x}(\tilde z_1),d_{\tilde x}(\tilde y_1')]$.
Namely, $\tilde\pi_{e_1}$ is the identical on 
 $\tilde\blacktriangle y_1y_1'z_1$ and projects
the other part to $[\tilde y_1',\tilde z_1]$ along the $d_{\tilde x}$-fibers.
Note that $\tilde \pi_{e_1}$ is $(1+R_1^{-2})$-Lipschihtz.
Then $f_1:=\tilde\pi_{e_1}\circ\varphi_{e_1}$ is the 
required map, which is $(1+R_1^{-2})(1+\tau_x(\e))$-Lipschitz. The condition (5) of Lemma \ref{lem:induction=Di} can be easily checked.
\end{proof}

\begin{proof}[Proof of Lemma \ref{lem:induction=Di}]
Suppose Lemma \ref{lem:induction=Di} holds
for $i-1$. The basic idea of the construction of 
$f_i$ is similar to that of $f_1$ in Lemma \ref{lem:fe:DtoD}.
We may assume that $e_i$ is adjacent to $e_{i-1}$
with $y_{i-1}'=y_i$.
Let $\tilde w_{i-1}$ be the nearest point of 
$P_{i-1}=f_{i-1}(D_{\rm int}^{i-1})$ from $\tilde x$.
Note that $\tilde w_{i-1}\in \tilde C_1$.
We have to consider the two cases.

\pmed\n
Case 1).\, $d_x(e'_i)\le\min d_x(D_{\rm int}^{i-1})$.
\psmall
Take  $\tilde w_i\in\tilde C_1$ with
$d_{\tilde x}(\tilde w_{i})=d_x(e'_i)$.
Let $\tilde e_i''$ be the arc in $S(\tilde x,d_{\tilde x}(\tilde w_{i}))\cap \tilde x* [\tilde y_1, \tilde y_i']$
starting from $\tilde w_{i}$ such that 
$L(\tilde e_i'')=L(\tilde e_i')$. 
Let 
$\tilde u_i$ be the other endpoint of $e_i''$.
Let $\hat D_{e_i''}$ be the disk domain bounded by
$\tilde e_i$, $\tilde e_i''$ and the curve 
$[\tilde y_i,\tilde w_{i-1}]_{\pa P_{i-1}}\cup [\tilde w_{i-1},
\tilde w_i]$ and the geodesic
$\gamma_{\tilde y_i',\tilde u_i}$,
where $[\tilde y_i,\tilde w_{i-1}]_{\pa P_{i-1}}$ be the arc of $\pa P_{i-1}$ between $\tilde y_i$ and $\tilde w_{i-1}$.
Note that $|\angle_{\tilde w_i}(\hat D_{e_i''})-\pi/2|<
\tau_{i-1}(\e)$. In the present case,  $R_i$ is the largest 
among $\{ R_j\}_{j=1}^i$. It follows that 
\begin{align*}
|\angle_{\tilde y_i'}(\hat D_{e_i''}) -
\angle_{\tilde y_i'}(\tilde D_{e_i})|
&\le 2(a_1+\cdots +a_i)/R_i a_i \\
&\le 2(N_v-2)/R_i=\tau_i.
\end{align*}
Similarly we get
\begin{align*}
&\text{
$|\angle_{\tilde y_i}(\hat D_{e_i''}) -
\angle_{\tilde y_i}(\tilde D_{e_i})|
\le \tau_{i}$, \quad $|\angle_{\tilde u_i}(\hat D_{e_i''}) -
\pi/2|
\le \tau_{i}$,}\\
&\text{
$|\angle_{\tilde w_{i-1}}(\hat D_{e_i''}) -
\pi|
\le \tau_{i-1}$.}
\end{align*}
The inner angles at the other vertices  are also $\tau_{i-1}$-almost  equal to $\pi$
by the induction.
Since $\hat D_{e_i''}$ is isometrically embedded in $M^2_\kappa$, and the 
number of vertices of  $\hat D_{e_i''}$
is uniformly bounded,
$\hat D_{e_i''}$ is $\tau_i(\e)$-almost isometric to a quadrangle region $Z_i$ that is almost rectangular.

On the other hand, it follows from 
\eqref{eq:angle=tilde(zyw)} that 
the Gromov-Hausdorff distance
$d_{GH}(D_{e_i},Z_i)$ (not the pointed one) is less than $\tau_x(e)$
under $1/L(e_i)$-rescaling.
Therefore we have locally defined $\tau_x(e)$-almost isometries between corresponding
polygonal domains in $D_{e_i}$ $Z_i$, which are almost rectangular.
Using Lemma \ref{lem:psi:DtoD} together with Remark \ref{rem:almost-iso},
we can glue those locally defined $\tau_x(e)$-almost isometries to obtain  a globally defined 
$\tau_x(e)$-almost isometry
between $D_{e_i}$ and $Z_i$.
Thus $D_{e_i}$ is $\tau_{x,i-1}(\e)$-almost isometric to $\hat D_{e_i''}$.

Let $\varphi_i:\pa D_{e_i}\to \pa\hat D_{e_i''}$ be 
identical on $e_i$ and respect $d_x$ and $d_{\tilde x}$ on $C_i\cup C_i'$.
We may assume that $\varphi_i$ is 
$\tau_{x,i}(\e)$-almost isometric.
Then in a way similar to Lemma \ref{lem:isom=DtoD},
we obtain a $\tau_{x,i-1}$-almost isometry
$\psi_{e_i}:D_{e_i}\to \hat D_{e_i''}$
extending $\varphi_i$.
Let  $\Delta_i$ be the closure of 
the complement of 
$\blacktriangle \tilde y_i'\tilde w_i\tilde u_i$
in $\tilde D_{e''_i}$.
Let $\hat\pi_{e_i}$ be the projection from $\hat D_{e_i''}$ to $\Delta_i$
along $d_{\tilde x}$-fibers,
which is $\tau(1/R_i)$-Lipschitz.
Combining $\hat\pi_{e_i}\circ \psi_{e_i}$ and $f_{i-1}$,
we have the required $(1+\tau_{x,i}(\e))$-Lipschitz map $f_i:D_{\rm int}^i\to P_i$, where 
$P_i=P_{i-1}\cup \Delta_i$.

\begin{center}
\begin{tikzpicture}
[scale = 1]
\draw [-, thick] (-4,0) -- (4.5,0);
\draw [-, thick] (4.5,0) -- (4.5,0.8);
\draw [-, thick, dotted] (4.5,0.8) -- (4.3,1.6);
\draw [-, thick] (4.3,1.6) -- (3.9,2.4);
\draw [-, thick] (4.5,0) -- (4.3,-0.8);
\draw [-, thick] (-2.5,0) -- (3.9,2.4);
\draw [-, thick] (-1,0) -- (4.3,1.6);
\draw [-, thick] (1,0) -- (4.5,0.8);
\draw [-, thick] (0,0) -- (4.3,-0.8);
\draw [-, thick] (-2.5,0) -- (-2.6,0.8);
\draw [-, thick] (3.9,2.4) -- (-2.6,0.8);
\fill (4.5,0.8) coordinate (A) circle (2pt) node [right] {$\tilde{y}_1'$};
\fill (4.5,0) coordinate (A) circle (2pt) node [right] {$\tilde{y}_1$};
\fill (4.3,1.6) coordinate (A) circle (2pt) node [right] {$\tilde{y}_i = \tilde{y}_{i-1}'$};
\fill (3.9,2.4) coordinate (A) circle (2pt) node [above] {$\tilde{y}_i'$};
\fill (4.35,2.1) coordinate (A) circle (0pt) node [] {$\tilde{e}_i$};
\fill (1,0) coordinate (A) circle (2pt) node [above] {$\tilde{z}_1$};
\fill (-1,0) coordinate (A) circle (2pt) node [below] {$\tilde{w}_{i-1}$};
\fill (-2.5,0) coordinate (A) circle (2pt) node [below] {$\tilde{w}_i$};
\fill (-2.6,0.8) coordinate (A) circle (2pt) node [above] {$\tilde{u}_i$};
\fill (-2.85,0.4) coordinate (A) circle (0pt) node [] {$\tilde{e}_i''$};
\fill (-4,0) coordinate (A) circle (0pt) node [above] {$\tilde{C}_1$};
\fill (1.6,1.16) coordinate (A) circle (0pt) node [] {$\hat{D}_{e_i''}$};
\filldraw [fill=gray, opacity=.1]
(-2.5,0) -- (3.9,2.4) -- (4.3,1.6) -- (4.5,0.8) -- (4.5,0) -- (4.3,-0.8) -- (0,0) -- (-2.5,0);
\fill (0,-0.3) coordinate (A) circle (0pt) node [below] {$P_i$};
\end{tikzpicture}
\end{center}

\pmed\n
Case 2).\,  $d_x(e'_i) > \min d_x(D_{\rm int}^{i-1})$.
\psmall\n

Let $\tilde w_i\in [\tilde y_i, \tilde w_{i-1}]_{\pa P_{i-1}}$ be such that
$d_{\tilde x}(\tilde w_i)=d_x(e_i')$.
Take the arc $\tilde e_i''$ in $S(\tilde x,d_{\tilde x}(\tilde w_{i}))\cap \tilde x * [\tilde y_1,\tilde y_{i}']$ with $\{\tilde w_{i}, \tilde u_i\} :=\pa \tilde e_i''$ 
and $L(\tilde e_i'')=L(\tilde e_i')$
as in Case 1).
Let $\tilde D_{e_i''}$ be the disk domain 
in $\tilde x * [\tilde y_1,\tilde y_{i}']$
bounded by $\tilde e_i$, $[\tilde y_i, \tilde w_{i-1}]_{\pa P_{i-1}}$, $\tilde e_i''$  and 
$\gamma_{\tilde y_i', u_i}$.
In a way similar to Case 1), we obtain a
$(1+\tau_{x,i}(\e))$-Lipschitz map $D_{e_i}\to \blacktriangle \tilde y_i\tilde y_i' \tilde w_i$
which is identical on $e_i$ and respects $d_x$ and $d_{\tilde x}$ on $C_i$ and $C_i'$. 
Combining this with $f_{i-1}$, we 
obtain the required $f_i$.

In both cases, the map $f_i$ satisfies all the conclusions of 
Lemma \ref{lem:induction=Di}.
This completes the proof of Lemma \ref{lem:induction=Di}.
\end{proof}

\begin{rem}\upshape
Note that in each step of the inductive argument above, the estimates
$\tau_{x,i}(\e)$ becomes worse.
However the number of $E_{\rm int}(T)$ is not greater
than $N_v-2$. Therefore finally we have 
a proper estimates.
\end{rem}
\pmed
Let $f_{\rm int}:=f_k:D_{\rm int}\to \tilde K(x,v)$, and let  $W_*$ and $\tilde W_*$  be the cores of $K(x,v)$ and $\tilde K(x,v)$ respectively.
From construction, we have
 $D_{\rm int}\subset W_*$, and $f_{\rm int}(D_{\rm int})\subset \tilde W_*$.
 We extend $f_{\rm int}$ to the map $f_*:W_*\to \tilde W_*$ as follows.

Let $C_1'$ be a singular curve from $y_1$ to $x$
extending $C_1$.
Set $D_{\rm int,1}:=D_{\rm int}\cup C_1'$, and define the extension $f_{\rm int,1}:D_{\rm int,1}\to f_{\rm int}(D_{\rm int})\cup\tilde C_1$ by the property of respecting $d_x$ and $d_{\tilde x}$
on $C_1'$.
For any $y\in W_*\setminus D_{\rm int,1}$,
let us denote by $G(y)$ the subtree of $W_*\cap S(x,d_x(y))$
consisting of the component of 
$S(x,d_x(y))\cap W_*\setminus D_{\rm int,1}$ containing $y$.
Taking any $z\in \pa G(y)\cap D_{\rm int,1}$, we define
\begin{align}\label{eq:def=f1}
      f_*(y):= f_{\rm int,1}(z).
\end{align}
We have to verify that the map $f_*:W_*\to 
f_{\rm int}(D_{\rm int})\cup\tilde C_1$ defined by 
\eqref{eq:def=f1} is well-defined.
This is obvious on $W_*\cap B(x, d_x(x_*))$, where
$x_*$ is the nearest point of $D_{\rm int}$ from $x$.
By Lemma \ref{lem:no-swinging}, this is also obvious on $T_*(y)$ (see \eqref{eq:T(y,t)}) 
for all $y\in V_{\rm int}(T)$.
Consider the general case $d_x(y)>d_x(x_*)$.
We may assume that 
$\pa G(y)\cap D_{\rm int}$ contains distinct two points
$z_1,z_2$.
Take $e_{i_\alpha}\in E_{\rm int}(T)$ such that 
$z_\alpha\in\pa D_{e_{i_\alpha}}$ for $\alpha=1,2$.
Let $c$ be the segment in $T$ joining $e_{i_1}$
and $e_{i_2}$.
By Lemma \ref{lem:monotone-De},
$S(x,d_x(y))\cap D_{e}$ is empty for any other edge $e$ contained 
in $c$.
Then obviously, $f_{\rm int}(D_e)$ does not affect
$S(x,d_x(y))$. The construction of $f_{\rm int}$
then implies that $f_{\rm int}(z_1)=f_{\rm int}(z_2)$.
Thus  $f_*$ is well-defined.
\psmall
 
Finally we consider any exterior edge $e\in E_{\rm ext}(T)$.
Let  $e=[y,z]$  with $z\in\pa T$, and take the singular curve $C_e$ contained in $\pa W_*$ joining
$x$ and $y$ such that $z*C_e\setminus C_e$ does not meet
$\ca S(X)$.
Let $\blacktriangle_*xyz$ be the disk
domain of $K(x,v)$ bounded by
$e$, $\gamma_{x,z}$ and $C_e$.
Observe that $f_*$ is defined on $C_e$.
Let $\tilde\blacktriangle_*xyz:=\tilde z*f_*(C_e)$,
which is topologically a disk, where
$\tilde z\in\pa\tilde T$ is the point corresponding to $z$.
Now, define the $\tau_{x,k}(\e)$-almost isometry
$\varphi_0:\pa \blacktriangle_*xyz\to
\pa \tilde\blacktriangle_*xyz$ by the conditions that 
it is identical on $e\cup\gamma_{x,z}$,
and coincides with $f_*$ on $C_e$.
In a way similar to Lemma \ref{lem:isom=DtoD},
using rescalings at $x$ and $\tilde x$, we can show 
that $\blacktriangle_*xyz$ is $\tau_x(\e)$-almost
isometric to $\tilde\blacktriangle_*xyz$.
Therefore in a way similar to Lemma 
\ref{lem:psi:DtoD}, we obtain a 
$\tau_{x,k}(\e)$-almost isometry 
$\psi_e:\blacktriangle_*xyz\to\tilde\blacktriangle_*xyz$
extending $\varphi_0$.

Thus, combining $f_*$ and $\psi_e$ for all $e\in E_{\rm ext}(T)$, we have the desired
$(1+\tau_{x,k}(\e))$-Lipschitz map
$f:K(x,v)\to \tilde K(x,v)$.

\begin{rem}\label{rem:choice=onstant} \upshape
Here is a remark on the estimates
$\tau_{x,k}(\e)$. For given $n$, 
we  take large enough $R_1,
\ldots, R_k$ satisfying $\tau_k<1/n$.
This is possible
with $\tau_x(\e)<1/n$ if $\e=\e_x$ is small enough in the  construction of $X_n$.
\end{rem}
\psmall

The construction of $\tilde f:\tilde K(x,v)\to K(x,v)$
is immediate by the following.

\begin{lem}[\cite{Rsh:Inext}] \label{lem:tilde-f}
Let $\triangle xyz$ be a geodesic triangle in a 
$\CAT(\kappa)$-space of general dimension. 
Let $\blacktriangle xyz$
denote the join $x*yz$.
Then the identical map $\varphi_0:\tilde\triangle xyz
\to \triangle xyz$ extends to a $1$-Lipschitz map
$\varphi:\tilde\blacktriangle xyz\to \blacktriangle xyz$.
\end{lem}

By Lemma \ref{lem:tilde-f}, for each $\tilde e\in E(\tilde T)$,
we have a $1$-Lipschitz
map $\tilde f_e:\tilde\blacktriangle (x,e)\to\blacktriangle (x,e)$
that is identical on $\pa\tilde\blacktriangle (x,e)$.
Combining $\tilde f_e$ for all $\tilde e\in E(\tilde T)$,
we obtain a $1$-Lipschitz map
$\tilde f:\tilde K(x,v)\to K(x,v)$ which is identical on
$\tilde K(x,v)$.
This completes the proof of Theorem  \ref{thm:approx}.
\end{proof}

\setcounter{equation}{0}

\psmall
\section{Gauss-Bonnet Theorem} \label{sec:GB}
\psmall \upshape
\n
{\bf Preliminary argument}.\,
Let $\triangle$ be a geodesic triangle bounding a region 
homeomorphic to a disk, in a two-dimensional
metric space $X$ with certain curvature condition.
The excess $\delta(\triangle)$ of $\triangle$ is defined as
\[
  \delta(\triangle):=\alpha+\beta+\gamma - \pi,
\]
where $\alpha$, $\beta$ and $\gamma$ denote the inner
angles of  $\triangle$.
In the case when $X$ is a topological two-manifold of bounded curvature in the sense of \cite{AZ:bddcurv}, 
Alexandrov and Zalgeller \cite{AZ:bddcurv} extended it to a signed Borel regular measure $\omega_{AZ}$ on $X$ called the curvature measure, 
and obtained Gauss-Bonnet Theorem.

The following is a basic relation between 
the turn and the curvature measure.

\begin{prop}$($\cite[Theorem VI.3]{AZ:bddcurv}$)$
\label{prop:tau+tau}
For any open simple arc $e$ in  a topological two-manifold  $X$ of bounded curvature, we have 
\[
\tau_{F_+}(e) +\tau_{F_-}(e)=\omega_{AZ}(e).
\]
\end{prop}

For a polyhedral space $X$ with curvature bounded above, 
Burago-Buyalo \cite{BurBuy:upperII} (cf.\cite {ArsBuy})
obtained a natural curvature measure $\omega_{BB}$ 
extending that in \cite{AZ:bddcurv}.

As before, let $X$ be a two-dimensional locally compact  geodesically complete metric space with 
curvature $\le\kappa$. 
 We define the curvature measure of $X$ in 
Definition \ref{defn:curv-meas2}.
We begin with the following temporary one.

\begin{defn} \label{defn:omega}\upshape
We define a ``curvature measure''
$\omega$ on $X$ extending those of \cite{AZ:bddcurv} and \cite{BurBuy:upperII}.
\psmall\n
(1)\, First we define 
$\omega:=\omega_{AZ}$ on the surface $X\setminus\ca S(X)$.
In particular, for $x\in X\setminus\ca S(X)$, we have
$\omega(x)=2\pi - L(\Sigma_x(X))$
(see \cite[Theorem V.8]{AZ:bddcurv}).
\psmall\n
(2)\, 
Let $e$ be an open edge in $E(\ca S(X))$.
Then a neighborhood $U$ of $e$ in $X$ is a finite 
union of open half-disks $H_i$\,$(1\le i\le m(e))$ bounding $e$:
$$
      U=\bigcup_{i=1}^{m(e)} H_i, \quad \pa H_i\cap U=e.
$$
Then  we have 
\[
        \omega|_e = \sum_{i=1}^{m(e)}\, \tau_{H_i},
\]
where $\tau_{H_i}$ is the turn of $e$ from the side $H_i$
(see \cite{BurBuy:upperII} for the polyhedral case).

\psmall\n
(3)\, 
Finally, for  $x\in V(\ca S(X))$, we define
\[
    \omega(x):=\pi(2-\chi(\Sigma_x(X)))-L(\Sigma_x(X)),
\]
which is a traditional definition, and represents the curvature concentration
at $x$ (see \cite{ArsBuy}, \cite{BurBuy:upperII}).
\end{defn}

\begin{rem} \label{rem:Omega-Edge}\upshape
For any Borel subset $A\subset V(\ca S(X))$,
we let 
\beq \label{eq:curv.measure-welldef}
     \omega'(A) := \sum_{x\in A}\omega(x),
\eeq
as the sum of point measures $\omega(x)$ over $x\in A$.
Actually, we can show that the sum in the RHS of \eqref{eq:curv.measure-welldef} is  at most countable 
and converges if $A$ is bounded 
(see Proposition \ref{prop:pt-meas}).

For the curvature measure $\omega$ defined in 
Definition \ref{defn:curv-meas2}, we certainly have 
$\omega'(A)=\omega(A)$ when $A$ is 
 a countable bounded subset of $V(\ca S(X))$. However, as Example \ref{ex:vertex=positivem} shows, in the case when $V_*(\ca S(X))$ is 
 uncountable (see \cite[Example 6.9]{NSY:local}), 
\eqref{eq:curv.measure-welldef} is not the right definition 
for the curvature measure of $A$.
\end{rem}

\begin{ex} $($cf. \cite[Example 6.9]{NSY:local}$)$\label{ex:vertex=positivem} \upshape
For any $0<\e<1$, let $f:\R\to[0,1]$ be a smooth 
function defined in \cite[Example 6.9]{NSY:local}.
Note that $f^{-1}(0)\cap [0,1]$ is an 
$\e$-Cantor set, say $I_\infty$. Let 
$\Omega=\{|y|\le f(x)\}\subset\R^2$ and take
six copies $H_{\pm}^i$\,$(1\le i\le 3)$ of 
a concave domain of  $\R^2$ as in \cite[Example 6.9]{NSY:local}.
Let $X$ be the result of the  gluings
of $\Omega$ and $H_{\pm}^i$\,$(1\le i\le 3)$
along
$\Omega_{\pm}=\{ y=\pm f(x)\}$  and 
$\pa H_{\pm}^i$\,$(i=1,2)$ and along $\{ y=0\}\subset\Omega$ and 
$\pa H_{\pm}^3$.
Then $X$ is a two-dimensional locally compact,
geodesically complete $\CAT(0)$-space with 
$V_*(\ca S(X))=V(\ca S(X))=I_\infty$.
Note that $\omega(x)=0$ for all $x\in V(\ca S(X))$,
and therefore we have $\omega'(V(\ca S(X)))=0$.
However, from the gluing of
 $\{ y=0\}$ and $\pa H_{\pm}^3$,
it is reasonable to expect the curvature of any point of
$\{ 0\le y\le 1\}\subset X$ is negative, and therefore the {\it real} curvature of 
$V(\ca S(X))$ should be negative. 
\end{ex}
\pmed
\subsection{Convergence of curvature measures} 
\label{ssec:conv}
\psmall

Example \ref{ex:vertex=positivem} suggests
that it is complicated to define 
the curvature measure on $V_*(\ca S(X))$
directly and explicitly (see Theorem \ref{thm:explicit}).
Therefore we go another indirect way, and 
discuss the convergence of $\omega^{X_n}_{BB}$.

We begin with 
\begin{defn} \label{defn:admissible-D}\upshape
Let $D$ be an open bounded domain in $X$, and let
$\pa D$ be the topological boundary of $D$ in $X$.
We call $D$ {\it admissible} if it satisfies
\begin{itemize}
\item[(1)] $\pa D\setminus \ca S(X)$ consists of finitely many simple arcs $L_i$\, $(1\le i\le \ell)$ and simple closed curves $S_j$\,$(1\le j\le m)\,;$ 
\item[(2)] for any $x\in \pa D\cap \ca S(X)$, if $x\in L_i$, 
then $\bar L_i$ has a definite direction 
$\Sigma_x(\bar L_i)$ at $x$ and 
satisfies the transversality to $\ca S(X)$ in the sense that
\[
                   \angle_x(\Sigma_x(\bar L_i), \Sigma_x(\ca S(X))) >0 \,;
\]
\item[(3)] $L(\Sigma_x(\bar D))>0$ for all $x\in \pa D\cap \ca S(X)$.
\end{itemize}
\end{defn}
\psmall 

We show that any compact  set of $X$ 
can be approximated by an admissible domain.

\begin{lem} \label{lem:approx-admissible}
For any compact subset $A$ of $X$ and $\e>0$, there exists
an admissible domain $D$ such that 
$A\subset D\subset B(A,\e)$.
\end{lem}
\begin{proof}
For each $x\in \pa A$, choose $\e_x\le\e$
such that $B(x,\e_x)$ satisfies the 
results in Subsection \ref{ssec:lcal-str}
if $x\in\ca S(X)$,  and is convex and does not touch 
$\ca S(X)$ if $x\notin \ca S(X)$.
Let $U$ be the union of a locally finite 
covering of $\pa A$ by such metric balls,
and let $D$ be the union of $A$ and $U$.
Then  (1) and (2) are immediate.
(3) is also clear since 
$L(\Sigma_x(\bar D))\ge\pi$ for all $x\in \pa D\cap \ca S(X)$.
\end{proof}

For an admissible domain $D$ of $X$, let
$\{ q_1,\ldots q_n\}:= \pa D\cap \ca S(X)$, and 
let $L_1,\ldots, L_{\ell}$ (resp. $S_1,\ldots, S_m$) be the arc-components (resp. the circle components) of 
$\pa D \setminus \ca S(X)$. 
\cite[Theorem VI.2]{AZ:bddcurv} implies that 
the turns $\tau_D(L_i)$ and $\tau_D(S_j)$ are well-defined for all  $i$ and $j$.

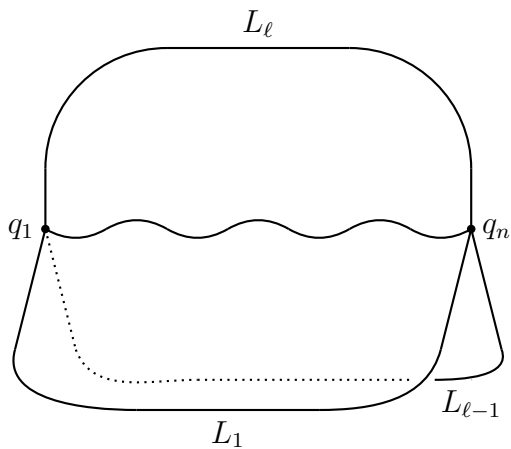
\begin{center}
\begin{tikzpicture}
[scale = 0.8]
\draw [-, thick] (3.5,0) to [out=210, in=330] (2.5,0);
\draw [-, thick] (2.5,0) to [out=150, in=30] (1.5,0);
\draw [-, thick] (1.5,0) to [out=210, in=330] (0.5,0);
\draw [-, thick] (0.5,0) to [out=150, in=30] (-0.5,0);
\draw [-, thick] (-0.5,0) to [out=210, in=330] (-1.5,0);
\draw [-, thick] (-1.5,0) to [out=150, in=30] (-2.5,0);
\draw [-, thick] (-2.5,0) to [out=210, in=330] (-3.5,0);
\draw [-, thick] (-3.5,0) -- (-3.5,1);
\draw [-, thick] (3.5,0) -- (3.5,1);
\draw [-, thick] (-1.5,3) -- (1.5,3);
\draw [-, thick] (-3.5,1) to [out=90, in=180] (-1.5,3);
\draw [-, thick] (3.5,1) to [out=90, in=0] (1.5,3);
\draw [-, thick] (-3.5,0) -- (-4,-2);
\draw [-, thick] (3.5,0) -- (3,-2);
\draw [-, thick] (-4,-2) to [out=255, in=180] (-2,-3);
\draw [-, thick] (-2,-3) -- (1,-3);
\draw [-, thick] (1,-3) to [out=0, in=255] (3,-2);
\draw [-, dotted, thick] (-3.5,0) -- (-3,-2);
\draw [-, thick] (3.5,0) -- (4,-2);
\draw [-, dotted, thick] (-1,-2.5) -- (2.5,-2.5);
\draw [-, dotted, thick] (-3,-2) to [out=290, in=180] (-1,-2.5);
\draw [-, thick] (4,-2) to [out=290, in=0] (2.9,-2.5);
\fill (-3.5,0) coordinate (A) circle (2pt) node [left] {$q_1$};
\fill (3.5,0) coordinate (A) circle (2pt) node [right] {$q_n$};
\fill (0,3) coordinate (A) circle (0pt) node [above] {$L_{\ell}$};
\fill (-0.5,-3) coordinate (A) circle (0pt) node [below] {$L_1$};
\fill (3.5,-2.5) coordinate (A) circle (0pt) node [below] {$L_{\ell-1}$};
\end{tikzpicture}
\end{center}

\begin{defn}\label{defn:tauD(paD)}\upshape
We now define the {\it turn} of $\pa D$ from the side $D$
by
\begin{align}\label{eq:defn-tauD(paD)}
\tau_D(\pa D):=\sum_{i=1}^{\ell}\tau_D(L_i) +
                       \sum_{j=1}^{m}\tau_D(S_j) + 
                         \sum_{k=1}^n \omega_D(q_k),
\end{align}
where 
\[
   \omega_D(q_k) :=\pi(2-\chi(\Sigma_{q_k}(D))-L(\Sigma_{q_k}(D))
\]
is the {\it outer angle} of $D$ at $q_k$.
\end{defn}


\pmed
\n
{\bf Convergence of curvature measures}.\,

Let $(X_n,p_n)$ be a sequence of two-dimensional
polyhedral locally $\CAT(\kappa)$-spaces 
as in Theorem \ref{thm:approx}, which 
converges to  $(X, p)$ 
with respect to the homotopy convergence.

\begin{lem}\label{lem:constDn}
For any admissible domain $D$ in $X$, 
there is an admissible domain $D_n$ in $X_n$
such that 
\begin{enumerate}
\item  $D_n$ converges to $D$ under the convergence
$(X_n,  p_n)\to (X,p)\,;$
\item $D_n$ has the same homotopy type as $D$
for any large enough $n\,;$
\item $\lim_{n\to\infty}\tau_{D_n}(\pa D_n)=\tau_D(\pa D)$.
\end{enumerate}
\end{lem}
\begin{proof} 
Let $L_1,\ldots, L_{\ell}, S_1,\ldots, S_m\subset\pa D$ and $\{ q_1,\ldots, q_k\} = \pa D\cap\ca S(X)$ 
be as before. From the conditions (2),(3) in 
Definition \ref{defn:admissible-D}, as discussed in \cite[VI.3]{AZ:bddcurv},
we can approximate each $L_i$ by 
broken geodesics $L_{i,\mu}$ in ${\rm int}\,D$ (except
the endpoints)  joining 
the endpoints $a_i, b_i$ of $L_i$ such that
\[
\text{
$\lim_{\mu\to\infty}\angle_{a_i}(L_i, L_{i,\mu})
=\lim_{\mu\to\infty}\angle_{b_i}(L_i, L_{i,\mu})=0$.}
\]
Let $D_{\mu}$ be the domain of $X$ bounded by 
$L_{1,\mu}, \ldots, L_{\ell,\mu}, S_1,\ldots, S_m$.
Since 
$\lim_{\mu\to\infty} \tau_{D_\mu}(\pa D_\mu)=
\tau_D(\pa D)$ (\cite[Theorem VI.2]{AZ:bddcurv}), 
from the beginning, we may assume that 
\begin{align}\label{eq:paD=broken}
\text{ each $L_i$ is a broken geodesic.}
\hspace{4cm}
\end{align}

If all the $q_i$ are not in the interior of the surgery part in the construction of $X_n$ for any large $n$, then $\pa D$ is 
embedded also in $X_n$. Therefore, we let $D_n$ be the open domain  of $X_n$ bounded by 
$\pa D$, and 
we have all the required properties for $D_n$.

Suppose that some 
$q_i$ is in the interior of the surgery part of $X$, 
say $q_i\in \mathring{K}(x_{\alpha_n},v_{\alpha_n})$,
 in the construction of $X_n$ for any large $n$.
As in Section \ref{sec:Polyhedral}, we write $K(x_{\alpha_n},v_{\alpha_n})$ as the join 
$x_{\alpha_n}* T_{\alpha_n}$ via a geodesic tree 
$T_{\alpha_n}\subset \pa K(x_{\alpha_n},v_{\alpha_n})$. 
For simplicity, we set $q:=q_i$, and 
remove the subscript $\alpha$ like 
$x_n:= x_{\alpha_n}$, $T_n:=T_{\alpha_n}$,
$K(x_{n},v_{n}):=K(x_{\alpha_n},v_{\alpha_n})$.

Let $\lambda_i$ \,$(1\le i\le I)$ be the components
of $(\pa D\setminus\ca S(X))\cap B(q,c)$,
where $c$ is a small enough fixed constant
so as to satisfy 
\begin{itemize}
\item $\pa D\cap B(q,c)\cap\ca S(X) =\{ q\}\,;$
\item $\bar\lambda_i$ are geodesic segments from
$q$.
\end{itemize}

In what follows, 
we assume 
$n$ is large enough so that each 
$\lambda_i$ \,$(1\le i\le I)$ goes through $K(x_n,v_{n})$.
By the transversality condition of $\pa D$ with $\ca S(X)$, 
we assume each $\lambda_i$ transversally 
 meets  an edge of 
$\pa K(x_{n},v_{n})\setminus T_{n}$,
say at $z_i$. 
The case when $\lambda_i$ meets $T_{n}$ 
is similarly discussed.
Let $\{ u_1,\ldots,u_I\}:=\pa T_n$, and 
let $z_i$ be the intersection point of $\pa D$ with 
$\gamma_{x_n,u_i}$.
From the assumption, $K(x_{n},v_{n})\cap \pa D$
is a tree, say $T_{n}(\pa D)$, 
which is isomorphic to the cone over $\{ z_1, \ldots, z_I \}$.

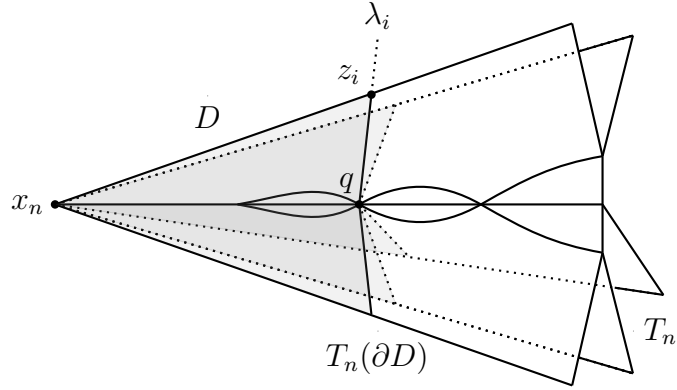
\begin{center}
\begin{tikzpicture}
[scale = 0.8]
\draw [-, thick] (4.5,0.8) to [out=190, in=30] (2.5,0);
\draw [-, thick] (2.5,0) to [out=210, in=330] (0.5,0);
\draw [-, thick] (0.5,0) to [out=150, in=10] (-1.5,0);
\draw [-, thick] (4.5,-0.8) to [out=170, in=330] (2.5,0);
\draw [-, thick] (2.5,0) to [out=150, in=30] (0.5,0);
\draw [-, thick] (0.5,0) to [out=210, in=350] (-1.5,0);
\fill (-4.5,0) coordinate (A) circle (2pt) node [left] {$x_n$};
\draw [thick] (4.5,0.8) -- (4.5,-0.8);
\draw [thick] (4.5,0.8) -- (4,3);
\draw [thick] (4.5,-0.8) -- (4,-3);
\draw [thick] (4.5,0.8) -- (5,2.8);
\draw [thick] (4.5,-0.8) -- (5,-2.8);
\draw [thick] (-4.5,0) -- (4.5,0);
\draw [thick] (-4.5,0) -- (4,3);
\draw [thick] (-4.5,0) -- (4,-3);
\draw [thick] (4.5,0) -- (5.5,-1.5);
\draw [thick] (4.7,-1.38) -- (5.5,-1.5);
\draw [thick, dotted] (-4.5,0) -- (5.5,-1.5);
\draw [thick, dotted] (-4.5,0) -- (5,2.8);
\draw [thick, dotted] (-4.5,-0) -- (5,-2.8);
\draw [thick, dotted] (-4.5,0) -- (5,-2.8);
\draw [thick] (4.1,-2.535) -- (5,-2.8);
\draw [thick, dotted] (-4.5,0) -- (5,2.8);
\draw [thick] (4.1,2.535) -- (5,2.8);
\draw [thick] (0.5,0) -- (0.7,1.82);
\draw [thick] (0.5,0) -- (0.7,-1.82);
\draw [thick, dotted] (0.8,2.73) -- (0.7,1.82);
\draw [thick, dotted] (0.5,0) -- (1.1,1.65);
\draw [thick, dotted] (0.5,0) -- (1.1,-1.65);
\draw [thick, dotted] (0.5,0) -- (1.3,-0.85);
\filldraw[fill=gray, opacity=.1] 
(-4.5,0) -- (0.7,1.82) -- (0.5,0) -- (-4.5,0);
\filldraw[fill=gray, opacity=.1] 
(-4.5,0) -- (0.7,-1.82) -- (0.5,0) -- (-4.5,0);
\filldraw[fill=gray, opacity=.1] 
(-4.5,0) -- (1.1,1.65) -- (0.5,0) -- (-4.5,0);
\filldraw[fill=gray, opacity=.1] 
(-4.5,0) -- (1.1,-1.65) -- (0.5,0) -- (-4.5,0);
\filldraw[fill=gray, opacity=.1] 
(-4.5,0) -- (1.3,-0.85) -- (0.5,0) -- (-4.5,0);
\fill (0.5,0) coordinate (A) circle (2pt) node [above left] {};
\fill (0.6,0) coordinate (A) circle (0pt) node [above left] {$q$};
\fill (0.7,1.82) coordinate (A) circle (2pt) node [above left] {$z_i$};
\fill (0.8,2.73) coordinate (A) circle (0pt) node [above] {$\lambda_i$};
\fill (5,-2.1) coordinate (A) circle (0pt) node [right] {$T_n$};
\fill (0.8,-2.1) coordinate (A) circle (0pt) node [below] {$T_n(\partial D)$};
\fill (-2,1.8) coordinate (A) circle (0pt) node [below] {$D$};
\end{tikzpicture}
\end{center}

Choose any interior vertex $\tilde q_n$ of  
the tree $H_0:=S(\tilde x_{n},  d_{x_{n}}(q))\cap \tilde K(x_{n},v_{n})$. 
Let $H$ be the geodesic tree obtained by straightening
the edges of $H_0$ with $V(H)=V(H_0)$.
Let $\tilde z_i\in\pa  \tilde K(x_{n},v_{n})$\,
$(1\le i\le I)$
be the point corresponding to $z_i$.
Let $\tilde d_i\in E_{\rm ext}(\tilde T_n)$ be 
the edge adjacent to $\tilde u_i$.
Let $\tilde y_i$ be the intersection point
of $\tilde x_n*\tilde d_i$ and $\pa H$. 
%
%
Let  $\tilde T_{n}(\pa D)$ be the union of $H$ and 
the geodesics $\gamma_{\tilde y_i, \tilde z_i}$\,
$(1\le i\le I)$, which is a geodesic tree 
in $\tilde K_*(x_{n},v_{n})$ isomorphic to $T_{n}$.
From the construction of $K(x_{n},v_{n})$
(see \eqref{eq:epsilon-delta}), we have
\begin{align}\label{eq:choice-lambda}
 \lim_{n\to\infty}\angle^{X_n} \tilde y_i \tilde z_i \tilde w_i =\pi,
\end{align}
where $\tilde w_i$ is any point of $\lambda_i\setminus 
K(x_{n},v_{n})$.

Replacing  $T_{n}(\pa D)$ by  $\tilde T_{n}(\pa D)$
in $\pa D$
for all  $q=q_i$,
we obtain  a subset  $\Lambda_n\subset  X_n$ 
consisting of finitely many simple broken geodesic arcs and the simple loops  $S_1,\ldots, S_m$ such that 
$\Lambda_n$ bounds  
an admissible domain $D_n\subset X_n$.
Then from construction, $D_n$ satisfies (1) and (2).

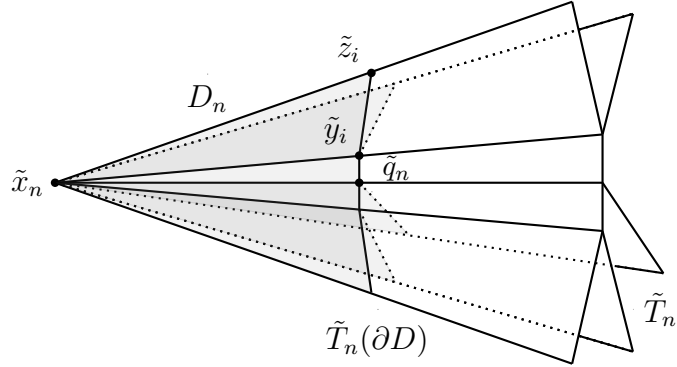
\begin{center}
\begin{tikzpicture}
[scale = 0.8]
\fill (-4.5,0) coordinate (A) circle (2pt) node [left] {$\tilde{x}_n$};
\draw [thick] (4.5,0.8) -- (-4.5,0);
\draw [thick] (4.5,-0.8) -- (-4.5,0);
\draw [thick] (4.5,0.8) -- (4.5,-0.8);
\draw [thick] (4.5,0.8) -- (4,3);
\draw [thick] (4.5,-0.8) -- (4,-3);
\draw [thick] (4.5,0.8) -- (5,2.8);
\draw [thick] (4.5,-0.8) -- (5,-2.8);
\draw [thick] (-4.5,0) -- (4.5,0);
\draw [thick] (-4.5,0) -- (4,3);
\draw [thick] (-4.5,0) -- (4,-3);
\draw [thick] (4.5,0) -- (5.5,-1.5);
\draw [thick] (4.7,-1.38) -- (5.5,-1.5);
\draw [thick, dotted] (-4.5,0) -- (5.5,-1.5);
\draw [thick, dotted] (-4.5,0) -- (5,2.8);
\draw [thick, dotted] (-4.5,-0) -- (5,-2.8);
\draw [thick, dotted] (-4.5,0) -- (5,-2.8);
\draw [thick] (4.1,-2.535) -- (5,-2.8);
\draw [thick, dotted] (-4.5,0) -- (5,2.8);
\draw [thick] (4.1,2.535) -- (5,2.8);
\draw [thick] (0.5,-0.45) -- (0.5,0.45);
\draw [thick] (0.5,-0.45) -- (0.7,-1.82);
\draw [thick] (0.5,0.45) -- (0.7,1.82);
\draw [thick, dotted] (0.5,0.45) -- (1.1,1.65);
\draw [thick, dotted] (0.5,-0.45) -- (1.1,-1.65);
\draw [thick, dotted] (0.5,0) -- (1.3,-0.85);
\filldraw[fill=gray, opacity=.1] 
(-4.5,0) -- (0.7,1.82) -- (0.5,0.45) -- (-4.5,0);
\filldraw[fill=gray, opacity=.1] 
(-4.5,0) -- (0.7,-1.82) -- (0.5,-0.45) -- (-4.5,0);
\filldraw[fill=gray, opacity=.1] 
(-4.5,0) -- (1.1,1.65) -- (0.5,0.45) -- (-4.5,0);
\filldraw[fill=gray, opacity=.1] 
(-4.5,0) -- (1.1,-1.65) -- (0.5,-0.45) -- (-4.5,0);
\filldraw[fill=gray, opacity=.1] 
(-4.5,0) -- (1.3,-0.85) -- (0.5,0) -- (-4.5,0);
\filldraw[fill=gray, opacity=.1] 
(-4.5,0) -- (0.5,-0.45) -- (0.5,0.45) -- (-4.5,0);
\fill (0.5,0) coordinate (A) circle (2pt) node [above left] {};
\fill (0.7,-0.15) coordinate (A) circle (0pt) node [above right] {{\small $\tilde{q}_n$}};
\fill (0.7,1.82) coordinate (A) circle (2pt) node [above left] {$\tilde{z}_i$};
\fill (0.5,0.45) coordinate (A) circle (2pt) node [above left] {$\tilde{y}_i$};
\fill (5,-2.1) coordinate (A) circle (0pt) node [right] {$\tilde{T}_n$};
\fill (0.8,-2.1) coordinate (A) circle (0pt) node [below] {$\tilde{T}_n(\partial D)$};
\fill (-2,1.8) coordinate (A) circle (0pt) node [below] {$D_n$};
\end{tikzpicture}
\end{center}

To show (3),  we assume $\{ q\}=\pa D\cap \ca S(X)$
for simplicity. 
Let $\lambda_i^n:=\lambda_i\cap K(x_{n},v_{n})\in E(T_n(\pa D))$,  and let
$\tilde \lambda_i^n:=\gamma_{\tilde y_i,\tilde z_i}\in E_{\rm ext}(\tilde T_{n}(\pa D))$.
 Let $\{ \tilde e_{j}\}_{j=1}^J:=E_{\rm int}(\tilde T_{n}(\pa D))$ and 
$\{ \tilde v_{k}\}_{k=1}^K:=V(\tilde T_{n}(\pa D))\setminus
\pa \tilde T_{n}(\pa D)$.
Note that  $\tilde q_n\in\{ \tilde v_{k}\}_{k=1}^K$.
From \eqref{eq:paD=broken}, we have 
\begin{align*}
&|\tau_D(\pa D) -\tau_{D_n}(\pa D_n)| \\
  & =\biggl|\sum_{i=1}^I \tau_{D}(\lambda_i^n) +\omega(q) -\sum_{i=1}^I \omega_{D_n}(\tilde z_i)
   -\sum_{j=1}^J \tau_{D_n}(\tilde e_{j})
   -\sum_{k=1}^K \omega_{D_n}(\tilde v_{k})\biggr|.
 \end{align*}
It follows from \eqref{eq:choice-lambda} that
$\lim_{n\to\infty}\sum_{i=1}^I \omega_{D_n}(\tilde z_i)
 =0$.
Note that $\tau_{D_n}(\tilde e_{j})=0$ and 
$\lim_{n\to\infty}\tau_{D}(\lambda_i^n)=0$ 
since $\bigcap_n \lambda_i^n$ is empty.
Thus we obtain
\begin{align} \label{eq:tau(partialD)-n}
\lim_{n\to\infty} |\tau_D(\pa D) -\tau_{D_n}(\pa D_n)|
   = \lim_{n\to\infty} |\omega(q)
      -\sum_{k=1}^K \omega_{D_n}(v_{k})|.
 \end{align}
Set $\theta_i:=\angle_q(\bar\lambda_i, \gamma_{q,x_n})$.
Corollary \ref{cor:vert} implies that
\begin{align}\label{eq:omegaQ}
    \omega(q)=\lim_{n\to\infty}\biggl(\pi-\sum_{i=1}^I \theta_i\biggr).
\end{align}

Finally we have to estimate 
$\sum_{k=1}^K \omega_{D_n}(\tilde v_{k})$.
For $v,v'\in V(\tilde T_{n}(\pa D))$, 
define the distance $d(v,v')$ as 
the number of edges in $E(\tilde T_{n}(\pa D))$
contained in the shortest path in $\tilde T_{n}(\pa D)$
 joining $v$ and $v'$.
Set $m_i:=d(\tilde z_i, \tilde q_n)$, and $m_0:=\max_{1\le i\le I} m_i$.
For nonnegative integer $a\le m_0$, we put
\[
\ca V(a):=\{ v\in V(\tilde T_{n}(\pa D))\,|\,
\exists i \,\, d(\tilde z_i,v)=a,
 d(\tilde z_i,\tilde q_n)=a+d(v,\tilde q_n)
\}.
\]
For each $v\in V(\tilde T_{n}(\pa D))$, let $I_v$ denote the
set of all $1\le i\le I$ such that 
the shortest path from $\tilde z_i$ to $\tilde q_n$ contains 
$v$.
Let $E(v)$ denote the set of all edges in $E(\tilde T_{n}(\pa D))$ adjacent to $v$.
Let $E_-(v)$ be the set of edges $e$ in $E(v)$
contained in the shortest path from $v$ to 
$z_i$ with $i\in I_v$. Then we have a
unique edge $e_+(v)$ from $v$ in the direction 
$\tilde q_n$ with 
$E(v)=E_-(v)\cup \{ e_+(v)\}$.

For $e\in E(v)$, set 
$\tilde\eta_e(v):=\angle^{D_n}(e,\gamma_{v,\tilde x_n})$
\,(the inner angle in $D_n$).
Then we obtain
\beq \label{eq:Sum-omegaD(v)}
\begin{aligned}
&\sum_{k=1}^K \omega_{D_n}(\tilde v_{k})=\sum_{k=1}^K
        (\pi-L(\Sigma_{\tilde v_k}(D_n))) \\
   \\
&=\sum_{a=1}^{m_0-1}\sum_{v\in\ca V(a)}
  \biggl( \pi - \biggl(\sum_{e\in E_-(v)}\tilde\eta_{e}(v)\biggr) -\tilde\eta_{e_+(v)}(v)\biggr) +\omega_{D_n}(\tilde q_n).
\end{aligned}
\eeq
For arbitrary $v\in \ca V(m)$ and
$v'\in \ca V(m+1)$ adjacent to an edge
$e$, from \eqref{eq:K(xv)=x*T}
we have 
$\lim_{n\to\infty}\angle v \tilde x_n v'=0$,
which implies 
\begin{align}\label{eq:eta-pm=pi}
  \lim_{n\to\infty} (\tilde\eta_{e_+(v)}(v)+
  \tilde\eta_{e}(v'))=\pi  \quad \text{for all $e\in E_-(v')$}.
\end{align}
For each $1\le i\le I$,  we put 
$\tilde\zeta_{i}:=\angle \tilde z_i \tilde y_i \tilde x_n$. 
Since 
$\lim_{n\to\infty}|q,z_i|/|\tilde y_i, \tilde z_i|=1$,
Lemma \ref{lem:comparison} implies 
\begin{align} \label{eq:zeta=theta}
\lim_{n\to\infty}|\tilde\zeta_{i}-\theta_i|
=0.
\end{align}
By \eqref{eq:eta-pm=pi}, we have for each $1\le a\le m_0-1$
\begin{align}\label{eq:sum-rule}
\lim_{n\to\infty}\bigg\{ \sum_{v\in\ca V(a)}-\tilde\eta_{e_+(v)}(v)
+\sum_{v\in\ca V(a+1)}
\biggl(  \pi - \biggl(\sum_{e\in E_-(v)}\tilde\eta_{e}(v)\biggr)\biggr) \biggr\}=0.
\end{align}
Note also  that $\ca V(m_0)=\{ \tilde q_n\}$.
%
%
%
%
Rearranging the sum of the RHS in \eqref{eq:Sum-omegaD(v)}, we have from  \eqref{eq:omegaQ},
\eqref{eq:zeta=theta} and \eqref{eq:sum-rule}
\begin{align*}
\lim_{n\to\infty}&\sum_{k=1}^K \omega_{D_n}(v_{k}) 
 - \omega(q) \\
  &=\lim_{n\to\infty} \biggl[
  \biggl(\pi-\sum_{i=1}^{I} \tilde\zeta_{i} \biggr)
     -\biggl(\pi-\sum_{i=1}^I \theta_i \biggr)
      -\sum_{v\in\ca V(1)}\tilde\eta_{e_+(v)}(v) \\
 & +\sum_{a=2}^{m_0-1}\sum_{v\in\ca V(a)}
  \biggl(\pi - \biggl(\sum_{e\in E_-(v)}\tilde\eta_{e}(v)\biggr) -\tilde\eta_{e_+(v)}(v) \biggr)+
  \omega_{D_n}(\tilde q_n) \biggr] \\
& = \lim_{n\to\infty} \biggl[
    -\sum_{v\in\ca V(1)}\tilde\eta_{e_+(v)}(v)\\
 & + \sum_{a=2}^{m_0-1}\sum_{v\in\ca V(a)}
  \biggl(  \pi - \biggl(\sum_{e\in E_-(v)}\tilde\eta_{e}(v)\biggr) -\tilde\eta_{e_+(v)}(v)\biggr)+\omega_{D_n}(\tilde q_n) \biggr]\\
&= \lim_{n\to\infty} \biggl[
    -\sum_{v\in\ca V(2)}\tilde\eta_{e_+(v)}(v)\\
 & + \sum_{a=3}^{m_0-1}\sum_{v\in\ca V(a)}
  \biggl(  \pi - \biggl(\sum_{e\in E_-(v)}\tilde\eta_{e}(v)\biggr) -\tilde\eta_{e_+(v)}(v)\biggr)+\omega_{D_n}(\tilde q_n) \biggr]\\
& = \cdots 
 =  \lim_{n\to\infty} \biggl(
    -\sum_{v\in\ca V(m_0-1)}\tilde\eta_{e_+(v)}(v)
    +\omega_{D_n}(\tilde q_n)  \biggr)\\
&= \lim_{n\to\infty} \biggl(
    -\sum_{v\in\ca V(m_0-1)}\tilde\eta_{e_+(v)}(v)
    +\pi-\sum_{e\in E_-(\tilde q_n)}\tilde\eta_e(\tilde q_n)
    \biggr) =0.
 \end{align*}
Thus from \eqref{eq:tau(partialD)-n}, we conclude that 
$\lim_{n\to\infty} |\tau_D(\pa D) -\tau_{D_n}(\pa D_n)|
   =0$.
This completes the proof of (3).
\end{proof}

\begin{rem} \label{rem:yi}\upshape
In the proof of Lemma \ref{lem:constDn}, we used 
$H_0$ to define $\tilde T_{n}(\pa D)$ as 
the replacement of $T_{n}(\pa D)$ for the 
convenience.
However, as the proof shows, there is no strong   
reason to employ $H_0$ since the point is 
\eqref{eq:choice-lambda}.
\end{rem}

\begin{thm}  \label{thm:conv-Omega}
The curvature measure 
$\omega^{X_n}_{BB}$ converges to a unique signed Radon measure
$\omega^X$ on $X$ under the 
convergence $(X_n,p_n)\to (X,p)$.
\end{thm}

\begin{defn} \label{defn:curv-meas2} \upshape
We call $\omega^X$ the {\it curvature measure} of 
$X$. The construction of $X_n$ implies that 
on $X\setminus V_*(\ca S(X))$,
$\omega^X$ coincides with $\omega$ 
given in Definition \ref{defn:omega}.
\end{defn}

\begin{proof}[Proof of Theorem \ref{thm:conv-Omega}] 
For any $R>0$, take 
an admissible domain $D$ containing 
$B(p,R)$, and choose domains
$D_n$ of $X_n$ as in Lemma \ref{lem:constDn}.
Set $\omega_n:=\omega^{X_n}_{BB}$.
By \cite{BurBuy:upperII},  we have Gauss-Bonnet Theorem 
\begin{align} \label{eq:GBF=Dn}
   \omega_n(D_n)= 2\pi\chi(D_n) -\tau_{D_n}(\pa D_n).
\end{align} 
Let $\omega_n^+$ and $\omega_n^-$ be the 
positive and negative parts of $\omega_n$.
We observe

\begin{slem} \label{slem:omega+(B)}
We have $\omega_n^+(B)\le |\kappa|{\rm area}(B)$
for any bounded subset $B$ of $X_n$, 
where ${\rm area}(B)$ denotes the two-dimensional Hausdorff measure of $B$.
\end{slem}
\begin{proof}
Since we could not find a reference, we give a short proof below. See also \cite[Theorem 1.5.3]{ArsBuy}
and \cite[Remark 2.4]{BurBuy:upperII}.
By \cite{BurBuy:upperII}, $X_n$ is the uniform
limit of piecewise smooth metrics  of 
curvature $\le\kappa$, say $X_{n,i}$, on $X_n$. For a fixed face $F$ of $X_n$, let 
$F_i$ be the corresponding face of $X_{n,i}$,
which consists of finitely many smooth simplices.
From Proposition \ref{prop:tau+tau}, we have 
$\omega_{n,i}^+(F_i)\le |\kappa|{\rm area}(F_i)$,
where $\omega_{n,i}$ denotes the curvature measure 
of $X_{n,i}$ defined in \cite{ArsBuy}.
By \cite[Theorem VII.7]{AZ:bddcurv} and the area convergence proved in \cite{Ng:volume},
we have
$\omega_{n}^+(F)\le |\kappa|{\rm area}(F)$.
Since $\omega_n\le 0$ on $\ca S(X_n)$
(see Proposition \ref{prop:tau+tau}), this completes 
the proof.
\end{proof}

From $|\omega_n|=\omega_n^+ +\omega_n^-
=2\omega_n^+ -\omega_n$ together with 
\eqref{eq:GBF=Dn},
we have 
$|\omega_n|(D_n)\le 2|\kappa|{\rm area}(D_n)
-2\pi\chi(D_n) +\tau_{D_n}(\pa D_n)$.
Thus by Lemma \ref{lem:constDn}
together with the area convergence (\cite{Ng:volume}),
$|\omega_n|(D_n)$ is uniformly bounded by 
a constant independent of $n$.
Therefore passing to a subsequence, we may assume that 
$\omega_n$ converges to a singed Radon measure 
$\omega^X$ on $X$.

The uniqueness of $\omega^X$ is deferred to 
Lemma \ref{lem:unique=omega}.
\end{proof}

\begin{proof}[Proof of Theorems \ref{thm:GBf} and \ref{thm:conv-omega}]
From Lemma \ref{lem:constDn} and the definition
of $\omega^X$, we obtain 
the conclusion (4) of Theorem \ref{thm:conv-omega}.
Theorem \ref{thm:GBf} immediately follows from  Theorem \ref{thm:conv-omega} and \eqref{eq:GBF=Dn}.
\end{proof}

\begin{lem}\label{lem:unique=omega}
$\omega^X$ is uniquely determined.
\end{lem}
\begin{proof}
Let $K$ be an arbitrary compact set of $X$.
For any $\e>0$, choose an admissible 
domain $D_\e$ satisfying $K\subset D_\e\subset
B(K,\e)$.
By Theorem \ref{thm:GBf}, we have
$\omega^X(D_\e)=2\pi\chi(D_\e)-\tau_{D_\e}(\pa D_\e)$.
Letting $\e\to 0$, we obtain 
\begin{align} \label{eq:BGF=K}
\omega^X(K)=\lim_{\e\to 0} (2\pi\chi(D_\e)-\tau_{D_\e}(\pa D_\e)),
\end{align}
which implies that 
$\omega^X$ does not depend on the choice of a
subsequence of $\omega^{X_n}_{BB}$.
\end{proof}

\begin{rem} \label{rem:general-form}\upshape
In other words,
\eqref{eq:BGF=K} is the general form of 
the curvature measure $\omega^X$.
\end{rem}

From Theorem \ref{thm:GBf}, we immediately have the following.

\begin{cor}
If $X$ is a two-dimensional geodesically complete, compact metric space with curvature bounded above, then we have
\[
          \omega^X(X) = 2\pi\chi(X).
\]
\end{cor}

\pmed

\subsection{Turn of singular locus} 
\label{ssec:relations}
\psmall\n

In what follows, we simply write as 
$\omega=\omega^X$.

For a fixed $p\in\ca S(X)$, choose $r=r_p>0$ as in 
Subsection 2.3. 
Let $C$ be a singular curve starting from $p$ and 
reaching $S(p,r)$.
Let $N$ be the branching number of the graph $\Sigma_p(X)$ at $v:=\Sigma_p(C)$.
 Choose $\nu_1,\ldots,\nu_N\in\Sigma_p(X)$ such that  for all $1\le i\neq j\le N$
\begin{itemize}
\item \text{ $\angle(\nu_i, v) = \delta$ and
$\angle(\nu_i, \nu_j) = 2\delta\,;$}
\item  $\CAT(\kappa)$ ruled surfaces $S_{ij}$ at $p$  spanned by $\gamma_{\nu_i}$ and $\gamma_{\nu_j}$ are all properly embedded in $B(p,r)$.
\end{itemize}
For $1\le i\le N$, set $\gamma_i:=\gamma_{\nu_i}$
and $z_i:=\gamma_i(r)$.
Consider the $X$-geodesic join $R_i:=z_i*C$.

\begin{center}
\begin{tikzpicture}
[scale = 0.8]
\draw [-, thick] (4.5,0.5) to [out=180, in=30] (2.5,0);
\draw [-, thick] (2.5,0) to [out=210, in=340] (0.5,0);
\draw [-, thick] (0.5,0) to [out=160, in=30] (-1.5,0);
\draw [-, thick] (-1.5,0) to [out=200, in=350] (-3.5,0);
\draw [-, thick] (-3.5,0) to [out=170, in=10] (-4.5,0);
\draw [-, thick] (4.5,-0.5) to [out=180, in=330] (2.5,0);
\draw [-, thick, dotted] (2.5,0) to [out=150, in=20] (0.5,0);
\draw [-, thick] (0.5,0) to [out=200, in=330] (-1.5,0);
\draw [-, thick, dotted] (-1.5,0) to [out=160, in=10] (-3.5,0);
\draw [-, thick] (-3.5,0) to [out=190, in=350] (-4.5,0);
\filldraw[fill=gray, opacity=.1] 
(4.5,0.5) to [out=180, in=30] (2.5,0)
to [out=210, in=340] (0.5,0)
to [out=160, in=30] (-1.5,0)
to [out=200, in=350] (-3.5,0)
to [out=170, in=10] (-4.5,0) -- (4,2.8) -- (4.5,0.5);
\fill (-4.5,0) coordinate (A) circle (2pt) node [left] {$p$};
\draw [thick] (4.5,0.5) -- (4.5,-0.5);
\draw [thick] (4.5,0.5) -- (4,2.8);
\draw [thick] (4.5,-0.5) -- (4,-2.8);
\draw [thick] (4.5,0.5) -- (4.8,2.3);
\draw [thick] (4.5,-0.5) -- (4.8,-2.3);
\draw [thick] (-4.5,0) -- (4,2.8);
\draw [thick] (-4.5,-0) -- (4,-2.8);
\draw [thick, dotted] (-4.5,0) -- (4.8,2.3);
\draw [thick, dotted] (-4.5,-0) -- (4.8,-2.3);
\draw [thick, dotted] (-4.5,0) -- (4.8,-2.3);
\draw [thick] (4.2,-2.15) -- (4.8,-2.3);
\draw [thick, dotted] (-4.5,0) -- (4.8,2.3);
\draw [thick] (4.2,2.15) -- (4.8,2.3);
\fill (4,2.8) coordinate (A) circle (2pt) node [right] {$z_i$};
\fill (2,1) coordinate (A) circle (0pt) node [left] {$R_i$};
\fill (4,0.4) coordinate (A) circle (0pt) node [above left] {$C$};
\end{tikzpicture}
\end{center}

\begin{lem} \label{lem:Si=CAT}
$R_i$ is a $\CAT(\kappa)$-disk.
\end{lem}

Although Lemma \ref{lem:Si=CAT} is immediate from 
a recent result in \cite{LS:dim2}, we give a simple
geometric proof below.

\begin{proof}[Proof of Lemma \ref{lem:Si=CAT}]
For simplicity, we assume $i=N$, and set $R:=R_N$.
For any $\e>0$, consider a decomposition
$\Delta_\e$ of $C$: \,$p=x_0<x_1<\cdots< x_K=q$
with $|\Delta_\e|<\e$ such that for each 
$1\le \alpha\le K$,
$x_{\alpha-1}, x_\alpha\in S_{i_\alpha N}$
for some $1\le i_\alpha\le N-1$.
Set $S_\alpha:=S_{i_\alpha N}$, and 
consider the union
$R_\e$ of the $S_\alpha$-geodesic triangle regions
$\blacktriangle^{S_\alpha} z_N x_{\alpha-1} x_\alpha$ \,$(1\le \alpha \le K)$.
Although some $\blacktriangle^{S_\alpha} z_N x_{\alpha-1} x_\alpha$ may be degenerate at $z_N$, the union $R_\e$ is surely  a $\CAT(\kappa)$-disk.
Now we verify  that $\pa (R_\e\cap R)\setminus (\gamma_{z_N,p}\cup\gamma_{z_N,q})$ is a  singular arc, say
 $C_\e$, joining  $p$ and $q$, where
$\pa (R_\e\cap R)$ denotes the boundary of $R_\e\cap R$ in $R_\e$.

Since $R_\e$ converges  to $R$ as $\e\to 0$,
$R$ is $\CAT(\kappa)$.
Let $\mathring{R}_\e$ (resp. $\mathring{R}$)
denote the complement of 
$\gamma_{z_N,p}\cup\gamma_{z_N,q}\cup C_\e$
(resp. of
$\gamma_{z_N,p}\cup\gamma_{z_N,q}\cup C$)
in $R_\e$ (resp. in $R$).
Since $\mathring{R}=\bigcup_{\e>0}\mathring{R}_\e$
is a surface, one can conclude that $R$ is a $\CAT(\kappa)$-disk 
bounded by $\gamma_{z_N,p}\cup\gamma_{z_N,q}\cup C$
(see also  \cite[Theorem 2.1]{GPW}).
\end{proof}

We consider the turn
$\tau_{R_i}(C)$ of $C$ from the side
$R_i$.

\begin{defn}\label{defn:SXfinite-turn} \upshape
We say that $\ca S(X)$ has {\it locally finite turn
variation}  if 
for any $p\in\ca S(X)$ and any singular curve $C$ starting from $p$ to a point of $S(p,r)$, the turn 
$\tau_{R_i}$ of $C$ from the side
$R_i$ has finite variation for each $1\le i\le N_v$,
where $R_i$ is  the $\CAT(\kappa)$-disk
defined as above.
\end{defn}

Let $p\in\ca S(X)$, $r=r_p>0$,  $v\in V(\Sigma_x(X))$
and $\{ S_{ij}\}$ be as above.
Let $E(v)$ be defined as before.
For the proof of Theorem \ref{thm:bdd-turn(Si)},
we need to study the relations
between $\omega(x)$ and $\omega_{AZ}^{S_{ij}}(x)$
for any $x\in E(v)\cap V(\ca S(X))\setminus \{ p\}$,  where $S_{ij}$ are ruled surfaces containing $x$.

We fix any $x\in B(p,r)\cap V(\ca S(X))\setminus \{ p\}$,
and give an estimate on $|\omega(x)|$.
From Lemma \ref{lem:near-vert},  there is a unique positive
integer $m\ge 3$ such that $\Sigma_x(X)$ is close to 
the spherical suspension over $m$ points
with respect to the Gromov-Hausdorff distance.
Take distinct $m$ points $\xi_1,\ldots, \xi_m\in\Sigma_x(X)$
such that $\angle(\xi_i,-\nabla d_p)=\pi/2$ for all $1\le i\le m$.
Let $B_x:=B(x, r(x))$ be a small metric ball as in 
Subsection \ref{ssec:lcal-str} for the point $x$. By
Theorem  \ref{thm:embedded-disk},
$B_x$ can be written as the union 
\beq \label{eq:Bx=suspension}
       B_x=\bigcup_{1\le i < j\le m} E_{ij},
\eeq
where $E_{ij}$ is a properly embedded $\CAT(\kappa)$-disk in $B_x$ with 
$\xi_i,\xi_j\in\Sigma_x(E_{ij})$.

\begin{lem} \label{EsubsetS}
For any $1\le i<j\le N$,  $E_{ij}$ is contained in 
$S_{i'j'}$ for some $1\le i'<j'\le N$.
\end{lem}
\begin{proof}
Consider $X$-geodesics $\gamma_{\xi_i}$ and $\gamma_{\xi_j}$. Obviously,  $\gamma_{\xi_i}$ and $\gamma_{\xi_j}$ meet $\gamma_{i'}$ and $\gamma_{j'}$
for some $1\le i'\neq j'\le N$.
Then it is easy to verify $E_{ij}\subset S_{i'j'}$. 
\end{proof}

Compare with Lemma \ref{lem:Sx=S}.

\begin{prop} \label{prop:pt-meas}
Under the above situation,  we have
\[
    |\omega(x)| \le  M\sum_{1\le i< j\le m}
        |\omega_{AZ}^{E_{ij}}|(x),
\]
where $M$ is a uniform constant and the sum in the
RHS is taken only for $E_{ij}$ containing $x$.

In particular, $\omega(x)=0$ holds if $x$ is a regular point of 
each $E_{ij}$.
\end{prop}

\begin{rem}  \label{rem:countable} \upshape
Note that $\omega_{AZ}^{E_{ij}}(x)=0$
on $E_{ij}$ except at most countably many 
points $x$. Together with this, 
Proposition \ref{prop:pt-meas} shows that the sum in the RHS of \eqref{eq:curv.measure-welldef}
is at most countable and that 
\begin{align}\label{eq:omega(AB)}
|\omega'( B_x\cap V(\ca S(X)))|\le M\sum_{1\le i< j\le m}
        |\omega_{AZ}^{E_{ij}}|(E_{ij}\cap V(\ca S(X))) <\infty.
\end{align}
\end{rem}

\pmed

For the proof of Proposition \ref{prop:pt-meas}, we need the 
following lemma.
Let $I_i(x)$ be the closure of the component of 
$\Sigma_x(X)\setminus V(\Sigma_x(X))$ containing $\xi_i$.
Note that 
$\Sigma_x(X)\setminus \bigcup_{i=1}^m I_i(x)$ 
consists of two trees, say $\Omega_1(x)$ and $\Omega_2(x)$.

\psmall
\begin{lem} \label{lem:tree-meas}
For each $k=1,2$ and $1\le \ell\le m$, we have 
\begin{align}
 &  L(\Omega_k(x))\le  
   M\sum_{1\le i< j\le m} |\omega_{AZ}^{E_{ij}}|(x) 
    \label{eq:L(Omega)}, \\
 & |L(I_\ell(x))-\pi|\le  
   M\sum_{1\le i< j\le m} |\omega_{AZ}^{E_{ij}}|(x),
       \label{eq:L(I(x))}
\end{align}
where $M$ is a uniform constant.
\end{lem}

We assume Lemma \ref{lem:tree-meas} for a moment.

The idea of the proof of Theorem \ref{thm:bdd-turn(Si)} is as follows. Let us assume $i=N$ and $R:=R_N$.  In the union 
$E_{ijN}=S_{ij}\cup S_{iN}\cup S_{jN}$, which is 
$\CAT(\kappa)$, the singular curve $C_{ijN}$ is 
``nice'' in the sense that it has finite  turn variation
from any side $S\in \{  S_{ij}, S_{iN}, S_{jN}\}$.
For the general singular curve $C$, there might be no ruled surface containing $C$, even in local sometime.
To overcome this difficulty, we shall find a concave region $\tilde R$ in
$M_\kappa^2$ with concave boundary piece 
$\tilde C\subset\pa \tilde R$ isometric to $C$
such that the gluing of $R$ and the two copies of
$\tilde R$ along $C=\tilde C$ becomes $\CAT(\kappa)$. That is, $C$ is nice in the gluing,
which yields the conclusion.

\begin{proof}[Proof of Theorem \ref{thm:bdd-turn(Si)}
assuming Lemma \ref{lem:tree-meas}]
Let $C$ be a  singular curve from $p$ to $q\in S(x,r)$, and let $R_i:=z_i*C$ be the $\CAT(\kappa)$-disks
for $1\le i\le N$ as in the beginning of this subsection.
As before we assume $i=N$ and set $R:=R_{N}$.

For any $\e>0$, from a limit argument together with 
Lemma \ref{lem:sing-arc},
we obtain  a decomposition
$\Delta_\e$ of $C$: \,$p=x_0<x_1<\cdots< x_K=q$
with $|\Delta_\e|<\e$ such that for each 
$1\le \alpha\le K$,
\begin{enumerate}
\item $x_{\alpha-1}, x_\alpha\in C_{i_\alpha j_\alpha N}
=S_{i_\alpha j_\alpha}\cap S_{i_\alpha N}\cap S_{j_\alpha N}$
for some $1\le i_\alpha<j_\alpha\le N-1\,;$
\item let $C_\alpha,  C_{\alpha,\e}$ be the arcs of $C$, $C_{i_\alpha j_\alpha N}$ between $x_{\alpha-1}$ and $x_\alpha$ respectively.
 Then we have 
$|L(C_\alpha)/L(C_{\alpha,\e}) -1|  <\e$.
\end{enumerate}
Note that an endpoint of $\Sigma_{x_\alpha}(C)$ is contained in $\Omega_1(x_\alpha)$ and the other one
in $\Omega_2(x_\alpha)$.
Lemma \ref{lem:tree-meas} then  implies that 
\begin{align}\label{eq:angle(S)-pi}
  \sum_{\alpha=1}^{K-1} |\angle_{x_\alpha}(R)-\pi|\le M\sum_{1\le i<j\le N} |\omega_{AZ}^{S_{ij}}(\mathring{C}\cap S_{ij})|<\tau_p(r).
\end{align}

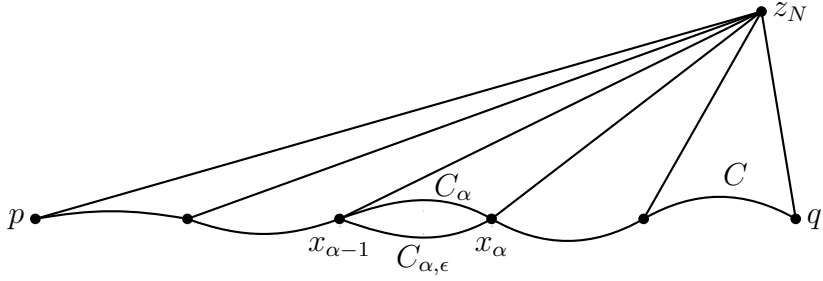
\begin{center}
\begin{tikzpicture}
[scale = 1]
\draw [-, thick] (4.5,0) to [out=150, in=30] (2.5,0);
\draw [-, thick] (2.5,0) to [out=210, in=330] (0.5,0);
\draw [-, thick] (0.5,0) to [out=150, in=20] (-1.5,0);
\draw [-, thick] (-1.5,0) to [out=200, in=340] (-3.5,0);
\draw [-, thick] (-3.5,0) to [out=170, in=10] (-5.5,0);
\draw [-, thick] (0.5,0) to [out=210, in=340] (-1.5,0);
\fill (-5.5,0) coordinate (A) circle (2pt) node [left] {$p$};
\fill (4.5,0) coordinate (A) circle (2pt) node [right] {$q$};
\draw [thick] (4.5,0) -- (4.05,2.75);
\fill (4.05,2.75) coordinate (A) circle (2pt) node [right] {$z_N$};
\fill (-3.5,0) coordinate (A) circle (2pt) node [below] {};
\fill (-1.5,0) coordinate (A) circle (2pt) node [below] {};
\fill (-1.5,-0.1) coordinate (A) circle (0pt) node [below] {$x_{\alpha-1}$};
\fill (0.5,-0.1) coordinate (A) circle (0pt) node [below] {$x_{\alpha}$};
\fill (0.5,0) coordinate (A) circle (2pt) node [below] {};
\fill (2.5,0) coordinate (A) circle (2pt) node [below] {};
\fill (-0.4,0.1) coordinate (A) circle (0pt) node [above right] {$C_{\alpha}$};
\fill (-0.4,-0.2) coordinate (A) circle (0pt) node [below] {$C_{\alpha,\epsilon}$};
\fill (3.4,0.3) coordinate (A) circle (0pt) node [above right] {$C$};
\draw [-, thick] (4.05,2.75) -- (-5.5,0);
\draw [-, thick] (4.05,2.75) -- (-3.5,0);
\draw [-, thick] (4.05,2.75) -- (-1.5,0);
\draw [-, thick] (4.05,2.75) -- (0.5,0);
\draw [-, thick] (4.05,2.75) -- (2.5,0);
\end{tikzpicture}
\end{center}

For each $\alpha$,  set $S_\alpha:=S_{i_\alpha N}$.
For any $\nu>0$, we can choose a decomposition
$\Delta_{\alpha,\nu}$ of $C_{\alpha,\e}$: \,$x_{\alpha-1}=y_0<y_1<\cdots< y_L=x_\alpha$
with $|\Delta_{\alpha,\e}|<\nu$ such that for each 
$1\le \beta\le L$,
\begin{enumerate}
\item[(3)] $ L((C_{\alpha,\e})_{y_{\beta-1},y_\beta})/|y_{\beta-1},y_\beta|_{S_\alpha} <1+\nu$,
where $(C_{\alpha,\e})_{a,b}$ denotes the arc of  
$C_{\alpha,\e}$ between $a$
and  $b\,;$
\item[(4)] letting 
\begin{align*}
&\eta_{\alpha,\beta}:=
\angle(\dot\gamma^{S_\alpha}_{y_{\beta},y_{\beta+1}}(0),
 (\dot C_{\alpha,\e})_{y_{\beta},y_{\beta+1}}(0)),\\
&\rho_{\alpha,\beta}:=
\angle(\dot\gamma^{S_\alpha}_{y_{\beta},y_{\beta-1}}(0),
(\dot C_{\alpha,\e})_{y_{\beta},y_{\beta-1}}(0)),
\end{align*}
we have 
\begin{align}\label{eq:Sum=eta+rho}
\sum_{\beta=1}^{L-1} (\eta_{\alpha,\beta}+\rho_{\alpha,\beta}) +\eta_{\alpha,0}+\rho_{\alpha,L} < \nu/2^\alpha,
\end{align}
where $(\dot C_{\alpha,\e})_{a,b}(0)$ denotes the 
direction at $a$ of the arc of $(C_{\alpha,\e})_{a,b}$.
\end{enumerate}
The last condition is realized from Corollary \ref{cor:vert} together with a limit argument
since we are working on 
a fixed $\CAT(\kappa)$-surface $S_\alpha$.

Consider the broken geodesics 
$C_{\alpha,\e,\nu}$ joining $x_{\alpha-1}$ to $x_\alpha$ and $C_{\e,\nu}$ joining $p$
to $q$ defined as 
$$
C_{\alpha,\e,\nu}
:=\bigcup_{\beta=1}^L \gamma^{S_\alpha}_{y_{\beta-1},y_{\beta}}, \quad
C_{\e,\nu}
:=\bigcup_{\alpha=1}^K C_{\alpha,\e,\nu}.
$$
Now we are going to define a 
concave broken geodesic  $\tilde C_{\e,\nu}$ in $M_\kappa^2$ corresponding to $C_{\e,\nu}$.
First, for each $\alpha$, 
choose a concave broken geodesic 
$\tilde C_{\alpha,\e,\nu}$  
in $M_\kappa^2$ 
joining two points $\tilde x_{\alpha-1}$ and 
$\tilde x_\alpha$   with 
break points 
$\tilde  x_{\alpha-1}=\tilde y_0<\tilde y_1<\cdots< \tilde y_L=\tilde x_\alpha$
satisfying the following
for each $1\le \beta\le L$:
\begin{enumerate}
\item[(5)] $|\tilde y_{\beta-1},\tilde y_\beta|=
|y_{\beta-1}, y_\beta|_{S_\alpha}\,;$
\item[(6)] Let  $\Theta_{\tilde y_\beta}$ denote the angle at
$\tilde y_\beta$ between the geodesics
$[\tilde y_\beta, \tilde y_{\beta-1}]$ and $[\tilde y_\beta, \tilde y_{\beta+1}]$
viewed from the concave region, and let
$\phi_{\alpha,\beta}$ denote the angle
at
$y_\beta$ between the $S_\alpha$-geodesics
$[y_\beta,  y_{\beta-1}]$ and $[y_\beta, y_{\beta+1}]$
viewed from $R$. Then we have 
\begin{align}\label{eq:Theta(y)}
\Theta_{\tilde y_\beta}=
\pi+|\pi-\phi_{\alpha,\beta}|.
\end{align}
\end{enumerate}
Let $R_{\alpha,\e}=z_N*C_{\alpha,\e}$, and let 
$\psi_{\alpha,\beta}:=\angle_{y_\beta}(R_{\alpha,\e})$.
For $1\le i<j\le N-1$, 
let $R_N^{ij}$ be the closure of the component of 
$S_{iN}\setminus C_{ijN}$ containing $z_N$.
Then we have 
\beq
\begin{aligned}\label{eq:sum=pi-psi}
\sum_{\alpha=1}^K\sum_{\beta=1}^L |\pi-\psi_{\alpha,\beta}|
& \le
  \sum_{\alpha=1}^K |\tau_{R_{\alpha,\e}}|(C_{\alpha,\e})\\
&  \le\sum_{1\le i<j\le N-1}|\tau_{R_N^{ij}}|(C_{ijN})
  <\tau_p(r).
\end{aligned}
\eeq
\eqref{eq:Sum=eta+rho} and \eqref{eq:sum=pi-psi}
imply 
\begin{align}\label{eq:sum=pi-phi}
\sum_{\alpha=1}^K\sum_{\beta=1}^L
 |\pi-\phi_{\alpha,\beta}|
   <\tau_p(r) +\nu.
\end{align}

\begin{center}
\begin{tikzpicture}
[scale = 1]
\draw [-, thick] (0,0) to [out=315, in=225] (8,0);
\fill (0,0) coordinate (A) circle (2pt) node [above left] {$x_{\alpha-1}$};
\fill (8,0) coordinate (A) circle (2pt) node [above right] {$x_{\alpha}$};
\fill (4,-1.65) coordinate (A) circle (2pt) node [below] {$y_{\beta}$};
\fill (2,-1.28) coordinate (A) circle (2pt) node [below left] {$y_{\beta-1}$};
\fill (6,-1.28) coordinate (A) circle (2pt) node [below right] {$y_{\beta+1}$};
\draw [-, thick] (4,-1.65) -- (2,-1.28);
\draw [-, thick] (4,-1.65) -- (6,-1.28);
\fill (4,-1.65) coordinate (A) circle (2pt) node [above] {$\phi_{\alpha,\beta}$};
\fill (7.5,0) coordinate (A) circle (0pt) node [below left] {$C_{\alpha,\epsilon}$};
\end{tikzpicture}
\end{center}

Let $\tilde C_{\e,\nu}$  be the concave broken geodesic in $M_\kappa^2$ obtained as the union of $\tilde C_{\alpha,\e,\nu}$\, $(1\le \alpha\le K)$
joining the terminal point of $\tilde C_{\alpha,\e,\nu}$ to the initial point of 
$\tilde C_{\alpha+1,\e,\nu}$ at $\tilde x_{\alpha}$.
Here  if $\Theta_{\tilde x_\alpha}$ denote 
the angle at $\tilde x_\alpha$ between 
$\tilde C_{\alpha,\e,\nu}$ and  
$\tilde C_{\alpha+1,\e,\nu}$ viewed from the 
concave region,
then we require  for each $1\le \alpha\le K$
\begin{align}\label{eq:Theta(x)}
\Theta_{\tilde x_\alpha}
=
\pi+L(\Omega_1(x_\alpha))+L(\Omega_2(x_\alpha))
+\rho_{\alpha,L}+\eta_{\alpha+1,0}.
\end{align}
From Lemma \ref{lem:tree-meas}, \eqref{eq:Theta(y)} and\eqref{eq:sum=pi-phi}, taking small $r>0$ and $\nu$
if necessary, we can realize $\tilde C_{\e,\nu}$ as a concave broken geodesic  in $M_\kappa^2$  in such a way that 
 the concave region $\tilde R_{\e,\nu}$ in $M_\kappa^2$
bounded by $\tilde C_{\e,\nu}$ and two geodesics
joining the endpoints of $\tilde C_{\e,\nu}$ to a point,
say $\tilde z\in M_\kappa^2$, is well-defined.
In other words,
$$
    \tilde R_{\e,\nu}=\tilde z*\tilde C_{\e,\nu}.
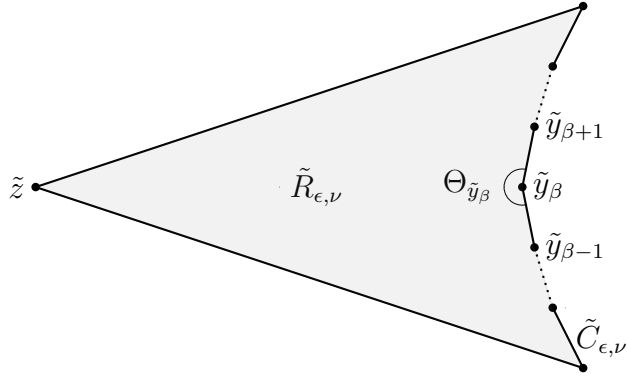
$$

\begin{center}
\begin{tikzpicture}
[scale = 0.8]
\draw [-, thick] (-4,0) -- (5,3);
\draw [-, thick] (-4,0) -- (5,-3);
\draw [-, thick] (4.5,2) -- (5,3);
\draw [-, thick, dotted] (4.2,1) -- (4.5,2);
\draw [-, thick] (4,0) -- (4.2,1);
\draw [-, thick] (4.5,-2) -- (5,-3);
\draw [-, thick, dotted] (4.2,-1) -- (4.5,-2);
\draw [-, thick] (4,0) -- (4.2,-1);
\fill (-4,0) coordinate (A) circle (2pt) node [left] {$\tilde{z}$};
\fill (5,3) coordinate (A) circle (2pt) node [left] {};
\fill (5,-3) coordinate (A) circle (2pt) node [left] {};
\fill (4.5,2) coordinate (A) circle (2pt) node [left] {};
\fill (4.5,-2) coordinate (A) circle (2pt) node [left] {};
\fill (4,0) coordinate (A) circle (2pt) node [right] {$\tilde{y}_{\beta}$};
\fill (4.2,1) coordinate (A) circle (2pt) node [right] {$\tilde{y}_{\beta+1}$};
\fill (4.2,-1) coordinate (A) circle (2pt) node [right] {$\tilde{y}_{\beta-1}$};
\fill (3.7,0) coordinate (A) circle (0pt) node [left] {$\Theta_{\tilde{y}_{\beta}}$};
\fill (0,0) coordinate (A) circle (0pt) node [right] {$\tilde{R}_{\epsilon,\nu}$};
\fill (4.7,-2) coordinate (A) circle (0pt) node [below right] {$\tilde{C}_{\epsilon,\nu}$};
\draw [-] (3.7,0) arc(180:75:0.3);
\draw [-] (3.7,0) arc(180:285:0.3);
\filldraw[fill=gray, opacity=.1] 
(-4,0) -- (5,3) -- (4.5,2) -- (4.2,1) -- (4,0) -- (4.2,-1) -- (4.5,-2) -- (5,-3) -- (-4,0);
\end{tikzpicture}
\end{center}

Set $R_{\e,\nu}:=z_N*C_{\e,\nu}$.
Take two copies $\tilde R'_{\e,\nu}$, $\tilde R''_{\e,\nu}$ of
$\tilde R_{\e,\nu}$ with 
$\tilde C'_{\e,\nu}\subset \pa \tilde R'_{\e,\nu}$, 
$\tilde C''_{\e,\nu}\subset \pa \tilde R''_{\e,\nu}$ corresponding to $\tilde C_{\e,\nu}$.
Let $\tilde E_{\e,\nu}$ be the gluing of 
$R_{\e,\nu}$, $\tilde R'_{\e,\nu}$, $\tilde R''_{\e,\nu}$ along 
$C_{\e,\nu}$, $\tilde C'_{\e,\nu}$, $\tilde C''_{\e,\nu}$ with 
the natural identification
$C_{\e,\nu}=\tilde C'_{\e,\nu}=\tilde C''_{\e,\nu}$.
From construction, we have 
$$
L(\Sigma_{x_\alpha}(\tilde E_{\e,\nu}))\ge 2\pi,\quad
L(\Sigma_{y_\beta}(\tilde E_{\e,\nu}))\ge 2\pi.
$$
Theorem \ref{thm:BB-gluing} then 
implies that $\tilde E_{\e,\nu}$
is $\CAT(\kappa)$. Therefore the GH-limit,
say $\tilde E_\e$, of  $\tilde E_{\e,\nu}$ as $\nu\to 0$ is 
$\CAT(\kappa)$. Then,
the GH-limit,
say $\tilde E$, of  $\tilde E_{\e}$ as $\e\to 0$ is 
$\CAT(\kappa)$.

Let $R_\e(\nu)$ be the complement of 
$\nu$-neighborhood of $C_\e$ in $R_e$.
From construction, $R_\e(\nu)$ is also a subset of  $R_{\e,\nu}$. Note that both 
$B(C_\e,\nu)$ and $B(C_{\e,\nu},\nu)$ converge to 
$C_\e$ as $\nu\to 0$ with respect to the length metric.
From this observation, it is easily verified that
$R_{\e,\nu}$ converges to $R_\e$ as $\nu\to 0$
with respect to the length metric.
Similarly,
$R_{\e}$ converges to $R$ as $\e\to 0$
with respect to the length metric.

We may assume that the concave region 
$\tilde R_{\e,\nu}=\tilde z*\tilde C_{\e,\nu}$
converges to a concave region 
$\tilde R=\tilde z*\tilde C$,
where $\tilde C$ is a concave curve in $M_\kappa^2$ defined as the limit of $\tilde C_{\e,\delta}$.
Since $\tilde C_{\e,\delta}$ has uniformly bounded turn variation,
\cite[Theorem IX.5]{AZ:bddcurv} ensures that 
$L(\tilde C)=L(\tilde C_{\e,\nu})=L(C)$.
Let  $(\tilde R',\tilde C')$ and $(\tilde R'',\tilde C'')$  be the copies of $(\tilde R,\tilde C)$.
Similarly,
$\tilde R'_{\e}$ and $\tilde R''_{\e}$ converge to 
$\tilde R'$ and $\tilde R''$ as $\nu\to 0$ and then $\e\to 0$
with respect to the length metric.

From the above argument, 
it is easy to verify that the GH-limit 
$\tilde E$ is isometric to the gluing of $R$ and 
$\tilde R'$, $\tilde R''$ along  
$C=\tilde C'=\tilde C''$. 
Then Theorem \ref{thm:BB-character} yields that 
$C$ has finite turn variation from the side $R$.
This completes the proof.
\end{proof} 
\begin{rem} \label{rem:proof-bddturn}\upshape
In the proof of Theorem \ref{thm:bdd-turn(Si)},
we took approximations of $C$ twice, firstly by 
$C_\e$ via nice singular arcs 
$\{ C_{\alpha,\e}\}_\alpha$, and
secondly by $\{ C_{\alpha,\e,\nu}\}_{\alpha,\e}$
via broken geodesics in ruled surfaces.
Those two steps seems necessary. 
For instance, if one takes a direct approximation
of $C$ through broken geodesics 
(in ruled surfaces), it does not work since we have no control of ruled surfaces,
which might change quickly.  
\end{rem}
\begin{proof}[Proof of Lemma \ref{lem:tree-meas}]
We simply write as $\Omega_k$ and $I_\ell$ for $\Omega_k(x)$ and $I_\ell(x)$ respectively.
Since the numbers of the vertices of $\Omega_1$ and 
$\Omega_2$ are
uniformly bounded (see Theorem \ref{thm:graph}), it suffices to show 
\[
 \diam(\Omega_k)\le  2\max_{1\le i<j\le m}
        |\omega_{AZ}^{E_{ij}}|(x).
\]
Let $d_1:=\diam(\Omega_1)$, and take 
$u_1,u_2\in \pa \Omega_1$ 
such that $|u_1,u_2|=d_1$.
There are $1\le i <j\le m$ and $1\le k< \ell\le m$
such that $u_1\in\pa I_i\cap \pa I_j$ and $u_2\in\pa I_k\cap \pa I_{\ell}$.
For simplicity we assume  $(i,j)=(1,2)$ and  $(k,\ell)=(3,4)$.
Let $v_i$ \,$(1\le i\le 4)$ be the other point of $\pa I_i$ than $u_1$ or $u_2$.

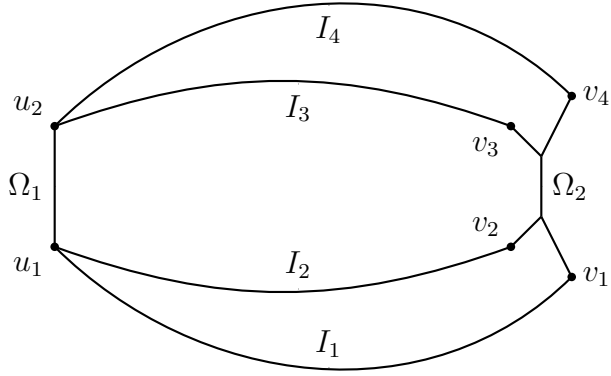
\begin{center}
\begin{tikzpicture}
[scale = 0.8]
\draw [-, thick] (-4,-1) -- (-4,1);
\draw [-, thick] (4,-0.5) -- (4,0.5);
\draw [-, thick] (4.5,1.5) -- (4,0.5);
\draw [-, thick] (3.5,1) -- (4,0.5);
\draw [-, thick] (4.5,-1.5) -- (4,-0.5);
\draw [-, thick] (3.5,-1) -- (4,-0.5);
\draw [-, thick] (4.5,1.5) to [out=135, in=45] (-4,1);
\draw [-, thick] (3.5,1) to [out=160, in=20] (-4,1);
\draw [-, thick] (4.5,-1.5) to [out=225, in=315] (-4,-1);
\draw [-, thick] (3.5,-1) to [out=200, in=340] (-4,-1);
\fill (-4,1) coordinate (A) circle (2pt) node [above left] {$u_2$};
\fill (-4,-1) coordinate (A) circle (2pt) node [below left] {$u_1$};
\fill (4.5,1.5) coordinate (A) circle (2pt) node [right] {$v_4$};
\fill (4.5,-1.5) coordinate (A) circle (2pt) node [right] {$v_1$};
\fill (3.5,1) coordinate (A) circle (2pt) node [below left] {$v_3$};
\fill (3.5,-1) coordinate (A) circle (2pt) node [above left] {$v_2$};
\fill (-4,0) coordinate (A) circle (0pt) node [left] {$\Omega_1$};
\fill (4,0) coordinate (A) circle (0pt) node [right] {$\Omega_2$};
\fill (0,1.7) coordinate (A) circle (0pt) node [below] {$I_3$};
\fill (0,-1.7) coordinate (A) circle (0pt) node [above] {$I_2$};
\fill (0.5,3) coordinate (A) circle (0pt) node [below] {$I_4$};
\fill (0.5,-3) coordinate (A) circle (0pt) node [above] {$I_1$};
\end{tikzpicture}
\end{center}

First consider 

\psmall
\n Case 1). \, $[v_1,v_2]$ does not meet $[v_3,v_4]$.
\psmall
Take $w_1\in [v_1,v_2]$ and $w_2\in [v_3,v_4]$ 
such that $[w_1,w_2]=[v_1,v_3]\cap[v_2,v_4]$.
Set 
\beq
\begin{aligned}\label{eq:set-ell1234}
&\text{
$I_i':=I_i\cup [v_i,w_1]$\, $(i=1,2)$,  
$I_i':=I_i\cup [v_i,w_2]$\, $(i=3,4)$, } \\
&\text{ $\ell_i:=L(I_i')$\,$(1\le i\le 4)$. }
\end{aligned}
\eeq
Since $\Sigma_x(X)$ is a $\CAT(1)$-graph, we have
\[
   \ell_1 +\ell_2\ge 2\pi,\qquad  \ell_3 +\ell_4\ge 2\pi,
\]
where we assume $\ell_1\ge \ell_2$, $\ell_3\ge \ell_4$
for instance, and hence $\ell_1,\ell_3\ge \pi$. Since $I_1'$, $I_3'$, $[u_1,u_2]$ and $[w_1,w_2]$
form a simple loop in $\Sigma_x(E_{13})$ of length $\ge 2\pi +d_1$, we have 
\[
       L(\Sigma_x(E_{13}))\ge 2\pi +d_1,
\]
and therefore $d_1\le |\omega_{AZ}^{E_{13}}(x)|$.

\psmall
\n Case 2). \, $[v_1,v_2]$ meets $[v_3,v_4]$.
\psmall

Take $w_1\in [v_1,v_3]$ and $w_2\in [v_2,v_4]$ 
such that $[w_1,w_2]=[v_1,v_2]\cap[v_3,v_4]$, 
and set $\tilde d_2:=|w_1, w_2|$. 
Let $\ell_i$ \,$(1\le \ell\le 4)$ be defined as in \eqref{eq:set-ell1234}.
As before, we have 
\[
  \ell_1 +\ell_2 +\tilde d_2 \ge 2\pi,\qquad  
  \ell_3 +\ell_4+\tilde d_2\ge 2\pi,
\]
where we assume $\ell_1\ge \ell_2$, $\ell_3\ge \ell_4$ as before.
If $d_1\ge 2\tilde d_2$, then combining the above 
inequalities, we have 
\[
     L(\Sigma_x(E_{13}))= \ell_1+ \ell_3+d_1
        \ge 2\pi +d_1-\tilde d_2\ge 2\pi + d_1/2,
\]
and hence $d_1\le 2|\omega_{AZ}^{E_{13}}(x)|$.
Now suppose that $d_1 < 2\tilde d_2$.
Since  $\ell_2 +\ell_4+d_1\ge 2\pi$,
we get 
\[
   L(\Sigma_x(E_{14}))=\ell_1+ \ell_4+d_1+\tilde d_2
        \ge 2\pi + \ell_1-\ell_2 +\tilde d_2\ge 2\pi + d_1/2,
\]
and hence 
$d_1\le 2|\omega_{AZ}^{E_{14}}(x)|$.
Therefore we conclude  
$$
    \diam(\Omega_1)\le  2\max_{1\le i< j\le m}
        |\omega_{AZ}^{E_{ij}}(x)|.
$$
By a similar argument, we obtain 
the same inequality for $\diam(\Omega_2)$.
Thus we obtain \eqref{eq:L(Omega)}.

Let $\zeta_{ij}$ be the simple loop in $\Sigma_x(X)$
containing $I_i\cup I_j$.
Since $|L(\zeta_{ij})-2\pi|\le|\omega_{AZ}^{E_{ij}}(x)|$
for all $1\le i<j\le m$,
\eqref{eq:L(I(x))} follows from \eqref{eq:L(Omega)}.
This completes the proof of Lemma \ref{lem:tree-meas}. 
\end{proof}

\psmall
\begin{proof}[Proof of Proposition \ref{prop:pt-meas}]
We proceed by induction on $m$.
By Theorem \ref{thm:ruled},  
$$
                  X_0:=\bigcup_{1\le i< j\le m-1} E_{ij}
$$
is a $\CAT(\kappa)$-space.
Note that $I_m$ coincides with the closure of 
$\Sigma_x(X)\setminus\Sigma_x(X_0)$.
Let $\{ u_m, v_m\}:=\pa I_m$,
where we may assume $u_m\in\pa \Omega_1$ and 
$v_m\in\pa \Omega_2$.
Since $u_m,v_m\in V(\ca S(X))$, we take
$(i,j)$ and $(k,\ell)$ such that 
\[
\{ u_m\}=\Omega_1\cap \Sigma_x(C_{ijm}), \qquad
\{ v_m\}=\Omega_2\cap \Sigma_x(C_{k\ell m}),
\]
where $C_{ijm}$ is a singular curve in $\ca S(X)$ as in Lemma \ref{lem:sing-arc}.
We set 
\[
\{ u_m, v_m'\} :=\Sigma_x(C_{ijm}), \,v_m'\in \Omega_2,\quad 
\{ u_m', v_m\} :=\Sigma_x(C_{k\ell m}), \,u_m'\in \Omega_1.
\]
From $E_{ij}$, $E_{im}$, $E_{jm}$ and $C_{ijm}$,
we have three half disks $F_i$, $F_j$ and $F_m$
around $x$, satisfying $F_i\cap F_j\cap F_m=C_{ijm}$,  
whose spaces of directions at $x$ 
contain $\xi_i$, $\xi_j$ and $\xi_m$ respectively.
Letting $\ell_i=L(\Sigma_x(F_i))$, 
we  have
\begin{align*}
 &   \ell_i +\ell_j -2\pi=|\omega_{AZ}^{E_{ij}}(x)|,\quad
 \ell_i +\ell_m -2\pi=|\omega_{AZ}^{E_{im}}(x)|,
\\
&\hspace{1.7cm} \ell_j +\ell_m -2\pi=|\omega_{AZ}^{E_{jm}}(x)|,
\end{align*}
which imply
\[
|\ell_m-\pi| \le \frac{1}{2}\left(|\omega_{AZ}^{E_{ij}}(x)|+
     |\omega_{AZ}^{E_{im}}(x)| + |\omega_{AZ}^{E_{jm}}(x)|
      \right).
\]
It follows from Lemma \ref{lem:tree-meas} that 
\begin{align*}
|\pi-&|u_m,v_m||  \le |\pi -\ell_m|+|v_m,v_m'|\\
      &\le |\pi -\ell_m|+\diam (\Omega_2)\le M' \sum_{\alpha<\beta}|\omega_{AZ}^{E_{\alpha\beta}}(x)|.
\end{align*}
Finally, from the inductive hypothesis, we conclude that 
\begin{align*}
|\omega(x)| &= |m\pi - L(\Sigma_x(X))| \\
      &\le |(m-1)\pi -L(\Sigma_x(X_0))|+|\pi - |u_m,v_m||\\
     &\le M'' \sum_{\alpha<\beta}|\omega_{AZ}^{E_{\alpha\beta}}(x)|.
\end{align*}
This completes the proof of Proposition \ref{prop:pt-meas}.
\end{proof}

\begin{center}
\begin{tikzpicture}
[scale = 0.8]
\draw [-, thick] (3.5,0) to [out=150, in=30] (0,0);
\draw [-, thick] (0,0) to [out=210, in=330] (-3.5,0);
\draw [-, thick] (-3.5,0) -- (3.5,0);
\draw [-, thick] (-3.5,0) -- (-3.5,1);
\draw [-, thick] (3.5,0) -- (3.5,1);
\draw [-, thick] (-1.5,3) -- (1.5,3);
\draw [-, thick] (-3.5,1) to [out=90, in=180] (-1.5,3);
\draw [-, thick] (3.5,1) to [out=90, in=0] (1.5,3);
\draw [-, thick] (-3.5,0) -- (-4,-2);
\draw [-, thick] (3.5,0) -- (3,-2);
\draw [-, thick] (-4,-2) to [out=255, in=180] (-2,-3);
\draw [-, thick] (-2,-3) -- (1,-3);
\draw [-, thick] (1,-3) to [out=0, in=255] (3,-2);
\draw [-, dotted, thick] (-3.5,0) -- (-3,-2);
\draw [-, thick] (3.5,0) -- (4,-2);
\draw [-, dotted, thick] (-1,-2.5) -- (2.5,-2.5);
\draw [-, dotted, thick] (-3,-2) to [out=290, in=180] (-1,-2.5);
\draw [-, thick] (4,-2) to [out=290, in=0] (2.9,-2.5);
\fill (0,0) coordinate (A) circle (2pt) node [left] {};
\fill (0,-0.1) coordinate (A) circle (0pt) node [below] {$x$};
\fill (0,3) coordinate (A) circle (0pt) node [above] {$F_m$};
\fill (-0.5,-3) coordinate (A) circle (0pt) node [below] {$F_i$};
\fill (3.5,-2.5) coordinate (A) circle (0pt) node [below] {$F_j$};
\draw [thick] (1,0) arc(0:205:1);
\fill (1,0) coordinate (A) circle (1pt) node [below] {$v_m'$};
\fill (0.92,0.4) coordinate (A) circle (1pt) node [above right] {$v_m$};
\fill (-1,0) coordinate (A) circle (1pt) node [above left] {$u_m$};
\fill (-0.92,-0.4) coordinate (A) circle (1pt) node [below] {$u_m'$};
\fill (0,1) coordinate (A) circle (0pt) node [above left] {$I_m$};
\fill (3.5,0.3) coordinate (A) circle (0pt) node [above left] {$C_{k \ell m}$};
\fill (3.4,0) coordinate (A) circle (0pt) node [below left] {$C_{ijm}$};
\end{tikzpicture}
\end{center}

\psmall 
\begin{rem} \upshape
Proposition \ref{prop:pt-meas} holds for every point $x$ of 
$B\cap \ca S(X)\setminus \{ p\}$, too.
Actually, if $x\in B\cap E(\ca S(X))\setminus \{ p\}$, 
then both $\Omega_1(x)$ and $\Omega_2(x)$ are points,  and 
$\Sigma_x(X)$ is the two points union of $I_1,\ldots, I_m$
along boundaries.
It follows that 
$\omega(x)=m\pi-\Sigma_x(X)$.
Therefore the same proof as 
Proposition \ref{prop:pt-meas}  applies to
$x$.
\end{rem}
\psmall
\begin{rem} \upshape
Under the situation of Proposition \ref{prop:pt-meas},
the condition $\omega(x)=0$ does not imply 
the regularlity of $x$ in all $E_{ij}$ containing $x$.
For instance, take the $\CAT(1)$-graph defined as 
follows: Let $I_i=[\xi_i, \eta_i]$\,$(i=1,2)$ be two arcs of 
length $\e>0$.
Let $L_1$ and $L_2$ be segments of  length $\pi$
joining $\xi_1$ to $\xi_2$ and $\eta_1$ to $\eta_2$ respectively.
Let $L_3$ and $L_4$ be segments of  length $\pi-\e$
joining $\xi_1$ to $\eta_2$ an $\eta_1$ to $\xi_2$ respectively.
Then $I_1$,$I_2$, $L_1, \ldots, L_4$ form 
a $\CAT(1)$-graph $\Sigma_\e$ and if $\Sigma_x(X)=\Sigma_\e$
for some $(X,x)$, then we have $\omega(x)=0$
while $\omega_{AZ}^{E_{12}}(x)=-2\e$.

This really happens in our situation. Namely, slightly modifying \cite[Example 4.4]{NSY:local},
we have a two-dimensional, locally compact geodesically complete $\CAT(0)$-space $X$ such 
that there are points $p$ and $x_n$ in $X$ such that 
$x_n\to p$ and $\Sigma_{x_n}(X)=\Sigma_{\e_n}$ for some $\e_n\to 0$
\end{rem}

\pmed
\subsection{Explicit formula} \label{sec:explicit}
\psmall

Since $\omega$ is defined as the limit of 
$\omega^{X_n}_{BB}$, it is interesting to  find
an explicit formula of $\omega$ on singular curves $C$, especially when $C\cap V_*(\ca S(X))$ is a large set. 
In this subsection, we assume that $C$ is a singular curve starting from $p$ to $q\in S(p,r)$ that is contained in a ruled surface.
We need this condition to start the induction process,
which is discussed below.
Remember that $\ca S(X)$ is a locally finite union of
those singular curves.

 Let $\gamma_1,\ldots,\gamma_N$,
$\{ S_{ij}\}_{1\le i<j\le N}$, 
and  $R_i=\gamma_i(r)*C$\,  $(1\le i\le N)$ be as before.
For a subset $I$ of $\{ 1,\ldots,N\}$ 
with $\sharp I\ge 2$,
let  
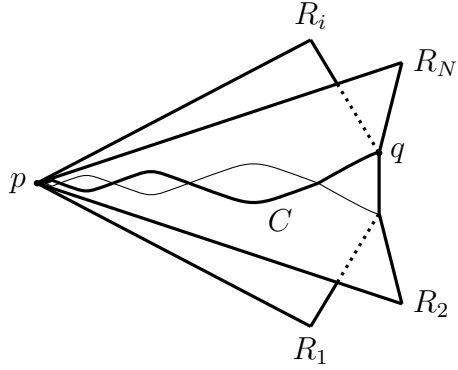
\begin{align}\label{eq:WI}
W_I:= {\rm int}  \biggl(\bigcap_{i\in I} R_i \setminus C
\biggr), \qquad e_I:=\bar W_I\cap C.
\end{align}

\begin{center}
\begin{tikzpicture}
[scale = 1]
\fill (0,0) coordinate (A) circle (1.2pt) node [left] {$p$};
\fill (4.5,0.4) coordinate (A) circle (1.2pt) node [right] {$q$};
\fill (4.8,1.6) coordinate (A) circle (0pt) node [right] {$R_N$};
\fill (3.6,1.9) coordinate (A) circle (0pt) node [above] {$R_i$};
\fill (3.6,-1.9) coordinate (A) circle (0pt) node [below] {$R_1$};
\fill (4.8,-1.6) coordinate (A) circle (0pt) node [right] {$R_2$};
\fill (3.2,-0.2) coordinate (A) circle (0pt) node [below] {$C$};
\draw [-, very thick] (0,0)--(0.1,0);
\draw [-, very thick] (4.5,0.4)--(4.8,1.6);
\draw [-, very thick] (4.5,-0.4)--(4.8,-1.6);
\draw [-, very thick] (0,0)--(4.8,1.6);
\draw [-, very thick] (0,0)--(4.8,-1.6);
\draw [dotted, very thick] (3.95,1.315)--(4.5,0.4);
\draw [-, very thick] (3.95,1.315)--(3.6,1.9);
\draw [-, very thick] (0,0)--(3.6,1.9);
\draw [-, very thick] (0,0)--(3.6,-1.9);
\draw [dotted, very thick] (3.95,-1.315)--(4.5,-0.4);
\draw [-, very thick] (3.95,-1.315)--(3.6,-1.9);
\draw [-, very thick] (4.5,0.4) -- (4.5,-0.4);
\coordinate (P1) at (3.7,0);
\coordinate (P2) at (2,0);
\coordinate (P3) at (1,0);
\coordinate (P4) at (0.3,0);
\coordinate (P5) at (0.1,0);
\coordinate (P6) at (4.5,0.4);
\coordinate (P7) at (4.5,-0.4);
\draw [very thick]
(P1) .. controls +(30:1cm) and +(180:0cm) .. (P6);
\draw 
(P6) .. controls +(270:0.5cm) and +(90:0.5cm) ..
(P7) .. controls +(180:0cm) and +(330:1cm) .. (P1); 
\draw 
(P1) .. controls +(160:1cm) and +(20:1cm) .. (P2);
\draw [very thick]
(P2) .. controls +(340:1cm) and +(200:1cm) .. (P1); 
\draw [very thick]
(P2) .. controls +(160:0.6cm) and +(20:0.6cm) .. (P3);
\draw 
(P3) .. controls +(340:0.6cm) and +(200:0.6cm) .. (P2); 
\draw 
(P3) .. controls +(160:0.4cm) and +(20:0.4cm) .. (P4);
\draw [very thick]
(P4) .. controls +(340:0.4cm) and +(200:0.4cm) .. (P3); 
\draw [very thick]
(P4) .. controls +(160:0.15cm) and +(20:0.15cm) .. (P5);
\draw 
(P5) .. controls +(340:0.15cm) and +(200:0.15cm) .. (P4); 
\end{tikzpicture}
\end{center}

We easily have the following.

\begin{lem}\label{lem:Wi}
Suppose $W_I$ is nonempty. Then we have 
\begin{enumerate}
 \item $W_I$ is at most countable union
 of open disks $W_{I,\alpha}$\,$(\alpha\in A_I)$\,$;$
 \item $\pa W_I$ is at most countable union of singular curves\,$;$
 \item $e_I$ is the union of subarcs 
 $e_{I,\alpha}:=e_I\cap \bar W_{I,\alpha}$ of $C$.
\end{enumerate}
\end{lem}

Set $E_1:=S_1$.
For $2\le k\le N$,
let $E_k:=\bigcup_{1\le i<j\le k} S_{ij}$. 
We may assume that $C\subset S_{12}$.
Note that 
$$
      E_k=E_{k-1}\cup R_k\quad  (2\le k\le N).
$$
In \cite[Theorem 7.4]{NSY:local}, we showed that
$E_k$ is the GH-limit of polyhedral 
$\CAT(\kappa)$-space
$E_{k,n}$, the result of $\e_n$-surgery on
$E_k$. In particular, $E_k$ is $\CAT(\kappa)$
for all $1\le k\le N$.
Let $\omega_{k,\infty}$ be the curvature measure
of $E_k$ defined as the limit of 
$\omega^{E_{k,n}}_{BB}$ as $n\to\infty$. 

In what follows,  we assume that $C$ is the {\it open}
singular arc joining $p$ to $q$ as above, and inductively
provide an explicit definition of
$\omega_k(C)$ satisfying 
$\omega_k(C)=\omega_{k,\infty}(C)$.
By Proposition \ref{prop:tau+tau}, for $k=2$, it suffices to set
$\omega_2(C):=\tau_{R_1}(C)+\tau_{R_2}(C)$.
Suppose we already defined $\omega_{k-1}(C)$
satisfying $\omega_{k-1}(C)=\omega_{k-1,\infty}(C)$.

Let $\{ W_{k,\alpha}\}_{\alpha\in A}$ be the set of components of $(E_{k-1}\cap R_k)\setminus C$.
By Lemma \ref{lem:Wi}, $A$ is at most countable. 
For 
$e_{k,\alpha}:=\bar W_{k,\alpha}\cap C$,
let $\{ x_{k,\alpha}, x'_{k,\alpha}\}:= \pa e_{k,\alpha}$
with $|p,x_{k,\alpha}|<|p,x'_{k,\alpha}|$.
In some cases, there are $\beta,\beta'\in A$ such that 
$p\in\bar W_{k,\beta}$ and $q\in\bar W_{k,\beta'}$,
or equivalently, $p=x_{k,\beta}$ and $q=x'_{k,\beta'}$.
We fix such  $\beta,\beta'$ if any.
Under this circumstance, we set 
\begin{align*}
&\theta_{k,\alpha}:=\angle_{x_{k,\alpha}}(\bar W_{k,\alpha}),\quad \theta'_{k,\alpha}:=\angle_{x'_{k,\alpha}}(\bar W_{k,\alpha}), \\
&\eta_{k,\alpha}:=\angle_{x_{k,\alpha}}(R_k)-\theta_{k,\alpha},\quad 
\eta'_{k,\alpha}:=\angle_{x'_{k,\alpha}}(R_k)-\theta'_{k,\alpha}, \\
&\tau_{E_{k-1}\cap R_k}(C)
:=\sum_{\alpha\in A} \tau_{\bar W_{k,\alpha}}(\mathring{e}_{k,\alpha}), \\
&\theta_{E_{k-1}\cap R_k}(C)
:=\sum_{\alpha\neq \beta,\beta'} (\theta_{k,\alpha} +\theta'_{k,\alpha}) +
   \theta'_{k,\beta}+ \theta_{k,\beta'}.
\end{align*}
Since $\Sigma_p(\bar W_{k,\beta})=\{ v\}$, we have 
$\theta_{k,\beta}=0$.
Lemma \ref{lem:tree-meas} implies  that 
$\theta_{E_{k-1}\cap R_k}(C)$ is finite.
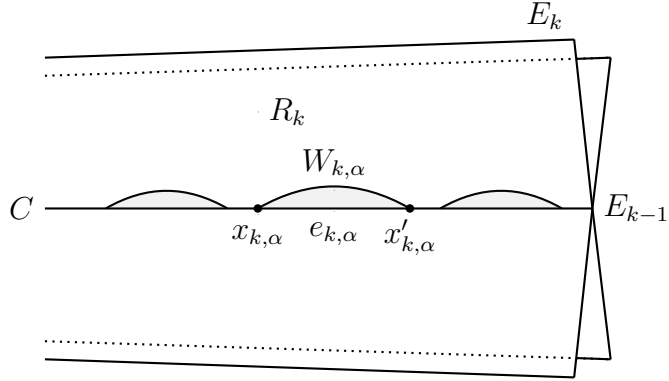
\psmall

\begin{center}
\begin{tikzpicture}
[scale = 0.8]
\draw [-, thick] (4.5,0) -- (-4.5,0);
\draw [-, thick] (4.5,0) -- (4.2,2.8);
\draw [-, thick] (4.5,0) -- (4.8,2.5);
\draw [-, thick] (4.5,0) -- (4.2,-2.8);
\draw [-, thick] (4.5,0) -- (4.8,-2.5);
\draw [-, thick] (4.25,2.48) -- (4.8,2.5);
\draw [-, thick] (4.25,-2.48) -- (4.8,-2.5);
\draw [-, thick] (-4.5,2.5) -- (4.2,2.8);
\draw [-, thick, dotted] (-4.5,2.2) -- (4.8,2.5);
\draw [-, thick] (-4.5,-2.5) -- (4.2,-2.8);
\draw [-, thick, dotted] (-4.5,-2.2) -- (4.8,-2.5);
\draw [-, thick] (4,0) to [out=150, in=30] (2,0);
\draw [-, thick] (1.5,0) to [out=150, in=30] (-1,0);
\draw [-, thick] (-1.5,0) to [out=150, in=30] (-3.5,0);
\filldraw[fill=gray, opacity=.1] 
(4,0) to [out=150, in=30] (2,0) -- (4,0);
\filldraw[fill=gray, opacity=.1] 
(1.5,0) to [out=150, in=30] (-1,0) -- (1.5,0);
\filldraw[fill=gray, opacity=.1] 
(-1.5,0) to [out=150, in=30] (-3.5,0) -- (-1.5,0);
\fill (-1,1.6) coordinate (A) circle (0pt) node [right] {$R_k$};
\fill (4.5,0) coordinate (A) circle (0pt) node [right] {$E_{k-1}$};
\fill (4.2,2.8) coordinate (A) circle (0pt) node [above left] {$E_k$};
\fill (0.25,-0.05) coordinate (A) circle (0pt) node [below] {$e_{k,\alpha}$};
\fill (0.25,0.3) coordinate (A) circle (0pt) node [above] {$W_{k,\alpha}$};
\fill (1.5,0) coordinate (A) circle (2pt) node [below] {$x_{k,\alpha}'$};
\fill (-1,0) coordinate (A) circle (2pt) node [below] {};
\fill (-1,-0.1) coordinate (A) circle (0pt) node [below] {$x_{k,\alpha}$};
\fill (-4.5,0) coordinate (A) circle (0pt) node [left] {$C$};
\end{tikzpicture}
\end{center}

We have the following explicit formula for $\omega(C)$. 

\begin{thm} \label{thm:explicit} 
Assume $C\subset S_{12}$. Then we have
\begin{align} \label{eq:explicit=omega(C)}
\omega(C)= &\sum_{k=1}^m \tau_{R_k}(C) 
      -\sum_{k=3}^m \left\{ \tau_{E_{k-1}\cap R_k}(C) -\theta_{E_{k-1}\cap R_k}(C)\right\}
\end{align}
\end{thm}
\psmall

In the below, we see that the second term of 
the RHS of \eqref{eq:explicit=omega(C)} is 
the correction term for the duplication in the sum of the first term. 

For the proof of Theorem \ref{thm:explicit},
we define 
\begin{align}  \label{eq:omgak(C)}
\omega_k(C):=  \omega_{k-1}(C)+\tau_{R_k}(C) -\{ \tau_{E_{k-1}\cap R_k}(C) -\theta_{E_{k-1}\cap R_k}(C)\},
\end{align}
and prove

\begin{claim} \label{claim:omega=omega}
$\omega_k(C)=\omega_{k,\infty}(C)$.
\end{claim}

Theorem \ref{thm:explicit} immediately follows 
from the cases $k=3, \dots, N$ in Claim \ref{claim:omega=omega}.

\begin{proof}[Proof of Claim \ref{claim:omega=omega}]
Take an admissible domain $D_\e=D_\e^k$  in $E_k$
with $C\subset D_\e\subset B(C,\e)$ such that 
\begin{itemize}
\item $\pa D_\e$ consists of broken $R_i$-geodesics $\gamma_{\e,i}$ \,$(1\le i\le k)$ joining
$p$ and $q$, 
transversally meeting $\bar C$ at $p,q$
 such that 
$$
\lim_{\e\to 0}\angle_p(\gamma_{\e,i},\bar C)=0,\quad
\lim_{\e\to 0}\angle_q(\gamma_{\e,i},\bar C)=0\,;
$$
\item $\gamma_{\e,i}=\gamma_{\e,j}$ on $R_i\cap R_j$ for all $1\le i\neq j\le k\,;$ 
\item $\gamma_{\e,k,\alpha}:=\gamma_{\e,k}\cap W_{k,\alpha}$ is an arc transversally meeting 
$\pa  W_{k,\alpha}$
if it is nonempty for each $\alpha\in A$.
\end{itemize}
This is certainly possible if one slightly deforms 
$\pa D_\e^{k-1}$ when needed.
Let $\{ y_{k,\alpha}, y'_{k,\alpha}\}
:=\pa \gamma_{\e,k,\alpha}$ with 
$|p,y_{k,\alpha}|<|p,y'_{k,\alpha}|$.

Since $\gamma_{\e,k}$ meets 
only finitely many $W_{k,\alpha}$,
we line up them as $\{ W_{k,\alpha}\}_{\alpha=1}^M$ in order.
Let $A_\e:=\{ 1,\ldots,M\}\subset A$.
Let $A_\e'$ be the set of $\alpha\in A_\e$
such that $x_{k,\alpha-1}'\neq x_{k,\alpha}$,
and $A_\e'':=A_\e\setminus A_\e'$.
We may assume that $A_{\e_1}''\subset A_{\e_2}''$ whenever $\e_2<\e_1$.

For $\alpha\in A_\e$,
let $C_{k,\alpha}$ (resp. $C'_{k,\alpha-1}$)
 be the arc of $\pa W_{k,\alpha}$ between
 $x_{k,\alpha}$ and $y_{k,\alpha}$
 (resp of $\pa W_{k,\alpha-1}$ between
 $x'_{k,\alpha-1}$ and $y'_{k,\alpha-1}$).
Let $\gamma_{\e,k,\alpha-1,\alpha}$ denote 
the arc of $\gamma_{\e,k}$ between $y_{k,\alpha-1}'$ and $y_{k,\alpha}$.
Let  $\zeta_{k,\alpha}$, $\zeta'_{k,\alpha-1}$ denote
the angles between 
$\gamma_{\e,k,\alpha-1,\alpha}$ and $C_{k,\alpha}$, $C_{k,\alpha-1}'$ at
$y_{k,\alpha}$, $y_{k,\alpha-1}'$ respectively.

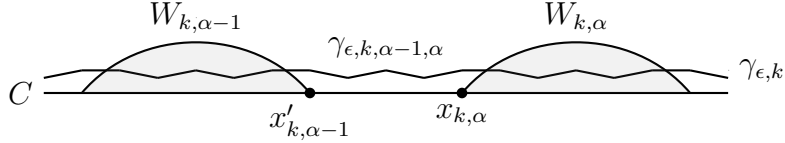
\begin{center}
\begin{tikzpicture}
[scale = 1]
\draw [-, thick] (4.5,0) -- (-4.5,0);
\draw [-, thick] (4,0) to [out=130, in=50] (1,0);
\draw [-, thick] (-1,0) to [out=130, in=50] (-4,0);
\draw [-, thick] (0,0.3) -- (0.5,0.2);
\draw [-, thick] (0.5,0.2) -- (1,0.3);
\draw [-, thick] (1,0.3) -- (1.5,0.3);
\draw [-, thick] (1.5,0.3) -- (2,0.2);
\draw [-, thick] (2,0.2) -- (2.5,0.3);
\draw [-, thick] (2.5,0.3) -- (3,0.2);
\draw [-, thick] (3,0.2) -- (3.5,0.3);
\draw [-, thick] (3.5,0.3) -- (4,0.3);
\draw [-, thick] (4,0.3) -- (4.5,0.2);
\draw [-, thick] (0,0.3) -- (-0.5,0.2);
\draw [-, thick] (-0.5,0.2) -- (-1,0.3);
\draw [-, thick] (-1,0.3) -- (-1.5,0.3);
\draw [-, thick] (-1.5,0.3) -- (-2,0.2);
\draw [-, thick] (-2,0.2) -- (-2.5,0.3);
\draw [-, thick] (-2.5,0.3) -- (-3,0.2);
\draw [-, thick] (-3,0.2) -- (-3.5,0.3);
\draw [-, thick] (-3.5,0.3) -- (-4,0.3);
\draw [-, thick] (-4,0.3) -- (-4.5,0.2);
\filldraw[fill=gray, opacity=.1] 
(4,0) to [out=130, in=50] (1,0) -- (4,0);
\filldraw[fill=gray, opacity=.1] 
(-1,0) to [out=130, in=50] (-4,0) -- (-1,0);
\fill (1,0) coordinate (A) circle (2pt) node [below] {$x_{k,\alpha}$};
\fill (-1,0) coordinate (A) circle (2pt) node [below] {$x_{k,\alpha-1}'$};
\fill (-2.5,1) coordinate (A) circle (0pt) node [] {$W_{k,\alpha-1}$};
\fill (2.5,1) coordinate (A) circle (0pt) node [] {$W_{k,\alpha}$};
\fill (-4.5,0) coordinate (A) circle (0pt) node [left] {$C$};
\fill (4.5,0.3) coordinate (A) circle (0pt) node [right] {$\gamma_{\epsilon,k}$};
\fill (0,0.3) coordinate (A) circle (0pt) node [above] 
{$\gamma_{\epsilon,k,\alpha-1,\alpha}$};
\end{tikzpicture}
\end{center}
 
For  $\alpha\in A_\e'$, we may assume  
the following:
\beq \label{eq:angle=yka}
\begin{aligned}
&\angle_{y_{k,\alpha}}(W_{k,\alpha}\cap\bar D_\e)=\pi-\theta_{k,\alpha},\quad \zeta_{k,\alpha}=\pi-\eta_{k,\alpha},
\\
&\angle_{y'_{k,\alpha-1}}(W_{k,\alpha-1}\cap\bar D_\e)
=\pi-\theta'_{k,\alpha-1},\quad \zeta'_{k,\alpha-1}=\pi-\eta'_{k,\alpha-1},
\end{aligned}
\eeq
Note that 
$\angle_{y_{k,\alpha}}(R_k\cap\bar D_\e)
=2\pi-\theta_{k,\alpha}-\eta_{k,\alpha}$, 
$\angle_{y'_{k,\alpha-1}} (R_k\cap\bar D_\e)
=2\pi-\theta'_{k,\alpha-1}-\eta'_{k,\alpha-1}$.
Since $\pa W_{k,\alpha}\setminus \mathring{e}_{k,\alpha}$ has definite directions everywhere and 
$y_{k,\alpha}$, $y'_{k,\alpha-1}$ are close to
$x_{k,\alpha}$, $x'_{k,\alpha-1}$ respectively, 
we can choose $D_\e$ satisfying the above 
conditions.

For  $\alpha\in A_e''$, set 
$\rho_{k,\alpha}:=\angle_{x_{k,\alpha}}(R_k)-\theta_{k,\alpha}-\theta_{k,\alpha-1}'$.
We may assume  
the following.
\begin{itemize}
\item[(a)] If $\rho_{k,\alpha}<\pi$, then the arc 
 $\gamma_{\e,k,\alpha-1,\alpha}$
 is the $R_k$-minimal geodesic $\,;$
\item[(b)] If  $\rho_{k,\alpha}\ge \pi$, then $\gamma_{\e,k,\alpha-1,\alpha}$ 
is a broken $R_k$-geodesic with a unique break point
such that $\zeta_{k,\alpha}<\e/M$ and 
$\zeta'_{k,\alpha-1}<\e/M$. 
\end{itemize}

For $D_\e$, choose  an admissible domain 
$D_{\e,n}$ in $E_{k,n}$ as in Lemma \ref{lem:constDn}.
Then by Gauss-Bonnet Theorem for 
$D_{\e,n}$ together with Lemma \ref{lem:constDn},
we have
\beq \label{eq:omega=infty}
\begin{aligned}
 \omega_{k,\infty}(C)
 &= \lim_{\e\to 0} \omega_{k,\infty}(D_\e) \\
& =\lim_{\e\to 0} 
     \lim_{n\to\infty} \omega_{BB}^{E_{k,n}}(D_{\e,n}) \\
&=2\pi - \lim_{\e\to 0} \tau_{D_\e}(\pa D_\e).
\end{aligned}
\eeq
Set $D'_\e:= D_\e\cap E_{k-1}$, which is an  admissible domain of $E_{k-1}$. Then we have
$\pa D'_\e= \pa D_\e\cap E_{k-1}$.

From now, we often omit the subscript $k$
when causing no confusion.
For $\alpha\in A_\e$,
let $C_{x'_{\alpha-1},x_\alpha}$ denote the 
subarc of $C$ between  
$x'_{\alpha-1}$ and $x_\alpha$, which is of course
the single point $\{ x_\alpha\}$ if $\alpha\in A_\e''$. 
Let $Q_\alpha$ be the region in $D_\e\cap R_k$ bounded by $\gamma_{\e,k,\alpha-1,\alpha}$,
 $C_{x_{\alpha-1}', x_\alpha}$, $C'_{\alpha-1}$ and 
 $C_\alpha$. 
From \eqref{eq:defn-tauD(paD)}, we obtain    

\begin{align*} 
\tau_{D_\e}(\pa D_\e) &= \tau_{D'_\e}(\pa D'_\e)
+\tau_{D_\e\cap R_k}(\pa D_\e\setminus \pa D'_\e)   
 \\
&-\sum_{\alpha\in A_\e}
\bigl(\angle_{y_{\alpha}}(Q_\alpha)+
    \angle_{y'_{\alpha-1}}(Q_\alpha)\bigr)
  \\
&- \angle_p(\gamma_{\e,k},\bar C)
-\angle_q(\gamma_{\e,k},\bar C),
\end{align*}
which implies 
\beq
\begin{aligned} \label{eq:turn(paD)}
2\pi-\tau_{D_\e}&(\pa D_\e) =
\,2\pi-\tau_{D'_\e}(\pa D'_\e)
-\tau_{D_\e\cap R_k}(\pa D_\e\setminus \pa D'_\e)    \\
&+\sum_{\alpha\in A'_\e}(2\pi-\eta_{\alpha}
-\eta'_{\alpha-1}) \\
&+\sum_{\alpha\in A''_\e}(\zeta_\alpha+\zeta'_{\alpha-1}) + \angle_p(\gamma_{\e,k},\bar C)
+\angle_q(\gamma_{\e,k},\bar C).
\end{aligned}
\eeq
By induction, we have
\begin{align} \label{eq:omega-infty}
\omega_{k-1,\infty}(C)=2\pi - \lim_{\e\to 0} \tau_{D_\e\cap E_{k-1}}(\pa D'_\e) = \omega_{k-1}(C).
\end{align}

Now we are going to calculate the third term of the RHS in \eqref{eq:turn(paD)}.
For $\alpha\in A_\e'$,
from Gauss-Bonnet Theorem for $Q_\alpha$
 together with \eqref{eq:angle=yka}, we have
\beq \label{eq:GBF=A1}
\begin{aligned}
 \omega(Q_\alpha)= & -\tau_{Q_\alpha}(\gamma_{\e,k,\alpha-1,\alpha}) - \tau_{Q_\alpha}(\mathring{C}_{x'_{\alpha-1}, x_\alpha})\\
  & \,\,-\tau_{Q_\alpha} (C_\alpha)-\tau_{Q_\alpha} (C'_{\alpha-1}).
\end{aligned}
\eeq

For $\alpha\in A_\e''$,
from Gauss-Bonnet Theorem for $Q_\alpha$,
we have
\beq \label{eq:GBF=A2}
\begin{aligned}
 \omega(Q_\alpha)= & -\tau_{Q_\alpha}(\gamma_{\e,k,\alpha-1,\alpha}) 
  -\tau_{Q_\alpha} (C_\alpha)-\tau_{Q_\alpha} (C'_{\alpha-1})  \\
  &+ \zeta_\alpha+\zeta_{\alpha-1}'+\rho_\alpha-\pi.
\end{aligned}
\eeq

Taking $y_\alpha$ and $y'_{\alpha-1}$ 
sufficiently close to $x_\alpha$ and $x'_{\alpha-1}$ respectively, we may assume
$|\tau_{Q_\alpha}(C_\alpha)|<\e/M$ and 
$|\tau_{Q_\alpha}(C'_{\alpha-1})|<\e/M$.
For $\alpha\in A_e'$, this is certainly possible since
the deformation is only locally performed near
$y_\alpha$ and $y'_{\alpha-1}$.
For $\alpha\in A_e''$, this is also possible since
this new requirement does not interfere with  the other parts of $D_\e\setminus D_\e'$.

Observe that
\begin{align*}
&\pi-\eta_\alpha=\tau_{R_k}(x_\alpha)+\theta_\alpha,
\quad
\pi-\eta'_{\alpha-1}=\tau_{R_k}(x'_{\alpha-1})+\theta'_{\alpha-1},\\
&\pi-\rho_\alpha=\tau_{R_k}(x_\alpha)+\theta_\alpha+
\theta'_{\alpha-1}.
\end{align*}
It follows that 
\begin{align}\label{eq:Sum=2pi-eta}
\lim_{\e\to 0}\sum_{\alpha\in A'_\e}
   (2\pi-\eta_\alpha-\eta'_{\alpha-1})
=\lim_{\e\to 0}\sum_{\alpha\in A'_\e}
(\tau_{R_k}(x_\alpha)+\tau_{R_k}(x'_{\alpha-1})
+\theta_\alpha + \theta'_{\alpha-1}).
\end{align}
Since $\lim_{\e\to 0}\sum_{\alpha\in A_\e} |\omega|(Q_\alpha)=0$,  from \eqref{eq:GBF=A2} we have

\begin{align} \label{eq:SumA2=tau}
 \lim_{\e\to 0}\sum_{\alpha\in A_\e''} \bigl\{
 -\tau_{Q_\alpha}(\gamma_{\e,k,\alpha-1,\alpha}) 
  +\zeta_\alpha+\zeta_{\alpha-1}'\bigr\}
 =\sum_{\alpha\in A''} (\tau_{R_k}(x_\alpha)
 +\theta_\alpha+\theta'_{\alpha-1}),
\end{align}
where $A'':=\bigcup_{\e>0}A_\e''$.
Since 
\begin{align}\label{eq:sum=gamma-alpha}
-\lim_{\e\to 0} & \tau_{D_\e\cap R_k}(\pa D_\e\setminus \pa D'_\e) 
=-\lim_{\e\to 0}\sum_{\alpha\in A_\e}
\tau_{Q_\alpha}(\gamma_{\e,k,\alpha-1,\alpha}),
\end{align}
combining \eqref{eq:omega=infty},
\eqref{eq:turn(paD)},
\eqref{eq:omega-infty} and \eqref{eq:sum=gamma-alpha} 
  we have
\begin{align*}
\omega_{k,\infty}(C)&=\omega_{k-1}(C)+\lim_{\e\to 0}
 \sum_{\alpha\in A_\e'} \bigl\{-\tau_{Q_\alpha}(\gamma_{\e,\alpha-1,\alpha}) +
 (2\pi-\eta_\alpha-\eta'_{\alpha-1})\bigr\} \\
 & 
 + \lim_{\e\to 0} \sum_{\alpha\in A_\e''} \bigl\{
  -\tau_{Q_\alpha}(\gamma_{\e,\alpha-1,\alpha}) 
  + (\zeta_\alpha +\zeta'_{\alpha-1})\bigr\}.
\end{align*} 
It follows from  \eqref{eq:Sum=2pi-eta} and \eqref{eq:SumA2=tau} that
\begin{align*}
\omega_{k,\infty}(C)&=
\omega_{k-1}(C)  \\
&+\lim_{\e\to 0}
 \sum_{\alpha\in A_\e'} \bigl\{\tau_{Q_\alpha}(C_{x'_{\alpha-1},x_\alpha}) +
 \tau_{R_k}(x_\alpha)+\tau_{R_k}(x'_{\alpha-1})
 +\theta_\alpha + \theta'_{\alpha-1}
\bigr\} \\
& 
 + \lim_{\e\to 0} \sum_{\alpha\in A_\e''} 
 (\tau_{R_k}(x_\alpha)
 +\theta_\alpha + \theta'_{\alpha-1}).
 \end{align*}

Let $K_\e$ be the union of $\bigcup_{\alpha\in A_\e'}  \bar C_{x'_{\alpha-1},x_\alpha}$ and 
$\bigcup_{\alpha\in A_\e''} \{ x_\alpha \}$.
Set $K:=\bigcup_{\e>0}K_\e$.
Together with Theorem \ref{thm:bdd-turn(Si)}, we conclude
\begin{align*}
\omega_{k,\infty}(C)&=\omega_{k-1}(C)+\lim_{\e\to 0}
 \tau_{R_k}(K_\e) + \theta_{E_{k-1}\cap R_k}(C) \\
 & 
 = \omega_{k-1}(C) +\tau_{R_k}(K)+
    \theta_{E_{k-1}\cap R_k}(C) \\
&=\omega_{k-1}(C) + \tau_{R_k}(C)-\tau_{E_{k-1}\cap R_k}(C)  + \theta_{E_{k-1}\cap R_k}(C).
 \end{align*}
This completes the proof of Claim \ref{claim:omega=omega}.
\end{proof}

\pmed
\setcounter{equation}{0}

\section{Characterization Theorem} \label{sec:converse}
\psmall
In this section, we prove Theorem \ref{thm:converse}.

\pmed\n
\subsection{Examples and preliminary argument} \label{ssec:prelim-character}

First, we exhibit several examples which show
the conditions in Theorem \ref{thm:converse} are optimal. 
It is easy to construct counterexamles
when one drops the condition (1) or (2) among all the conditions in Theorem \ref{thm:converse}. Therefore we shall focus on the conditions
(3) and (4).
\pmed

\begin{ex} \label{ex:gluing-disk} \upshape
Let $D\subset\R^2$ be a closed metric ball,
and let $X$ be the gluing of three copies of $D$
along their boundaries. Let $\ca S$ be the subset of $X$
corresponding to $\pa D$.
Then $(X,\ca S)$
satisfies all the conditions except $(3)$ in Theorem \ref{thm:converse}.
\end{ex}

\begin{ex}  \label{ex:attaching}\upshape
On the plane $\mathbb R^2$, choose a sequence
 $\{ B_i\}_{i=1}^\infty$
of mutually disjoint closed balls converging to a point
$p\in\mathbb R^2$.
Removing  $\mathring{B}_i$  from $\mathbb R^2$ and gluing with punctured tori $T_i$ with suitably chosen metrics along their boundaries for all $i$,
we have a space $X$ with the topological singular point set
$\ca S=\{ p\}$
such that $(X,\ca S)$ 
satisfies all the conditions  except $(4)$ in Theorem \ref{thm:converse}.
\end{ex}

\begin{ex} \label{ex:counter-inv}\upshape
On the $xy$-plane $\R^2$, let $C_+$ (resp. $C_-$) be the union of the ray $\{ (0,y)\,|\,y\le 0\}$ and the graph $y=\sqrt{x}$\,\,$(x\ge 0)$ (resp. $y=\sqrt{|x|}$ \,\,$(x\le 0)$).
Let $D_+$ and $D_-$ be the closed convex domains
bounded by $C_+$ and $C_-$ respectively.
Take three copies $\R^2_i$\,$(i=0,1,2)$ of $\R^2$, and
denote by $D_+^i$ and $D_-^i$ be domains on $\R^2_i$
corresponding to $D_+$ and $D_-$.
In the union $\R^2_0\cup\R^2_1\cup\R^2_2$, we  
glue $D_+^i$ and $ D_-^{i+1}$ \,$({\rm mod}\,3)$
by the isometry between them. The result $X$ of this gluing
is not $\CAT(0)$ since there are arbitrary short geodesic
loops near the origin of $X$. 
Let $\ca S$ be the subset of $X$ corresponding to 
the union $\pa D^i_+ \cup\pa D^i_+$ for all $i$.
Then $(X, \ca S)$ satisfies all the conditions except $(4)$ in Theorem \ref{thm:converse}.
\end{ex}

\psmall
Next, we provide some fundamental local properties of $X$ derived from the conditions 
$(1) - (4)$ in Theorem \ref{thm:converse}.

Let $(X, \ca S)$ be as in Theorem \ref{thm:converse}.
In what follows, $\angle^X$ denotes the upper angle 
in $X$.
More precisely, the upper angle between two $X$-shortest geodesics
$\gamma$ and $\sigma$ starting from $x$ is defined as 
\begin{align}\label{eq:upper-angleX}
  \angle^X(\gamma,\sigma):=  \limsup_{t,s\to 0} 
    \tilde\angle^X \gamma(t) x \sigma(s).
\end{align}
Then $\Sigma_x(X)$ is defined as the $\angle^X$-completion of the set of all $X$-geodesics starting from $x$.
\psmall

Let $p\in \ca S$, and $r=r_p>0$ be sufficiently small 
so as to satisfy the condition (3).
From now, we work on $B(p,r)$. 

For any fixed $v_0\in V(\Sigma_p(X))$, let $N_0:=N_{v_0}$ be the branching number of the graph $\Sigma_p(X)$ at $v_0$. 
For some $\delta_0<\min\{ \angle(v_0, w)\,|\,w\neq v_0 \in V(\Sigma_p(X))\,\}/10$, choose $\zeta_1,\ldots,\zeta_{N_0}\in\Sigma_p(X)$ such that 
\begin{itemize}
\item \text{ $\angle(\zeta_i, v) = \delta_0$ and
$\angle(\zeta_i, \zeta_j) = 2\delta_0$ for $1\le i\neq j\le N_0\,;$}
\item a $\CAT(\kappa)$-sector surface $S_{ij}$ at $p$ is provided by the condition (3) for the arc 
$[\zeta_i,\zeta_j]$ in $\Sigma_p(X)$ so as to satisfy 
$\iota(\Sigma_p(S_{ij}))=[\zeta_i,\zeta_j]$.
\end{itemize}

Let $\gamma_i$ be an $X$-geodesic
starting from $p$ in the direction $\zeta_i$.
We may assume that 
$S_{ij}$ is bounded by $\gamma_i,\gamma_j$ and $S(p,r)$. 

We set 
\begin{align} \label{eq:S(v)}
     E(v_0):=\bigcup_{1\le i<j\le N_0} S_{ij}.
\end{align}

\begin{lem} \label{lem:metric-circle}
We have the following for each small enough $t>0\,:$
\begin{enumerate}
 \item $S(p,t)\cap E(v_0)$ is a tree. 
 \item Each component of $S(p,t)\setminus \bigcup_{v\in V(\Sigma_p(X))} E(v)$ is an arc,
 where $E(v)$ is defined as in \eqref{eq:S(v)} for
 $v$.
\end{enumerate}
\end{lem}

\begin{proof} The conditions (3) and (4) imply that 
$S(p,t)\cap E(v)$ is a connected graph containing 
$N_v$ endpoints on the boundaries of the sector surfaces for each $v\in V(\Sigma_p(X))$. 
Each component $A$ of $S(p,t)\setminus \bigcup_{v\in V(\Sigma_p(X))} E(v)$ 
is a path joining an endpoint of the graph $S(p,t)\cap E(v_1)$
to an endpoint of $S(p,t)\cap E(v_2)$ for 
arbitrary adjacent vertices $v_1, v_2\in V(\Sigma_p(X))$.
The condition (4) then implies that $S(p,t)\cap E(v)$ 
must be 
contractible for each $v\in V(\Sigma_p(X))$ and $A$ is an arc. This completes the proof.
\end{proof}

Lemma \ref{lem:metric-circle} implies that 
$E(v_0)$ is the closed domain bounded by 
$\gamma_1,\ldots,\gamma_{N_0}$ and $S(p, r)\cap E(v_0)$.

For a closed subset $A$ with $p\in A$,
we denote by $\Sigma_p(A)$  the set of 
all $\xi\in\Sigma_p(X)$ which can be written as 
$\xi=\lim_{n\to\infty}\dot\gamma^X_{x,x_n}(0)$ with
$x_n\in A$ and $x_n\to x$.
\psmall
\begin{lem}\label{lem:regular-sector}
For every $p\in \ca S$ and every 
$\xi\in \Sigma_p(X)\setminus V(\Sigma_p(X))$, we have the following:
\begin{enumerate}
\item There is an $X$-geodesic $\gamma$ starting from 
$p$ in the direction $\xi\,;$
 \item For every 
closed arc $I$ in $\Sigma_p(X)\setminus V(\Sigma_p(X))$, 
 there is a $\CAT(\kappa)$-sector surface $S$ at $p$ 
 such that $\Sigma_p(S)= I$ and $S\cap\ca S=\{ p\}$.
\end{enumerate}
\end{lem}

\begin{proof}
Let $I=[\xi_1,\xi_2]$. In what follows, we may assume 
$\angle(\xi_1,\xi_2)<\pi$.
Take $X$-minimal geodesics  $\gamma_i:[0, \e]\to X$
starting from $p$ with directions, say $\xi_i'$, close to $\xi_i$ \,$(i=1,2)$ with 
$[\xi_1,\xi_2]\subset [\xi_1',\xi_2']$
and $\angle(\xi_1',\xi_2')<\pi$.
We may assume $\gamma_i((0,\e])$ do not meet $\ca S$.
For small $0<t<s<\e$, consider the region
$\Delta(t,s)$ bounded by $\gamma_1$, $\gamma_2$,
$\gamma_{\gamma_1(s),\gamma_2(s)}$ and 
$\gamma_{\gamma_1(t),\gamma_2(t)}$.
By the condition (1),  $\Delta(t,s)$ is a  locally $\CAT(\kappa)$-surface.  Lemma \ref{lem:metric-circle} (2) implies that $\Delta(t,s)$ is
homeomorphic to a disk, and therefore it is (globally)
$\CAT(\kappa)$. Taking the limit of $\Delta(t,s)$
as $t\to 0$, we obtain a $\CAT(\kappa)$-sector surface $S'$
at $p$ with $\Sigma_p(S')=[\xi_1',\xi_2']$. 
Since $S'\setminus\{ p\}\subset X\setminus\ca S$,
$S'$-geodesics in the direction $\xi_i$ are $X$-geodesics.
Thus we have the desired $\CAT(\kappa)$-sector surface $S$
at $p$ with $\Sigma_p(S)=[\xi_1,\xi_2]$ and
$S\cap\ca S=\{ p\}$.
\end{proof}
\psmall

The following is a key in the proof of Theorem \ref{thm:converse}. The proof is deferred to Subsection 
\ref{ssec:Proof-Thm1.4}.

\begin{thm} \label{thm:union-CAT}
$E(v_0)$ is a $\CAT(\kappa)$-space.
\end{thm}
The statement of Theorem \ref{thm:union-CAT} is proved in \cite[Theorem 7.4]{NSY:local} 
in a more general form, under the assumption 
that $X$ is locally $\CAT(\kappa)$, which we have to show eventually in the present setting.
The strategy of the proof of Theorem \ref{thm:union-CAT} is based on that of \cite[Theorem 7.4]{NSY:local}.

\begin{lem}\label{lem:Cijk}
For $1\le i<j<\ell \le N_0$, we have the following.
\begin{enumerate}
\item The intersection 
$S_{ij}\cap S_{j\ell}\cap S_{\ell i}$ provides a singular curve
$C_{ij\ell}$ starting from $p$ and reaching 
$S(p,r)$. Moreover, $d_p$ is strictly increasing along 
$C_{ij\ell}\,;$
\item $C_{ij\ell}$ has definite direction everywhere.
\end{enumerate}
\end{lem}
\begin{proof} (1)\,
By Lemma \ref{lem:metric-circle},
for every $0<t\le r$, 
$G(t):=S(p,t)\cap (S_{ij}\cup S_{j\ell}\cup S_{\ell i})$ must be a tree.
Let $T(t)$ be the tripod in $G(t)$ with 
$\pa T(t)=\{ \gamma_i(t),\gamma_j(t),\gamma_\ell(t)\}$. 
Let $C_{ij\ell} (t)$ be the branching point of 
$T(t)$. The uniquness of $C_{ij\ell} (t)$ shows
that it is continuous and provides a singular curve. Obviously we have 
$C_{ij\ell} (t)=S(p,t)\cap S_{ij}\cap S_{j\ell}\cap S_{\ell i}$.
\par\n
(2)\,
From Lemma \ref{lem:regular-sector}, we have 
\begin{align}\label{eq:SigmaC=v0}
\Sigma_p(C_{ij\ell})=\{ v_0\}.
\end{align}
The conclusion follows immediately since 
the same argument applies to any other point of $C_{ij\ell}$. 
\end{proof}

\eqref{eq:SigmaC=v0} immediately implies the following.

\begin{lem} \label{lem:S=Vertex}
For every  $p\in\ca S$, we have 
\[ 
       \Sigma_p(\ca S)=V(\Sigma_p(X)).
\]
\end{lem}

We set
$E_{ij\ell}:=S_{ij}\cup S_{j\ell}\cup S_{\ell i}$.

\begin{thm} \label{thm:Cijl}
$C_{ij\ell}$ has finite turn variation in any of  $S_{ij},  S_{j\ell}$ and $S_{\ell i}.$
\end{thm}
\begin{proof}
Let $F_i, F_j, F_\ell$ be the completions of components of $E_{ij\ell}\setminus C_{ij\ell}$ containing $\gamma_i,
\gamma_j,\gamma_\ell$ respectively. 
Let $e$ be any open arc of $C$.
For  arbitrary distinct elements $\alpha, \beta$ of 
$\{ i, j, \ell\}$, from Proposition \ref{prop:tau+tau},
we have
\begin{align} \label{eq:tau-alpha-beta}
    \tau_{F_\alpha}(e)+\tau_{F_\beta}(e)=\omega_{AZ}^{S_{\alpha\beta}}(e).
\end{align}
Note that $|\omega_{AZ}^{S_{\alpha\beta}}(e)|<\tau_p(r)$.  From 
$\tau_{F_\ell}(e)=
-\tau_{F_i}(e)+\omega_{AZ}^{S_{i\ell}}(e)
=-\tau_{F_j}(e)+\omega_{AZ}^{S_{j\ell}}(e)$
and 
$\tau_{F_i}(e)=
-\tau_{F_j}(e)+\omega_{AZ}^{S_{ij}}(e)$,
we have 
\[
\tau_{F_j}(e)=\frac{1}{2}(\omega_{AZ}^{S_{ij}}(e)+
\omega_{AZ}^{S_{j\ell}}(e)-\omega_{AZ}^{S_{i\ell}}(e)),
\]
and
$
|\tau_{F_j}(e)|\le \frac{1}{2}(|\omega_{AZ}^{S_{ij}}|(e)+
|\omega_{AZ}^{S_{j\ell}}|(e)+|\omega_{AZ}^{S_{i\ell}}|(e)).
$
This completes the proof.
 \end{proof}

\begin{cor}\label{cor:Cijk2}
For $1\le i<j<\ell \le N_0$, $E_{ij\ell}$ is a 
$\CAT(\kappa)$-space.
\end{cor}
\begin{proof}
Remark that $\Sigma_x(E_{ij\ell})$ is $\CAT(1)$ for 
any $x\in C$ since $\Sigma_x(E_{ij\ell})$ is a gluing of
the three arcs $\Sigma_x(S_{ij})$,
 $\Sigma_x(S_{j\ell})$ and $\Sigma_x(S_{\ell k})$.
 Since $S_{\alpha\beta}$ is $\CAT(\kappa)$,
 we have $\tau_{F_\alpha}+\tau_{F_\beta}\le 0$
 for all $1\le\alpha<\beta\le 3$.
 Therefore Theorems \ref{thm:BB-gluing} and
 \ref{thm:Cijl}
 imply that $E_{ij\ell}$ is $\CAT(\kappa)$.
 \end{proof}

\begin{lem} \label{lem:exist-singular-curve}
For every $x\in E(v_0)\cap\ca S$ with $x\in S_{ij}$,
there is a singular curve $C=C_{ij\ell}$ 
contained in $S_{ij}$ through $x$ starting from $p$ and reaching 
$S(p,r)$.
\end{lem}
\begin{proof}
Let $T:= S(p,|p,x|_X)\cap E(v_0)$, which is a tree
having $N$ endpoints $\{ a_i\}_{i=1}^N$ with $a_i\in\gamma_i$. Note that $x\in S_{ij}\cap\ca S$ is an interior vertex 
of $T$ contained in the segment $(a_i,a_j)\subset T$. Therefore we can choose some $\ell\neq
i,j$ such that $x$ is the unique branch point of 
the sub-tree of $T$ containing $a_i, a_j, a_\ell$.
Thus we have $x\in C_{ij\ell}$.
\end{proof}

\psmall\n
\subsection{Sector surfaces} \label{ssec:geom-S}

In this subsection, we obtain some fundamental 
properties of sector surfaces, and apply them to 
get geometric properties of $X$.

 For every $x\in E(v_0)\setminus \{ p\}$, the following is easily checked:
\begin{align} \label{eq:SigmaX=Union}
     \Sigma_x(X)=\bigcup_{1\le i<j\le N_0, x\in S_{ij}} \Sigma_x(S_{ij}),
\end{align}
where $\Sigma_x(S_{ij})\subset\Sigma_x(X)$ is the 
extrinsic space of directions of $S_{ij}$ at $x$,
and  two distinct $\Sigma_x(S_{ij})$ and 
$\Sigma_x(S_{i'j'})$ are glued based on 
$\Sigma_x(\ca S)$.

Here is a remark on our notation.
In Theorem \ref{thm:converse}, 
we also use $\Sigma_p(S)$ to denote 
the intrinsic space of directions of a sector surface
$S$. To distinguish it from the extrinsic one, we use the symbol $\Sigma^{\rm int}_p(S)$ 
to denote the intrinsic space of directions. 
Now we show  that they coincide.

We use the symbol $o_p(x)$ to denote a positive
constant depending on $p,x$ such that
$\lim_{x\to p} o_p(x)=0$.

\begin{prop} \label{prop:Eint-ext}
For every $x\in S_{ij}$, $\Sigma^{\rm int}_x(S_{ij})$ is isometric to $\Sigma_x(S_{ij})$.
\end{prop} 

\begin{proof}
Obviously, we may assume $x\in \ca S$. 
If $x=p$, the conclusion is assured by the condition (3).  Therefore we only have to consider the 
case $x\neq p$.
We first check that each component $\Sigma$ of 
$\Sigma_x(S_{ij})\setminus \Sigma_x(\ca S)$
is isometrically embedded in $\Sigma^{\rm int}_x(S_{ij})$.
For $\xi_1, \xi_2\in \Sigma$ with $|\xi_1,\xi_2|<\pi$,
let $\mu_n$ be an $X$-geodesic with
$\dot\mu_n(0)=\xi_n$ \,$(n=1,2)$. Then for small $\e$, we have  
 $\mu_n([0,\e])\subset S_{ij}$. 
 Note that the $X$-minimal geodesic 
 $\gamma^X_{\mu_1(t_1),\mu_2(t_2)}$ for any 
 $0<t_1, t_2\le \e$  
 does not meet $\ca S$,
 and hence $\gamma^X_{\mu_1(t),\mu_2(t)}$
  is contained in   $S_{ij}$.
 This implies that $\angle^X(\xi_1,\xi_2)= \angle^{S}(\xi_1,\xi_2)$ and the existence of  a locally isometric embedding   
 $\varphi:\Sigma_x(S_{ij})\setminus \Sigma_x(\ca S)
\to \Sigma^{\rm int}_x(S_{ij}\setminus (\ca S\setminus \{ x\}))$.

We show that $\varphi$ is surjective.
For every $\xi\in\Sigma^{\rm int}_x(S_{ij}\setminus (\ca S\setminus \{ x\}))$,
take an arc $I\subset \Sigma^{\rm int}_x(S_{ij}\setminus (\ca S\setminus \{ x\}))$
containing $\xi$ in its interior.
Choose $\xi_1,\xi_2,\eta_1, \eta_2$ in the interior of $I$ 
such that $\xi\in (\xi_1,\xi_2)\subset [\xi_1,\xi_2]\subset (\eta_1,\eta_2)$.
Set $\gamma_\alpha:=\gamma^{S_{ij}}_{\xi_\alpha}$ and 
$\sigma_\alpha:=\gamma^{S_{ij}}_{\eta_\alpha}$
\,$(\alpha=1,2)$.
From the choice of $I$, the $S_{ij}$-geodesics
$\gamma^{S_{ij}}_{\sigma_1(t),\sigma_2(t)}$
joining $\sigma_1(t)$ and 
$\sigma_2(t)$ do not meet 
$\ca S$ except $x$ for all small enough $t>0$,
 and therefore form a disk domain $S$
 in $S_{ij}\setminus (\ca S\setminus \{ x\})$.
 This implies that 
 $\gamma^{S_{ij}}_{\sigma_1(t),\sigma_2(t)}$,
 $\gamma_1$ and $\gamma_2$ 
 are local $X$-geodesics, and hence
 by Lemma \ref{lem:regular-sector},
 $\dot\gamma_1(0)$ and $\dot\gamma_2(0)$
 provide the elements, say
 $\dot\gamma^X_1(0)$ and $\dot\gamma^X_2(0)$
  of $\Sigma_x(S_{ij})$
 with $[\dot\gamma^X_1(0),\dot\gamma^X_2(0)]
 \subset \Sigma_x(S_{ij})\setminus\Sigma_p(\ca S)$. Thus we conclude $\xi\in [\dot\gamma_1(0),\dot\gamma_2(0)] =\varphi([\dot\gamma^X_1(0),\dot\gamma^X_2(0)])$.
 
 From Lemma \ref{lem:S=Vertex},
 $\Sigma^{\rm int}_x(S_{ij}\cap\ca S)
=\bigcup_{1\le \ell\neq i,j\le N_0} \Sigma^{\rm int}_x(C_{ij\ell})$ is identified with
 $\Sigma_x(S_{ij}\cap\ca S)=\bigcup_{1\le \ell\neq i,j\le N_0} \Sigma_x(C_{ij\ell})$.
This completes the proof.
\end{proof}
 
From now on, by Proposition \ref{prop:Eint-ext}, we make an identification 
$\Sigma^{\rm int}_x(S_{ij})=\Sigma_x(S_{ij})$.

\psmall
\begin{lem} \label{lem:almost-suspension}
For every $x\in E(v_0)\cap \ca S\setminus \{ p\}$, the following hold:
\begin{enumerate}
\item For every $S_{ij}$ containing $x$, 
$\Sigma_x(S_{ij})$ is a circle of 
length $\in [2\pi, 2\pi+o_p(x))\,;$
\item $\diam(\Sigma_{x,\pm}(\ca S))<o_p(x)$,
 where 
 $\Sigma_{x,-}(\ca S):=\Sigma_x(\ca S\cap B(p,d^X_p(x)))$,
 $\Sigma_{x,+}(\ca S):=\Sigma_x(\ca S\cap (X\setminus \mathring{B}(p,d^X_p(x))))\,;$
\item $\Sigma_x(X)$  is 
$o_p(x)$-close to a metric suspension
over finitely many points with respect to the 
Gromov-Hausdorff distance.
\end{enumerate}
\end{lem}
\begin{proof}
(1) is an immediate consequence of 
Proposition \ref{prop:Eint-ext}.

(2)(3)\,
In view of \eqref{eq:SigmaX=Union}, 
the condition (2) that $\Sigma_x(X)$ is $\CAT(1)$ implies 
that $\diam(\Sigma_{x,\pm}(\ca S))<o_p(x)$ and 
$\Sigma_x(X)$ is $o_p(x)$-close to a suspension
over finitely many points. 
\end{proof}

\begin{lem} \label{lem:infinitesimal-ray}
For every $x\in S_{ij}$ and every singular curve $C$ contained
in $S_{ij}$ with $x\in C$, $C$ is almost gradient-like for 
$d_x^{S_{ij}}$. Namely, setting
$C_{x,-}:=C\cap B(p,d^X_p(x))$ and 
$C_{x,+}:=C\cap (X\setminus \mathring{B}(p,d^X_p(x))$, 
we have
\[
          \frac{||x,y|_{S_{ij}} - |x,z|_{S_{ij}}|}{|y,z|_{S_{ij}}}>1-\tau_x(\max\{ |x,y|_{S_{ij}},|x,z|_{S_{ij}}\}),
\]
for all $y,z\in C_{x,\pm}$.
\end{lem}
\begin{proof}  
The conclusion follows immediately from 
 Corollary \ref{cor:Cijk2} and 
 \cite [Lemma 6.1]{NSY:local}. 
\end{proof}

\begin{lem} \label{lem:number=Nv-2}
$E(v_0)\cap \ca S$ can be written as a union of at most $N_{0}-2$ singular curves starting from $p$ and reaching $S(p,r)$. 
\end{lem}
\begin{proof} 
For $2\le k\le N_0$, let $E_k\subset E(v_0)$ be the union of all 
sector surfaces $S_{ij}$ at $p$ in the direction $v_0$ with $1\le i<j\le k$ as in \eqref{eq:S(v)}.
Let $\ca S(E_k)$ be the set of all points of 
$E_k$ having no neighborhoods homeomorphic to 
an open disk/an open half disk. 
We show that $\ca S(E_k)$ can be written as a union of at most $k-2$ singular curves
by the induction on $k$.
This is certainly true for $k=3$.
Suppose it for $k-1$.
Notice that $E_k$ is written as a union 
$E_k=E_{k-1}\bigcup_{C_k} F_k$ 
via a singular curve $C_k$ starting from $p$ and reaching $S(p,r)$, where 
$F_k$ is a component of $E_k\setminus \ca S(E_k)$ containing $\gamma_k$.
This completes the proof.
\end{proof}

\begin{lem} \label{lem:angle-xp-C}
For every 
$x\in E(v_0)\cap \ca S\setminus \{ p\}$, we have
\[
    \angle^X(\dot\gamma^X_{x,p}(0),\Sigma_x(\ca S))
       <o_p(x).
\]
\end{lem}
\begin{proof}
If the conclusion does not hold, there would be a sequence $x_n$ 
of $\ca S$ converging to $p$ such that for a positive constant $\theta$ independent of $n$,
\begin{align}  
&\angle^X(\dot\gamma^X_{x_n,p}(0),\Sigma_{x_n}(\ca S))> \theta.
               \label{eq:XE-radient}
\end{align}
Set $\gamma_n:=\gamma^X_{x_n,p}$.
Then,
$\gamma_n((0,t_1))\subset X\setminus \ca S$
for some $t_1>0$. 
Let us take such a maximal $t_1$, and suppose 
$t_1<|x_n,p|_X$.
Then $\gamma_n(t_1)\in C_1$ for some 
singular curve $C_1$ in $\ca S$.

Now one can take a $\CAT(\kappa)$-sector surface  $S_1:=S_{i_1j_1}$ containing $\gamma_n([0,t_1])$. 
By Lemma \ref{lem:exist-singular-curve}, choose a singular curve
$C_{i_1j_1\ell_1}$ containing $x$.
It follows from Lemma \ref{lem:almost-suspension},
\cite[Lemma 4.30]{NSY:local} and Lemma \ref{cor:vert} that 
\beq \label{eq:S1-Sigma}
\begin{aligned}
\angle^{S_1}&(\dot\gamma^{S_1}_{x,p}(0),  \Sigma_x(\ca S\cap S_1))  \\
& \le \angle^{S_1}(\dot\gamma^{S_1}_{x,p}(0),\Sigma_x(C_{i_1j_1\ell_1}))
+\angle^{S_1}(\Sigma_x(C_{i_1j_1\ell_1}),  \Sigma_x(\ca S\cap S_1)) \\
& \le\angle^{E_{i_1j_1\ell_1}}(\dot\gamma^{E_{i_1j_1\ell_1}}_{x,p}(0),  \Sigma_x(C_{i_1j_1\ell_1})) +o_p(x)
<o_p(x).
\end{aligned}
\eeq
 Therefore, considering a comparison triangle
of $\triangle^{S_1} px_n\gamma_n(t_1)$, we have 
$\angle^{S_1} x_n \gamma_n(t_1) p<\pi-\theta+\e_n$
with $\lim_{n\to\infty}\e_n=0$.
This implies 
\begin{align}  
&\angle^X(\dot\gamma^X_{\gamma_n(t_1),p}(0),\Sigma_{\gamma_n(t_1)}(\ca S)) > \theta-\e_n.
               \label{eq:XEradx}
\end{align}
 Namely,
$\gamma_n$ transversally meets $C_1$ at $t=t_1$ with angle $>\theta-\e_n$.
Note  that
$$
|d^{S_1}_p(\gamma_n(0))-d^{S_1}_p(\gamma_n(t_1))|<c(\theta)
       |\gamma_n(0), \gamma_n(t_1)|_X,
$$ where
$0<c(\theta)<1$ is a constant depending only 
on $\theta$.
If $\gamma_n(t_1)$ is on the same singular curve
containing $x_n$, this is a contradiction since  two points $y,z$ nearby $p$ in a singular curve from $p$, the ratio $\frac{|d^{S_1}_p(y)-d^{S_1}_p(z)|}{|y,z|_{S_1}}$
must be close to $1$ (see Lemma \ref{lem:infinitesimal-ray}).
Therefore  $C_1$ is different from the singular curve containing $x_n$.

We repeat this argument from 
$\gamma_n(t_1)$, and obtain $t_2>t_1$ that is maximal with respect to the property: 
$\gamma_n((t_1, t_2))\subset X\setminus \ca S$.
It follows  that
$$
|d^{S_2}_p(\gamma_n(t_1))-d^{S_2}_p(\gamma_n(t_2))|<c(\theta)
       |\gamma_n(t_1), \gamma_n(t_2)|_X,
$$
where  $S_2:=S_{i_2j_2}$ containing $\gamma_n([t_1,t_2])$.
Therefore  the singular curve $C_2$ containing 
$\gamma_n(t_2)$
 is different from $C_1$.

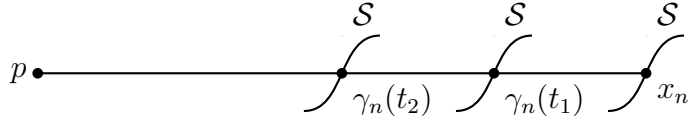
\begin{center}
\begin{tikzpicture}
[scale = 1]
\draw [-, thick] (-4,0) -- (4,0);
\draw [-, thick] (4.5,0.5) to [out=180, in=0] (3.5,-0.5);
\draw [-, thick] (2.5,0.5) to [out=180, in=0] (1.5,-0.5);
\draw [-, thick] (0.5,0.5) to [out=180, in=0] (-0.5,-0.5);
\fill (-4,0) coordinate (A) circle (2pt) node [left] {$p$};
\fill (4,0) coordinate (A) circle (2pt) node [below right] {$x_n$};
\fill (2,0) coordinate (A) circle (2pt) node [below right] {$\gamma_n(t_1)$};
\fill (0,0) coordinate (A) circle (2pt) node [below right] {$\gamma_n(t_2)$};
\fill (4,0.5) coordinate (A) circle (0pt) node [above right] {$\mathcal{S}$};
\fill (2,0.5) coordinate (A) circle (0pt) node [above right] {$\mathcal{S}$};
\fill (0,0.5) coordinate (A) circle (0pt) node [above right] {$\mathcal{S}$};
\end{tikzpicture}
\end{center}

After repeating this procedure at most $(N_{0}-2)$ times, we have $t_1<t_2<\cdots <t_T<|x_n,p|_X$ 
with $T\le N_{0} -2$
such that 
\begin{itemize}
\item $\gamma_n$ transversally meets $\ca S$ at each $t=t_j$ with angle $>\theta-j \e_n\,:$ 
\item if $C_j$  denotes singular curves containing $\gamma_n(t_j)$, then 
$C_i\neq C_j$ for all $i\neq j\,;$
\item $\gamma_n((t_T, |x_n,p|_X))$ never meets $\ca S$. 
\end{itemize}
If $S_T$ is a sector surface at $p$  containing 
$\gamma_n([t_T, |x_n,p|_X])$, then 
$\gamma_n|_{[t_T, |x_n,p|_X]}$ is an $S_T$-geodesic with
\[
\angle^{S_T}
(\dot\gamma^{S_T}_{\gamma_n(t_T),p},
\Sigma_{\gamma_n(t_T)}(\ca S)> \theta/2.
\]
This is a contradiction.
\end{proof}

\psmall
\begin{lem} \label{lem:infinitesimal-rayX}
For every singular curve $C$ starting from $p$, we have
\[
          \frac{||p,x|_X - |p,y|_X|}{|x,y|_X}>1-\tau_p(\max\{ |p,x|_X,|p,y|_X\}),
\]
for all $x,y\in C$. 
\end{lem}
\begin{proof} 
We assume $|p,x|_X < |p,y|_X$, and consider  the Lipschitz function $f(t):=d^X(p,C(t))$, where  $C(t)$ is an arc-length parameter from $x$ to $y$. 
Assume that $f$ is differentiable at $t_0$, and set
$$
\theta_-(t_0):=\angle^X(\dot\gamma^X_{C(t_0),p}(0), C_-(t_0)), \quad 
\theta_+(t_0):=\angle^X(\dot\gamma^X_{C(t_0),p}(0), C_+(t_0)),
$$
where $C_-(t_0)=C\cap B(p, d_p^X(C(t_0)))$,
$C_+(t_0)=C\cap X\setminus {\rm int}\,B(p, d_p^X(C(t_0)))$.
Lemma \ref{lem:angle-xp-C} shows that 
$\theta_-(t_0)<o_p(y)$ and $\theta_+(t_0) >\pi-o_p(y)$.
Using \cite[Lemma II.2]{AZ:bddcurv},
we have the following estimates for the left and right derivatives of $f$:
$$
  f_+'(t_0)\le -\cos \theta_+(t_0)\le1-o_p(y),\,\,
  f_-'(t_0)\ge \cos \theta_-(t_0)\ge1-o_p(y).
$$
The conclusion follows immediately.
\end{proof}

\psmall

The following lemma 
gives a new information on the local geometry 
of two-dimensional locally compact 
geodesically complete locally $\CAT(\kappa)$-space, which was not obtained in 
\cite{NSY:local}.

\begin{lem} \label{lem:ab-thin-ruled}
There is $0<r_1\le r$ such that 
if $t, t'\in (0,r_1]$, then $\gamma^X_{\gamma_i(t),\gamma_j(t')}$ is contained in $S_{ij}$.
\end{lem} 
\begin{proof}
We put  $S:=S_{ij}$ for simplicity.
If the lemma does not hold, then there are sequences 
$t_n\to 0$ and $t'_n\to 0$ such that  
$\mu_n:=\gamma^X_{\gamma_i(t_n),\gamma_j(t'_n)}$
are not contained in the surface $S$ for all $n$.
We may assume $t_n\ge t_n'$. 
Let $s_n$ be the largest $s\in
(0,|\gamma_i(t_n),\gamma_j(t_n')|)$ such that $\mu_n([0,s])\subset S$.
Set $y_n:=\mu_n(s_n)\in \ca S$.
Since $\lim_{n\to\infty}\angle y_n p \gamma_i(t_n)=\delta$, we have
\begin{align}\label{eq:angle Sayp>}
\tilde\angle^S py_n\gamma_i(t_n)<
\pi -\delta+\e_n,
\end{align}
with $\lim_{n\to\infty} \e_n=0$. 
If there is a sequence in 
$\ca S\setminus \{ y\}$ converging to $y_n$, then 
$\dot\mu(s_n)\in\Sigma_{y_n}(\ca S)$,
which contradicts Lemma \ref{lem:angle-xp-C}.
Therefore we have a point 
$z_n\in\mu_n\cap\ca S$ such that the segment $\sigma_n:[0,|y_n,z_n|_X]\to X$ of $\mu_n$  from $y_n$ to $z_n$ is away from $S$ except the endpoints.

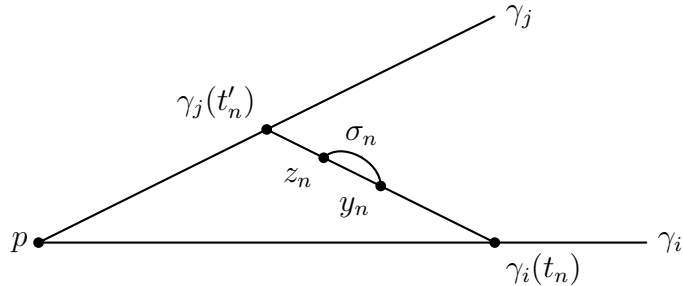
\begin{center}
\begin{tikzpicture}
[scale = 1]
\draw [-, thick] (0,0) -- (8,0);
\draw [-, thick] (0,0) -- (6,3);
\draw [-, thick] (3,1.5) -- (6,0);
\fill (6,0) coordinate (A) circle (2pt) node [below right] {$\gamma_i(t_n)$};
\fill (3,1.5) coordinate (A) circle (2pt) node [above left] {$\gamma_j(t_n')$};
\fill (4.5,0.75) coordinate (A) circle (2pt) node [below left] {$y_n$};
\fill (3.75,1.125) coordinate (A) circle (2pt) node [below left] {$z_n$};
\fill (0,0) coordinate (A) circle (2pt) node [left] {$p$};
\fill (8,0) coordinate (A) circle (0pt) node [right] {$\gamma_i$};
\fill (6,3) coordinate (A) circle (0pt) node [right] {$\gamma_j$};
\draw [-, thick] (4.5,0.75) to [out=90, in=45] (3.75,1.125);
\fill (4.25,1.15) coordinate (A) circle (0pt) node [above] {$\sigma_n$};
\end{tikzpicture}
\end{center}

For each $s\in [0,|y_n,z_n|_X]$,
consider the tree
$T(p, s):=S^X(p, d_p^X(\sigma_n(s)))\cap E(v_0)$.
The path in $T(p,s)$ starting from the point 
$\sigma_n(s)$ 
meets $S$ at a unique point, say $\zeta_n(s)$, which provides
a singular curve $\zeta_n=C_{\sigma_n}$ 
between $y_n$ and $z_n$.

By Lemmas \ref{lem:infinitesimal-ray} and 
\ref{lem:infinitesimal-rayX},
\begin{align}\label{eq:Ca-almost-grad}
\text{$C_{\sigma_n}$ is almost gradient for $d^S_p$ and $d^X_p$}.\hspace{4cm}
\end{align}
This implies 
$\frac{L(C_{\sigma_n})}{|y_n,z_n|_S}<1+\e_n$
and 
$\frac{L(C_{\sigma_n})}{|y_n,z_n|_X}<1+\e_n$,
and therefore, 
\begin{align}\label{eq:quat=yz}
(1-\e_n)|y_n,z_n|_S < |y_n,z_n|_X.
\end{align}
Let $w_n$ be the point of $[\gamma_i(t_n),y_n]\subset \mu_n$ with
$|w_n,y_n|_S=|y_n,z_n|_S$.
From \eqref{eq:quat=yz}, we have 
\begin{align*}
|w_n,z_n|_S&
\ge |w_n,z_n|_X=|w_n,y_n|_X+|y_n,z_n|_X \\
& \ge |w_n,y_n|_S+(1-\e_n)|y_n,z_n|_S
    = (2-\e_n)|w_n,y_n|_S.
\end{align*}
On the other hand, \eqref{eq:angle Sayp>}  implies
$\tilde\angle^S z_n y_n w_n<\pi -\delta/2$
yielding a contradiction to the above
inequality. 
\end{proof}

\psmall
From now on, by Lemma \ref{lem:ab-thin-ruled},
taking smaller $r=r_p>0$ if necessary,
we assume that 
\begin{align} \label{eq:gammatt'S}
\gamma^X_{\gamma_i(t),\gamma_j(t')}\subset S_{ij},
\end{align}
for all $t, t'\in (0,r]$ and for all $1\le i<j\le N$.

As an immediate consequence of \eqref{eq:gammatt'S}, we have

\begin{prop} \label{prop:thin-ruled}
Every sector surface $S_{ij}$ is a thin ruled surface as described in \cite[Section 4]{NSY:local}.
Namely, $S_{ij}$ is formed by 
$\gamma^X_{\gamma_i(t),\gamma_j(t)}$ for all 
$0\le t\le r$.

Moreover,  $\gamma^X_{\gamma_i(t),\gamma_j(t')}$
is uniquely determined for all $t, t'\in (0,r]$.
\end{prop} 

Proposition  \ref{prop:thin-ruled} shows that 
the sector surface $S$ in the condition (3) is essentially uniquely determined.

\pmed

For any fixed $x\in E(v_0)\cap \ca S$ and for every  $v\in V(\Sigma_x(X))$,
choose $\delta=\delta_x>0$  and 
$\e=\e_x>0$ such that $\e\ll \delta<\min\{ \angle(v, w)\,|\,w\neq v \in V(\Sigma_x(X))\,\}/10$.
Let 
\begin{align}\label{eq:E(x,v)}
   E_x(v)=E_x(v,\delta,\e),
\end{align}
be defined as in \eqref{eq:S(v)} for $x$ in place of $p$.
More precisely, 
let $m:=N_{v}$ be the branching number of the graph $\Sigma_x(X)$ at $v$, and 
choose $\nu_1,\ldots,\nu_m\in\Sigma_x(X)$ such that  for all $1\le i\neq j\le m$
\begin{itemize}
\item \text{ $\angle(\nu_i, v) = \delta$ and
$\angle(\nu_i, \nu_j) = 2\delta\,;$}
\item  $\CAT(\kappa)$-sector surfaces $R_{ij}(x)$ at $x$  provided by the condition (3) for the arc 
$[\nu_i,\nu_j]\subset \Sigma_x(X)$ are
all properly embedded in $B(x,\e)$.
\end{itemize}
Here we may assume that
\begin{align}\label{eq:BxsubsetE}
    B(x,\e)\subset \mathring{E}(v_0).
\end{align}
Set $\sigma_i:=\gamma^X_{\nu_i}$.
Then $E_x(v)$ is the domain bounded  by
$\sigma_1, \ldots,\sigma_m$ and $S(x,\e)$.
In other words, 
$$
   E_x(v)=\bigcup_{1\le i<j\le m} R_{ij}(x).
$$

Since $x$ is fixed, 
in what follows, we simply write as 
$R_{ij}:=R_{ij}(x)$.

\begin{lem} \label{lem:SE-geom-distance}
For each $R_{ij}$ and for every $y\in R_{ij}\setminus \{ x\}$, we have 
\begin{align*}
     \frac{|x,y|_{R_{ij}}}{|x,y|_X} <1+\tau_x(|x,y|_X).
\end{align*}
\end{lem}
\begin{proof}
The proof is identical with \cite[Sublemma 4.31, Lemma 7.10]{NSY:local}.
\end{proof}

\begin{lem} \label{lem:SE-geom}
For a fixed  $x\in E(v_0)$, we have for every $y\in R_{ij}\setminus \{ x\}$
\begin{align*}
    & (1)\,\,  \angle^X \left(\dot\gamma^{R_{ij}}_{x,y}(0), 
         \dot\gamma^X_{x,y}(0)\right) <\tau_x(|x,y|_X),\\
   & (2)\,\,  \angle^X \left(\dot\gamma^{R_{ij}}_{y,x}(0), 
         \dot\gamma^X_{y,x}(0)\right) <\tau_x(|x,y|_X).
\end{align*}
\end{lem}
\begin{proof}

We assume $x\in \mathring{E}(v_0)$.
The other case is similar.
If (1) does not hold, there would be a sequence $x_n$ 
of $R_{ij}\setminus \{ x\}$ converging to $x$ such that for a uniform positive constant $c$,
\begin{align*}  
&\angle^X(\dot\gamma^{R_{ij}}_{x,x_n}(0), \dot\gamma^X_{x,x_n}(0)) > c.
\end{align*} 
Suppose $\gamma^X_{x,x_n}\setminus \{ x\}$ does not meet $\ca S$ for some $n$. Then by continuity, $\gamma^X_{x,x_n}\subset R_{ij}$, 
and hence
$\gamma^X_{x,x_n}=\gamma^{R_{ij}}_{x,x_n}$. This is a contradiction.
Therefore we may assume that 
 $\gamma^X_{x,x_n}\setminus \{ x\}$ meets $\ca S$ for each $n$. 
Passing to a subsequence, we may assume that 
$\dot\gamma^X_{x,x_n}(0)$ (resp. $\dot\gamma^{R_{ij}}_{x,x_n}(0)$)
converges to a direction 
$\xi^X\in V(\Sigma_x(X))$
(resp. $\xi^{R_{ij}}\in\Sigma_x(R_{ij})$).
From the proof of Proposition \ref{prop:Eint-ext},
we see that  $\xi^{R_{ij}}= \xi^X$,
which is a contradiction.

If (2) does not hold, there would be a sequence $x_n$ 
of $E(v_0)$ converging to $x$ such that for a positive constant $\theta$ independent of $n$,
\begin{align}  &\angle^X(\dot\gamma^{R_{ij}}_{x_n,x}(0), \dot\gamma^X_{x_n,x}(0)) > \theta.
               \label{eq:XE-gradx}
\end{align}
Suppose $x_n\in \ca S$ for any large $n$.
By Lemma \ref{lem:angle-xp-C}, we have 
\begin{align}\label{eq:angle-ES>theta}
\angle^X(\dot\gamma^{R_{ij}}_{x_n, x}(0), \Sigma_{x_n}(\ca S))>\theta-\e_n.
\end{align}
In the proof of Lemma \ref{lem:angle-xp-C}, 
replacing $\gamma^X_{x_n,x}$ by 
$\gamma^{R_{ij}}_{x_n,x}$,
we can proceed as in the proof of Lemma \ref{lem:angle-xp-C}
to obtain a contradiction.
  
Suppose $x_n\in X\setminus \ca S$, and choose
a sector surface 
$R_n:=R_{i_nj_n}$ containing $x_n$. 
From \eqref{eq:XE-gradx}, we may assume that
$\gamma^X_{x_n,x}$ first leave $R_n$ at a point 
$y_n\in R_n\cap\ca S$. 
By Lemma \ref{lem:angle-xp-C},
$\dot\gamma^X_{y_n,x}(0)$ makes an angle 
with $\ca S$ less than $\e_n$.
Passing to a subsequence, we may assume that 
$R_n$ is a fixed sector surface, say $R$ containing $x_n$ for all $n$. Thus, 
\eqref{eq:S1-Sigma} implies that 
$\dot\gamma^{R}_{y_n,x}(0)$ also makes an angle 
with $\ca S$ less than $\e_n$.
Thus we have 
$
\angle^X(\dot\gamma^X_{y_n,x}(0),\dot\gamma^{R}_{y_n,x}(0))<\e_n.
$
Together with Lemma \ref{lem:comparison},
this implies
\beq
\begin{aligned}
\angle^{R} xx_ny_n &\le \tilde\angle^{R} xx_ny_n\\
&\le\pi- \tilde\angle^{R} x_ny_n x-\tilde\angle^{R} y_nx x_n +  \e_n \\
&<\pi- \angle^{R} x_ny_n x <3\e_n.
\end{aligned}
\eeq
It follows from \eqref{eq:XE-gradx} that  
$$
\angle^X(\dot\gamma^{R_{ij}}_{x_n,x}(0),\dot\gamma^{R}_{x_n,x}(0))
>\theta-3\e_n.
$$
Let $\gamma^{R_{ij}}_{x_n,x}$ first leaves $R$ at a point 
$z_n\in R\cap \ca S$. Then we have
\begin{align*}
\angle^X(\dot\gamma^{R_{ij}}_{z_n,x}(0), &\dot\gamma^{R}_{z_n,x}(0)) > \pi - \angle^{R} x z_nx_n -\e_n \\
&\ge \pi - (\pi - \angle^{R} x x_n z_n +\e_n) -\e_n
   \ge \theta-5\e_n,
\end{align*}
which implies 
$$
\angle^X(\dot\gamma^{R_{ij}}_{z_n, x}(0), \Sigma_{z_n}(\ca S))>\theta-6\e_n.
$$
As in \eqref{eq:angle-ES>theta},  
in the same manner as the proof of Lemma 
\ref{lem:angle-xp-C},
we  obtain a contradiction.
This completes the proof.
\end{proof}

\begin{center}
\begin{tikzpicture}
[scale = 1]
\draw [-, thick] (0,0) -- (8,0);
\fill (0,0) coordinate (A) circle (2pt) node [left] {$x$};
\fill (8,0) coordinate (A) circle (2pt) node [right] {$x_n$};
\fill (6,0) coordinate (A) circle (2pt) node [below left] {$y_n$};
\fill (4.86,1.14) coordinate (A) circle (2pt) node [above right] {$z_n$};
\draw [-, thick] (0,0) to [out=30, in=150] (8,0);
\draw [-, thick] (6,0) -- (4.5,1.5);
\draw [-, thick] (6,0) -- (6.25,-0.25);
\fill (6.15,-0.15) coordinate (A) circle (0pt) node [below right] {$\mathcal{S}$};
\fill (4,0) coordinate (A) circle (0pt) node [below] {$\gamma_{x_n,x}^X$};
\fill (6.25,0.8) coordinate (A) circle (0pt) node [above] {$\gamma_{x_n,x}^{R_{ij}}$};
\end{tikzpicture}
\end{center}

\begin{rem} \label{rem:S(t)capS} \upshape
Lemma \ref{lem:SE-geom}(2) implies that 
$S(p,t)\cap S_{ij}$ is an arc in the situation of 
 Lemma \ref{lem:metric-circle}.
\end{rem}

Next, we investigate the geodesic extension property of $X$.

\begin{lem} \label{lem:geod-extension}
For arbitrary distinct points $y,z\in E(v_0)\setminus \{ p\}$, let  
$\gamma:=\gamma^X_{y,z}:[a,b]\to X$
be any $X$-minimal geodesic between them such that 
\begin{align} \label{eq:anglexxp}
\angle^X(\dot\gamma^X_{y,p}(0), \dot\gamma(0))
\in (\pi/10, 9\pi/10).
\end{align}
Then $\gamma$
extends to an $X$-minimal geodesic $\hat\gamma$ in the 
both directions, and it reaches some 
$\gamma_i$ and $\gamma_j$ \,$(1\le i\neq j\le N_0)$.
Furthermore, we have $\hat\gamma\subset S_{ij}$.
\end{lem}
\begin{proof}
In what follows, we assume $y,z\in\ca S$.
The other cases are similarly discussed.
From \eqref{eq:anglexxp}, we can take 
$y_1\in\gamma\cap\ca S\setminus \{ y\}$
such that the open subarc $(y,y_1)$ of $\gamma$ does not meet
$\ca S$.
Take a $\CAT(\kappa)$-sector surface $S$ containing $[y,y_1]$.
Since $|\angle^S py_1y+ \angle^S pyy_1-\pi|
 < \tau_p(r)$,
 Lemma \ref{lem:SE-geom} implies
$$
\angle^X(\dot\gamma^X_{y_1,p}(0), 
\dot\gamma_{y_1,z}(0)) 
\in  (\pi/10-\tau_p(r), 9\pi/10+\tau_p(r)).
$$
Therefore we can repeat the same argument
for $\gamma_{y_1,z}$ in place of $\gamma$.
From Lemma \ref{lem:number=Nv-2},
we obtain 
$$
   \sharp (\gamma\cap\ca S)\le N_0-2.
$$
Let $y=y_0<y_1<\cdots<y_K=z$
be the elements of $\gamma\cap\ca S$
with $K\le N_0-2$.
As in the previous argument based on Lemma \ref{lem:SE-geom}, taking smaller $r$ if necessary, 
we have for each $1\le k\le K$
\begin{align*}
&\angle^X py_ky_{k-1}, \quad \angle^X py_{k-1}y_{k}
\in (\pi/11, 10\pi/11), \\
&|\angle^X py_ky_{k-1}+ \angle^X py_{k-1}y_{k}-\pi|
 < \tau_p(r).
\end{align*}
We show the following by the inducion on $k$:
\begin{itemize}
\item
$[y,y_k]\subset S_{ij}\,;$ 
\item  $[y,y_k]$ extends to 
an $X$-minimal geodesic joining points of 
$\gamma_i$ and $\gamma_j$.
\end{itemize}
First we may assume that $[y,y_1]\subset S_{ij}$
for some $i<j$. Extend  $[y,y_1]$ in $S_{ij}$
in the both directions. Then it reaches 
$\gamma_i$ and $\gamma_j$.
Then the extension $\hat\gamma$ is  $S_{ij}$-shortest, and hence is $X$-shortest by Lemma \ref{lem:ab-thin-ruled}.
Suppose that $[y,y_{k-1}]\subset S_{ij}$
and an $S_{ij}$-geodesic extension of $[y_{k-1},y]$ beyond $y$ 
reaches a point  $a$ of $\gamma_i$.
Suppose also $[y_{k-1},y_k]$ is contained in some $\CAT(\kappa)$-sector surface $S_{mn}$
and an $S_{mn}$-geodesic extension of $[y_{k-1},y_k]$ beyond $y_k$ 
reaches a point $b$ of $\gamma_m$.
Note that the arcs $[a,y_{k-1}]$ and $[y_{k-1},b]$
meet at $y_{k-1}$ with angle $\pi$ in the 
$\CAT(\kappa)$-space $E_{ijm}$
(Corollary \ref{cor:Cijk2}),
transversally to $C_{ijm}$.
This implies  $[y,y_{k}]\subset S_{im}$.
The case $k=K$ completes the proof.
\end{proof} 

Next we show the compatibility of the structure of the sector surfaces in $E_x(v)$ and those in $E(v_0)$.

\begin{lem} \label{lem:Sx=S}
For arbitrary $1\le i<j\le m$, there are
$1\le i'<j'\le N_{0}$ such that 
$R_{ij}(x)\subset S_{i'j'}$.
\end{lem}
\begin{proof}
Fix $t_0\in (0,\e]$, and  
consider an 
$X$-minimal geodesic $\gamma$
joining $\sigma_i(t_0)$ to $\sigma_j(t_0)$.
Using Lemma \ref{lem:geod-extension},
extend $\gamma$ in the both directions as an $X$-geodesic $\hat\gamma$ until $\hat\gamma$ 
reaches $\gamma_{i'}$ and $\gamma_{j'}$ for some 
$1\le i'<j'\le N_{0}$.
Lemma \ref{lem:geod-extension} also shows that
$\hat\gamma$ is $X$-shortest and is contained in 
$S_{i'j'}$. 
This completes the proof.  
\end{proof}

 \psmall
\subsection{Proof of Theorem \ref{thm:converse}}
\label{ssec:Proof-Thm1.4}
 
  From Lemmas   \ref{lem:Cijk} and 
  \ref{lem:number=Nv-2} (see also  Lemma \ref{lem:infinitesimal-rayX}),  $\ca S$ has the graph structure 
 in the sense of \cite[Definition 6.7]{NSY:local}
 whose branching  orders at vertices are locally 
 uniformly bounded.

\psmall \n
{\bf Setting}.\,
 By Lemmas \ref{lem:almost-suspension}(3)
and \ref{lem:angle-xp-C}, we may assume 
the inequalities corresponding to  \eqref{eq:10-10} 
for all $x\in \ca S\cap (B(p, r)\setminus \{ p\})$,

We set 
$r(x):=d^X_p(x)$ for simplicity.
\psmall

The basic idea of the proof of Theorem \ref{thm:union-CAT} is to approximate 
$E(v_0)$ by a polyhedral $\CAT(\kappa)$-space
$\tilde E(v_0)$ as in Section \ref{sec:Polyhedral}
using the surgery method (see \cite[Section 7]{NSY:local} for the details).

Let $\e_0$ be any positive number.
We fix any $x\in E(v_0)\cap V_{\sing}(\ca S)$ for a moment, and assume that the singularity of $x$ occurs from 
the positive direction. Namely, there is $v\in \Sigma_{x,+}^{\rm sing}(\ca S)$.
The other case $v\in \Sigma_{x,-}^{\rm sing}(\ca S)$ is similarly discussed.
Let $\ca S(v)$ denote the union of singular curve in $\ca S$  starting at $x$ in the direction $v$.
Let $m$ be the branching number of $\Sigma_x(X)$ at $v$.
Let $E_x(v)$ be as in \eqref{eq:E(x,v)}.
Lemma \ref{lem:metric-circle} implies that 
$E_x(v)\cap S(p, r(x)+\e)$ is a tree, say $T_0(x,v)$. 
Replacing each edge $e_0$ of $T_0(x,v)$ by an $X$-geodesic segment $e$ between the endpoints,  we obtain the $X$-geodesic tree $T(x,v)$. 
By Lemma \ref{lem:geod-extension}, $e$ is 
uniquely determined.
Let $K(x,v)$ be a closed domain of $E_x(v)$  bounded by 
$T(x,v)$ and the $X$-geodesic segments
$\sigma_1,\ldots,\sigma_m$
between $x$ and the endpoints of the tree $T(x,v)$.

Choosing smaller  $\delta=\delta_x>0$  and 
$\e=\e_x>0$ if necessary,  we may assume the following:
\beq  \label{eq:non-meeting-VC}
\begin{cases}
\begin{aligned}
 &\text{$\{ B^{\Sigma_x(X)}(v,2\delta)\}$\, 
 $(v\in\Sigma_x^{\sing}(\ca S))$ is mutually disjoint\,$;$} \\
 &\text{$E_x(v)$
  covers 
 $\ca S(v)\cap B_+(x,\e)\,;$} \\
 &\text{$\tilde\angle^{X} yxy' <\e_0$ for all $y,y'\in
    T(x,v)$\,;} 
    \\
 &\text{$E_x(v)\cap S(p,r(x)+\e)$ does not meet $V(\ca S)$,}     
\end{aligned}
\end{cases}
\eeq
where $B_+(x,\e):=B(x,\e)\setminus\mathring{B}(p,r(x))$.

\pmed
\n
{\bf Angle comparison and intermediate comparison objects}.
\pmed
In what follows, we set $T:=T(x,v)$.

For each edge 
$e\in E(T)$, let 
$\triangle(x,e)$ denote an $X$-geodesic triangle with vertices $x$ and $\pa e$,
and let $\tilde\blacktriangle(x,e)$ be the  comparison triangle region on
$M^2_\kappa$ for $\triangle(x,e)$.

\begin{prop}\label{prop:compar=Delata(x,e)}
The triangle comparison holds for
$\triangle(x,e)$ for any $e\in E(T)$.
\end{prop}

\begin{proof}
In what follows,  to prove Proposition \ref{prop:compar=Delata(x,e)}, we 
show that $\triangle(x,e)$ spans
a ``$\CAT(\kappa)$-disk domain''
in a generalized sense.

\pmed
\n
{\bf Case A)}. $e\in E_{\rm ext}(T)$.
\psmall
First consider the case when $e$ is an 
exterior edge $e\in E_{\rm ext}(T)$ with
$\{ y,z\}:=\pa e$, $z\in \pa T$.
Let $\gamma:=\gamma^X_{x,y}$.

\begin{lem} \label{lem:z*S}
The $X$-join $S:=z*\gamma$ is well-defined
and is a $\CAT(\kappa)$-disk.
\end{lem}

\begin{proof}
Let us assume $z\in\sigma_m$.
For each $w\in\gamma\setminus \{ x\}$,
choose $1\le i\le m-1$ such that $w\in R_{im}$.
Let $\mu$ be an $R_{im}$-geodesic from 
$z$ through $w$ that finally reaches 
a point of $\sigma_i$.
By Lemma \ref{lem:ab-thin-ruled},
$\mu$ is $X$-shortest.
This shows the first half of the conclusion. 

As in Lemma \ref{lem:Si=CAT},
consider a subdivision $\Delta$ of $\gamma$ such that there exists an associated $\CAT(\kappa)$-disk $S_\Delta$ approximating $S$. 
In a way similar to Lemma \ref{lem:Si=CAT},
letting $|\Delta|\to 0$, 
we conclude that $S$ is a $\CAT(\kappa)$-disk.
\end{proof}

By Lemma \ref{lem:z*S}, we have
\[
\angle^X xzy\le\angle^S xzy\le
\tilde\angle^S xzy= \tilde\angle^X xzy.
\]
Similarly we have 
$\angle^X xyz\le \tilde\angle^X xyz$ and 
$\angle^X yzx\le \tilde\angle^X yzx$. 

\pmed
\n
{\bf Case B)}. $e\in E_{\rm int}(T)$.
\psmall
Next we consider the case of $e\in E_{\rm int}(T)$.
In this case, 
we subdivide 
$\triangle(x,e)$ as follows:
Let  $\{ y,y'\}:=\pa e$, and 
set $\gamma:=\gamma^X_{x,y}$ and $\gamma':=\gamma^X_{x,y'}$.
Let $0\le t_*<|x,y|$ be the largest $t$ satisfying  
$\gamma(t)=\gamma'(t)$. 
Note that we do not yet know whether  
$\gamma([0,t_*])=\gamma'([0,t_*])$ holds.
Let $x_*:=\gamma(t_*)=\gamma'(t_*)$, $\ell:=|x_*,y|$ and 
$\ell'_*=|x_*,y'|$. 

Now we choose (possibly infinite)
decompositions of both 
$[t_*,|x,y|]$ and $[t_*,[x,y'|]$:
\begin{align*}
& t_*<\cdots<t_i<t_{i-1}<\cdots t_1 <t_0=|x,y|, \\
& t_*<\cdots<t_i'<t_{i-1}'<\cdots <t'_1<t_0'=|x,y'|,
\end{align*} 
with $\lim_{i\to\infty} t_i=\lim_{i\to\infty} t_i'=t_*$ 
such that letting $u_i:=\gamma(t_i)$, $u'_i:=\gamma'(t'_i)$,
we have for 
every $i\ge 1$, 
\beq
\begin{aligned}\label{eq:condition=decompose2}
\begin{cases}
 & |\angle u_{i-1}u_{i-1}' u'_{i}-\pi/2|<\pi/4, \quad 
    |\angle u_{i-1}u_{i}' u'_{i-1}-\pi/2|<\pi/4, \\
  & |\angle u_{i}'u_{i-1} u_{i}-\pi/2|<\pi/4, \quad 
     |\angle u_{i}' u_i u_{i-1}-\pi/2|<\pi/4.\\ 
\end{cases}
\end{aligned}
\eeq
By Lemma \ref{lem:geod-extension}, we can define 
the $X$-geodesic joins $\blacktriangle_i:=u_{i-1}*[u'_{i-1},u'_i]$ and 
$\blacktriangle'_i:=u'_i*[u_{i-1},u_{i}]$\,$(i\ge 1)$
by $X$-geodesics, where $[u_{i-1},u_{i}]$ and 
$[u'_{i-1},u'_i]$ are the subarcs of $\gamma$ and 
$\gamma'$ respectively. 
As in Case A), we see that 
$\blacktriangle_i$ and 
$\blacktriangle'_i$ are $\CAT(\kappa)$-disks.
Let $\blacktriangle_*(x_*,e)$
 be the closure of the gluing of 
of the union 
$$
\blacktriangle_{1}\cup
\blacktriangle'_{1}\cup\cdots\cup
\blacktriangle_{i}
\cup
\blacktriangle'_{i}\cup\cdots
$$
in an obvious way.
Note that 
$\blacktriangle_*(x_*,e)$
is a $\CAT(\kappa)$-disk.
Let  $\angle^{\blacktriangle_*(x_*,e)}
  y  x_* y'$
denote the angle of 
$\blacktriangle_*(x_*,e)$
at the point corresponding to $x_*$.
Since $\blacktriangle_*(x_*,e)$
is $\CAT(\kappa)$,
from construction we have
\begin{align} \label{eq:angle(yx*y)}
 \angle^X y x_* y' 
\le \angle^{\blacktriangle_*(x_*,e)}
  y  x_* y'
 \le \tilde\angle y x_* y'.
\end{align}

\begin{center}
\begin{tikzpicture}
[scale = 1]
\draw [-, thick] (-1.5,0) -- (3,0);
\draw [-, thick] (3,0) -- (8,1.5);
\draw [-, thick] (3,0) -- (8,-1.5);
\draw [-, thick] (8,1.5) -- (8,-1.5);
\draw [-, thick] (4,0.3) -- (4,-0.3);
\draw [-, thick] (5,0.6) -- (5,-0.6);
\draw [-, thick] (6,0.9) -- (6,-0.9);
\draw [-, thick] (7,1.2) -- (7,-1.2);
\draw [-, thick] (5,-0.6) -- (4,0.3);
\draw [-, thick] (6,-0.9) -- (5,0.6);
\draw [-, thick] (7,-1.2) -- (6,0.9);
\draw [-, thick] (8,-1.5) -- (7,1.2);
\fill (-1.5,0) coordinate (A) circle (2pt) node [below] {$x$};
\fill (3,0) coordinate (A) circle (2pt) node [below] {$x_{\ast}$};
\fill (4,0.3) coordinate (A) circle (2pt) node [above] {};
\fill (5,0.6) coordinate (A) circle (2pt) node [above] {$u_i'$};
\fill (6,0.9) coordinate (A) circle (2pt) node [above] {$u_{i-1}'$};
\fill (7,1.2) coordinate (A) circle (2pt) node [above] {$u_1'$};
\fill (8,1.5) coordinate (A) circle (2pt) node [above] {$y'$};
\fill (4,-0.3) coordinate (A) circle (2pt) node [above] {};
\fill (5,-0.6) coordinate (A) circle (2pt) node [below] {$u_i$};
\fill (6,-0.9) coordinate (A) circle (2pt) node [below] {$u_{i-1}$};
\fill (7,-1.2) coordinate (A) circle (2pt) node [below] {$u_1$};
\fill (8,-1.5) coordinate (A) circle (2pt) node [below] {$y$};
\fill (8,0) coordinate (A) circle (0pt) node [right] {$e$};
\fill (4,-1.2) coordinate (A) circle (0pt) node [below] {$\triangle(x,e)$};
\fill (5.65,0.33) coordinate (A) circle (0pt) node [] {$\blacktriangle_i$};
\fill (5.25,-0.33) coordinate (A) circle (0pt) node [] {$\blacktriangle_i'$};
\filldraw[fill=gray, opacity=.1] 
(5,0.6) -- (6,0.9) -- (6,-0.9) -- (5,-0.6) -- (5,0.6);
\end{tikzpicture}
\end{center}

In what follows, we show that 
$\gamma=\gamma'$ on the segment 
$[x,x_*]$. Otherwise, we have 
$t_1^*<t_2^*$ in $[0,t_*]$ such that 
$x_i^*:=\gamma(t_i^*)=\gamma'(t_i^*)$ \,$(i=1,2)$
such that $\gamma((t_1^*,t_2^*))$ does not 
meet $\gamma'((t_1^*,t_2^*))$.
In a way similar to the above
construction of 
$\blacktriangle_*(x_*,e)$,
we obtain a $\CAT(\kappa)$-disk
$\blacktriangle_*(x_1^*,x_2^*)$
bounded by two concave polygonal
curves joining $x_1^*$ and 
$x_2^*$.
Since these curves are geodesics in 
$\blacktriangle_*(x_1^*,x_2^*)$,
this is a contradiction.

Now we define
$$
\blacktriangle_*(x,e)
:=[x, x_*]\cup
\blacktriangle_*(x_*,e).
$$
Since $\blacktriangle_*(x,e)$ is $\CAT(\kappa)$,
it follows that 
\begin{align} \label{eq:angle(yxy)}
 \angle^X y x y' 
\le \angle^{\blacktriangle_*(x,e)}
  y  x y'
 \le \tilde\angle y x y'.
\end{align}
Similarly, we have 
$\angle^X  xy y'\le  \tilde\angle  xy y'$
and $\angle^X  xy' y\le  \tilde\angle  xy' y$.
This completes the proof of Proposition
\ref{prop:compar=Delata(x,e)}.
\end{proof}

\pmed
Gluing all the triangle regions in 
$\{ \tilde\blacktriangle(x,e)\,|\,e\in 
E(T)\}$ naturally, we obtain 
a polyhedral space $\tilde K(x,v)$ corresponding to 
$K(x,v)$. In $E(v_0)$, we do a surgery by removing 
$\mathring{K}(x,v)$ from $E(v_0)$, and gluing 
$E(v_0)\setminus \mathring{K}(x,v)$ and $\tilde K(x,v)$ along their isometric boundaries to get a new space, say 
$\tilde W_{x,v}$.  

\begin{lem}\label{lem:decompose-triangle}
For each vertex $\tilde y$ of $\tilde K(x,v)$,
$\Sigma_{\tilde y}(\tilde W_{x,v})$ is $\CAT(1)$.
\end{lem} 

\begin{proof}
The conclusion is immediate for any $\tilde y=\tilde z$ with $z\in \pa T$
from Proposition \ref{prop:compar=Delata(x,e)}
since $\Sigma_z(X)$ is $\CAT(1)$ from the 
condition (2).
 
If $y\in V_{\rm int}(T)$,
from the condition \eqref{eq:non-meeting-VC},  $\Sigma_y(X)$ is topologically a
suspension whose vertices consist of the directions of $\ca S$ at $y$.
From Proposition \ref{prop:compar=Delata(x,e)},
it is easily verified that there is an 
expanding map $\Sigma_y(K(x,v))\to \Sigma_{\tilde y}(\tilde K(x,v))$ for each
$y\in V(T)$, which implies that 
$\Sigma_{\tilde y}(\tilde W_{x,v})$ is 
$\CAT(1)$ (see \cite[Claim 7.15]{NSY:local}).

For the point $x$, we can conclude that 
$\Sigma_{\tilde x}(\tilde W_{x,v})$ is $\CAT(1)$
in the same way as \cite[Lemma 7.16]{NSY:local}.

This completes the proof of Lemma \ref{lem:decompose-triangle}.   
 \end{proof}

Note that in  $\tilde W_{x,v}$,  the arcs $[x,y]$ in
$\ca S$ joining $x$ and $y$ are replaced by the geodesic 
$[\tilde x,\tilde y]:=\gamma_{\tilde x,\tilde y}$.
We consider the graph structure of the singular locus 
$\tilde{\ca S}_{x,v}$ of $\tilde W_{x,v}$
inherited from $\ca S$
except that $\tilde x, \tilde y\in V(\tilde{\ca S}_{x,v})$
and $(\tilde x,\tilde y)\in E(\tilde{\ca S}_{x,v})$.

Now we do such surgeries finitely many times 
around points of $V_{\rm sing}(\ca S)$ so that the surgery parts
cover $V_{\rm sing}(\ca S)$ as in \cite[Section 7]{NSY:local}.
Let $\tilde E(v_0)$ be the result of those surgeries,
and let $\tilde{\ca S}$ be the singular locus of 
$\tilde E(v_0)$, with graph structure $V(\tilde{\ca S})$,
$E(\tilde{\ca S})$ defined as above.
Note that $V(\tilde{\ca S})$ is discrete,
which implies that $\tilde E(v_0)$ is  
polyhedral.

\begin{lem} \label{lem:tildeE-CAT}
$\tilde E(v_0)$ is a $\CAT(\kappa)$-space.
\end{lem}
\begin{proof}
Using Lemma \ref{lem:decompose-triangle} 
finitely many times, we have 
\begin{itemize}
 \item for every edge $e$ in 
 $E(\tilde{\ca S})$, the condition (A) in 
 Subsection \ref{ssub:polyhed-prelim}
 holds and $e$ has finite turn variation\,$;$
  \item $\Sigma_{\tilde y}(\tilde E(v_0))$ is $\CAT(1)$ for 
    every $\tilde y\in V(\tilde{\ca S})$ . 
\end{itemize}
Consider any triangulation of $\tilde E(v_0)$ extending
$V(\tilde{\ca S})$ and $E(\tilde{\ca S})$ by adding 
geodesic edges if necessary.
Now, we are ready to apply
Theorem \ref{thm:BB-gluing} to this triangulation to conclude   
that  $\tilde E(v_0)$ is $\CAT(\kappa)$.
\end{proof}

\psmall \n
{\bf GH-convergence}.\,
We have to check that $\tilde E(v_0)$ converges to $E(v_0)$ as 
$\e_0\to 0$ with respect to the Gromov-Hausdorff
distance.

\begin{lem} \label{lem:thin-est}
$(1)$\,Let $y, y'$ be arbitrary endpoints of $T$.
For arbitrary $z\in\gamma^X_{x,y}$ and $z' \in\gamma^X_{x,y'}$,
assuming $|x,z|_X\le |x,z'|_X$, we have 
   \[
     ||z,z'|_{X}-|\tilde z, \tilde z'||< \tau(\e_0)|x,z|_X,
   \]
where $\tilde z, \tilde z'\in\pa\tilde K(x,v)$ are the points corresponding to $z,z'$ respectively. 
\par\n
$(2)$\, For arbitrary $y\in T$
and $z\in \pa K(x,v)$, we have 
   \[
     ||y,z|_{X}-|\tilde y, \tilde z||< \tau(\e_0)|x,y|_X.
   \]
where $\tilde y, \tilde z$ are the points of 
$\pa\tilde K(x,v)$ corresponding to $y, z$.
\end{lem}
\begin{proof}
(1)\, 
Let $\hat z\in\gamma^{X}_{x,z'}$ be the point such 
that $|x,\hat z|_{X}=|x,z|_{X}$. 
\eqref{eq:non-meeting-VC} together with 
Lemmas \ref{lem:ab-thin-ruled}, \ref{lem:SE-geom-distance}
implies 
$\tilde\angle^X zx\hat z<\e_0$, and hence
$|z,\hat z|_X\le \tau(\e_0)|x,z|_X$.
By triangle inequality, we have
$$
||z,z'|_{X}-|z',\hat z|_{X}|\le |\hat z,z|_{X}\le\tau(\e_0)|x,z|_X.
$$

Since $\angle \tilde z\tilde x \tilde z'<\tau(\e_0)$, 
similarly we have 
$||\tilde z,\tilde z'|-(|x,z'|_{X}-|x,z|_{X})|\le \tau(\e_0)|x,z|_{X}$.
Combining the last two inequalities, we obtain the 
required inequality.
\par\n

The proof of (2) is the same as that of \cite[Lemma 7.12]{NSY:local}. 
\end{proof}

 \psmall
Now we are ready to show the Gromov-Hausdorff convergence $\tilde E(v_0)\to E(v_0)$ as $\e_0\to 0$.
Define $\varphi:\tilde E(v_0)\to E(v_0)$ as follows.
Let $\varphi$ be identical outside the surgery part.
For every surgery piece  $K(x,v)$ and $\tilde K(x,v)$,
define $\varphi:\tilde K(x,v)\to K(x,v)$ to be any map.
 Since ${\rm diam}(K(x,v))<\tau(\e_0)$,
the image of thus defined map $\varphi$ is $\tau(\e_0)$-dense in 
$E(v_0)$.
We can show that  $\varphi:\tilde E(v_0)\to E(v_0)$ is a 
$\tau(\e_0)$-approximation as in \cite[Section 7]{NSY:local}.
Thus we conclude that $E(v_0)$ is a $\CAT(\kappa)$-space.
This completes the proof of Theorem \ref{thm:union-CAT}.
\qed

\begin{proof}[Proof of Theorem \ref{thm:converse}]
For each $p\in\ca S$, take $r>0$ such that for every $v\in V(\Sigma_p(X))$,  the union $E(v)$  of all
$\CAT(\kappa)$-sector surfaces at $p$ in the direction $v$ 
is properly embedded in $B(p,r)$. For $\e_0>0$,
let $\tilde E(v)$ be the polyhedral $\CAT(\kappa)$-space that is $\tau(\e_0)$-close to $E(v)$ for the
Gromov-Hausdorff distance, as constructed in the proof of 
Theorem \ref{thm:union-CAT}.
Let $W$ denote the union of all $E(v)$\, $(v\in V(\Sigma_p(X))$, and let $\tilde W$ denote the one-point union of those 
$\tilde E(v)$\, $(v\in V(\Sigma_p(X))$ attached at 
$\tilde p$.
Consider any geodesic triangulation of the closure of
$B(p,r)\setminus W$ 
extending those of $\tilde W$.
Let $\tilde B(p,r)$ be the gluing of the closure of
$B(p,r)\setminus W$ and $\tilde W$ along their isometric boundaries.
As in the proof of Lemma \ref{lem:decompose-triangle}, 
we can verify that 
$\Sigma_p(\tilde B(p,r))$ is $\CAT(1)$.
Theorem \ref{thm:BB-gluing} then implies that $\tilde B(p,r)$ is 
$\CAT(\kappa)$. Thus, letting $\e_0\to 0$, we conclude that 
$B(p,r)$ is $\CAT(\kappa)$.
\end{proof}

\pmed\n


\end{document}